\documentclass[a4paper,11pt]{amsart}
\usepackage{amsfonts}
\usepackage{amssymb}
\usepackage[utf8]{inputenc}
\usepackage{amsmath}
\usepackage{pdflscape}
\usepackage{graphicx}
\usepackage[colorlinks=true, linkcolor=red, citecolor=blue]{hyperref}
\usepackage[]{epsfig}
\usepackage[]{pstricks}
\usepackage{tikz}
\newtheorem{theorem}{Theorem}[section]

\newtheorem{lemma}[theorem]{Lemma}

\newtheorem{corollary}[theorem]{Corollary}

\newtheorem{proposition}[theorem]{Proposition}
\newtheorem{problem}[theorem]{Problem}

\theoremstyle{definition}
\newtheorem{definition}[theorem]{Definition}
\newtheorem{remark}[theorem]{Remark}

\theoremstyle{remark}
\numberwithin{equation}{section}
\makeatletter
\@namedef{subjclassname@2010}{\textup{2020} Mathematics Subject Classification}
\makeatother

\begin{document}

	\pagenumbering{arabic}	
\title[Well-posedness and control of dispersive systems]{Global control aspects for long waves in nonlinear dispersive media}
\author[Capistrano-Filho]{Roberto de A. Capistrano-Filho}
\address{Departamento de Matem\'atica,  Universidade Federal de Pernambuco (UFPE), 50740-545, Recife (PE), Brazil.}
\email{roberto.capistranofilho@ufpe.br}
\author[Gomes]{Andressa Gomes}
\address{Universidade Federal do Delta do Parna\'iba, Campus Ministro Reis Velloso,
Coordenação de Matem\'atica, 64202-020, Parna\'iba (PI), Brazil.}
\email{gomes.andressa.mat@outlook.com}
%\thanks{\textbf{Funding:}
%\date{Version 2021-12}

\begin{abstract}A class of models of long waves in dispersive media with coupled quadratic nonlinearities on a periodic domain $\mathbb{T}$ are studied. We used two distributed controls, supported in $\omega\subset\mathbb{T}$ and assumed to be generated by a linear feedback law conserving the \textit{ ``mass" (or ``volume")}, to prove global control results. The first result, using spectral analysis, guarantees that the system in consideration is locally controllable in $H^s(\mathbb{T})$, for $s\geq0$. After that, by certain properties of Bourgain spaces we show a property of global exponential stability. This property together with the local exact controllability ensures for the first time in the literature that long waves in nonlinear dispersive media are globally exactly controllable in large time.  Precisely,  our analysis relies strongly on the \textit{bilinear estimates} using the Fourier restriction spaces in two different dispersions that will guarantee a global control result for coupled systems of the Korteweg–de Vries type. This result, of independent interest in the area of control of coupled dispersive systems, provides a necessary first step for the study of global control properties to the coupled dispersive systems in periodic domains. 
\end{abstract}
\maketitle

%\tableofcontents

\section{Introduction}Nonlinear dispersive wave equations arise in a number of important application areas. Because of this, and because their mathematical properties are interesting and subtle, they have seen enormous development since the 1960s when they first came to the fore\footnote{See \cite{Miura} for a sketch of the early history of the subject.}. The theory for a single nonlinear dispersive wave equation is well developed by now, though there are still interesting open issues, however  the theory for coupled systems of such equations is much less developed, though they, too, arise as models of a range of physical phenomena. 

Considered here is a class of such systems, namely coupled Korteweg–de Vries (KdV) equations. The systems we have in mind take the form
\begin{equation}\label{system}
\begin{cases}
\partial_{t} u + \partial_{x}^{3} u + \partial_{x}P(u,v) =0,&  \ x \in \mathbb{T},\ t \in \mathbb{R},\\
\partial_{t} v +  \alpha \partial_{x}^{3} v + \partial_{x}Q(u,v) =0,&  \ x \in \mathbb{T},\ t \in \mathbb{R},
\end{cases}
\end{equation}
which comprise two linear Korteweg–de Vries equations coupled through their nonlinearity. Here, $u=u(x,t)$ and $v=v(x,t)$ are real-valued functions of variables $(x,t)\in\mathbb{T}\times\mathbb{R}$ and the nonlinearities $P$ and $Q$ are taken to be homogeneous quadratic polynomials.
%\begin{equation}\label{nonlinearity}
%\begin{cases}
%P(u,v) = A u^{2} + B u v + C v^{2}, \\
%Q(u,v) = D v^{2} + E v u +F u^{2},
%\end{cases}
%\end{equation}
%where $A,B,C,D,E$ and $F$ are constants  will be employed when convenient. 

As far as we know, there are no studies of the \textit{global control properties} of this kind of coupled systems in a periodic domain. Thus, in this article, the goal is to fill this gap focusing on the global exact controllability and global asymptotic behavior to the solutions of the coupled system of KdV equations \eqref{system} when we add two control inputs in each equation and considering  initial conditions $(u(x,0),v(x,0))=(u_0(x),v_0(x))$ belonging in $H^s(\mathbb{T})\times H^s(\mathbb{T}),$ for any $s\geq0$.

\subsection{Models in the literature} Such systems arise as models for wave propagation in physical systems where both nonlinear and dispersive effects are important. Moreover, their close relatives arise as models for waves in a number of situations. Before to present details about the problem that we will study let us start to list a few specializations of systems \eqref{system} that appeared in the literature. 

\subsubsection{The coupled KdV system}
The classical Boussinesq systems were first derived by Boussinesq in \cite{Boussinesq1}, to describe the two-way propagation of small amplitude, long wavelength gravity waves on the surface of water in a canal. These systems and their higher-order generalizations also arise when modeling the propagation of long-crested waves on large lakes or on the ocean and in other contexts. Recently Bona \textit{et al.}, in \cite{bona-chen-saut}, derived a four-parameter family of Boussinesq systems to describe the motion of small amplitude long waves on the surface of an ideal fluid under the gravity force and in situations where the motion is sensibly two dimensional. More precisely, they studied a family of systems of the form
\begin{equation}
\left\{
\begin{array}
[c]{l}%
\eta_{t}+w_{x}+(   \eta w)  _{x}+aw_{xxx}-b\eta_{xxt}=0\text{,}\\
w_{t}+\eta_{x}+ww_{x}+c\eta_{xxx}-dw_{xxt}=0\text{.}%
\end{array}
\right.  \label{int_29e_crp_1}%
\end{equation}
In (\ref{int_29e_crp_1}), $\eta$ is the elevation from the equilibrium position, and $w=w_{\theta}$ is the horizontal velocity in the flow at height $\theta h$, where $h$ is the undisturbed depth of the liquid. The parameters $a$, $b$, $c$, $d$, that one might choose in a given modeling situation, are required to fulfill the relations
\begin{equation*}
a+b=\frac{1}{2}\left(   \theta^{2}-\frac{1}{3}\right)  \text{, \ \ \ }%
c+d=\frac{1}{2}(   1-\theta^{2})  \geq0 \text{, \ \ \ }\theta
\in\left[  0,1\right], 
\end{equation*}
where $\theta\in\left[  0,1\right]  $ specifies which horizontal velocity the variable $w$ represents (cf. \cite{bona-chen-saut}). Consequently,
\[
a+b+c+d=\frac{1}{3}.
\]
As it has been proved in \cite{bona-chen-saut}, the initial value problem for the linear system associated with \eqref{int_29e_crp_1} is well-posed on $\mathbb R$ if either $C_1$ or $C_2$ is satisfied, where
\begin{eqnarray*}
(C_1)&& b,d\ge 0,\ a\le 0,\ c\le 0;\\
(C_2)&& b,d\ge 0, \ a=c>0. 
\end{eqnarray*}
When $b=d=0$ and $(C_2)$ is satisfied, then necessarily $a=c=1/6$. Nevertheless, the scaling $x\to x/\sqrt{6}$, $t\to t/\sqrt{6}$ gives an system equivalent to \eqref{int_29e_crp_1} for which $a=c=1$, namely
\begin{equation}
\label{kdv-kdv}
\begin{cases}
\eta_t + w_x+w_{xxx}+  (\eta w)_x= 0, \\
w_t +\eta_x +\eta_{xxx} +ww_x=0,
\end{cases}
\end{equation}
which is the so-called \textit{Boussinesq system of KdV-KdV type}.

\subsubsection{Gear-Grimshaw system} In \cite{geargrimshaw1984} a complex system of equations was derived by Gear and Grimshaw as a model to describe the strong interaction of two-dimensional, weakly nonlinear, long, internal gravity waves propagating on neighboring pycnoclines in a stratified fluid, where the two waves correspond to different modes. It has the structure of a pair of KdV equations with both linear and nonlinear coupling terms and has been the object of intensive research in recent years. The system can be read as follows
\begin{equation}
\label{gg}
\begin{cases}
u_t + uu_x+u_{xxx} + a v_{xxx} + a_1vv_x+a_2 (uv)_x =0,\\
c v_t +rv_x +vv_x+abu_{xxx} +v_{xxx}+a_2buu_x+a_1b(uv)_x  =0,
\end{cases}
\end{equation}
where $a_1, a_2, a, b, c, r\in \mathbb{R}$ are physical constants and we may assume that 
$$1-a^2 b > 0 \quad \text{and} \quad  b, c > 0.$$

\subsubsection{Majda-Biello system} The following coupled system 
\begin{equation*}
\label{mb}
\begin{cases}
u_{t}+u_{xxz} =-v v_{x}, \\
v_{t}+\alpha v_{xxx} =-(u v)_{x}, \\
\end{cases}
\end{equation*}
when $\alpha \in (0,1)$\footnote{The parameter $\alpha>0$ depends upon the Rossby wave in question and it typically has a value near $1$.}, was proposed by Majda and Biello in \cite{MajdaBiello}  as a reduced asymptotic model to study the nonlinear resonant interactions of long wavelength equatorial Rossby waves and barotropic Rossby waves.

\subsubsection{Hirota-Satsuma system} In the eighties, Hirota and Satsuma introduced in \cite{HirotaSatsuma} the set of two coupled KdV equations, namely
\begin{equation*}
\label{HS}
\left\{\begin{array}{l}
u_{t}+au_{x x x}=6au u_{x}+b v v_{x}, \\
v_{t}+ v_{x x x}=-3 u v_{x},
\end{array}\right.
\end{equation*}
with $a\neq0$,  where $a,b\in\mathbb{R}$ are constants that appear in the model deduction. This model describes the interaction of two long waves with different dispersion relations.

We caution that this is only a small sample of the extant equations with the similar structure to the system \eqref{system}. For an extensive review of the physical meanings of these equations, as well as local and global well-posedness results, the authors suggest the following nice two references \cite{bonacohenwang,Zhang}.

\subsection{Setting of the problem} Since any solution $(u,v)$ of system \eqref{system} has its components with invariant mean value, we can introduce the numbers $[u]:=\beta$ and $[v] := \gamma$. Setting $\widetilde{u} = u - \beta$ and $\widetilde{v} = v- \gamma$, we obtain $[\widetilde{u}]  = [\widetilde{v}]=0$ and $(\widetilde{u}, \widetilde{v})$ solves
\begin{equation}\label{system1}
\begin{cases}
\partial_{t}\widetilde{u} + \partial_{x}^{3}\widetilde{u} +(2\beta A + \gamma B) \partial_{x}\widetilde{u} + (\beta B + \gamma C)\partial_{x} \widetilde{v} + \partial_{x}P(\widetilde{u}, \widetilde{v})= 0,&  \ x \in \mathbb{T},\ t \in \mathbb{R},\\
\partial_{t}\widetilde{v} + \alpha \partial_{x}^{3}\widetilde{v} +(\beta B + \gamma C  ) \partial_{x}\widetilde{u} + (2\gamma D + \beta C)\partial_{x} \widetilde{v} + \partial_{x}Q(\widetilde{u}, \widetilde{v})= 0,&  \ x \in \mathbb{T},\ t \in \mathbb{R}.
\end{cases}
\end{equation}
Throughout the paper, we will denote $\mu := 2 \beta A + \gamma B$, $\eta := \beta B + \gamma C$, $\zeta := 2 \gamma D + \beta C $ which are real constants. Thus, as mentioned before, this article presents for the first time the global control results for a class of models of long waves with coupled quadratic nonlinearities. Precisely, thanks to \eqref{system1} we will study the following system 
\begin{equation}\label{systema}
\begin{cases}
\partial_{t}u + \partial_{x}^{3}u + \mu \partial_{x} u + \eta \partial_{x} v + \partial_{x}P(u,v)= p(x,t),&  \ x \in \mathbb{T},\ t \in \mathbb{R},\\
\partial_{t}v + \alpha \partial_{x}^{3} v + \zeta \partial_{x} v +  \eta \partial_{x}u  + \partial_{x}Q(u, v)= q(x,t),&  \ x \in \mathbb{T},\ t \in \mathbb{R},\\
(u(x, 0), v(x,0)) = (u_{0}(x), v_{0}(x)), &  \ x \in \mathbb{T},
\end{cases}
\end{equation}
with quadratic nonlinearities
\begin{equation}\label{nonlinearitya}
\begin{cases}
P(u,v) = A u^{2} + B u v + \frac{C}{2} v^{2}, \\
Q(u,v) = D v^{2} + C v u +\frac{B}{2} u^{2},
\end{cases}
\end{equation}
where $\alpha, A,B,C$ and $D$ real constants, from a control point of view with forcing terms $p=p(x,t)$ and $q=q(x,t)$ added to the equation as two control inputs on the periodic domain. Therefore,  the following classical issues related with control theory are considered in this work.

\begin{problem}[Exact controllability]\label{p1}
Given an initial state $(u_0,v_0)$ and a terminal state $(u_1,v_1)$ in a certain space, can one find two appropriate control inputs $p$ and $q$ so that the equation \eqref{systema} admits a solution $(u,v)$ which satisfies $(u(\cdot,0),v(\cdot,0)= (u_0,v_0)$ and $(u(\cdot,T),v(\cdot,T))=(u_1,v_1)$ ?
\end{problem}
\begin{problem}[Stabilizability]\label{p2}
Can one find some (linear) feedback controls $p=K_1(u,v)$ and $q=K_2(u,v)$ such that the resulting closed-loop system \eqref{systema} is stabilized, i.e., its solution $(u,v)$ tends to zero in an appropriate space as $t\to\infty$?
\end{problem}

Note that system \eqref{systema} has the \textit{mass (or volume)} and the \textit{energy} conserved, which are 
\begin{equation*}
M_{1} (u,v) = \int_{\mathbb{T}} u(x,t) dx,\quad  \quad M_{2}(u,v) = \int_{\mathbb{T}} v(x,t) dx,\quad 
E(u,v) = \frac{1}{2}\int_{\mathbb{T}} (u^{2}(x,t) + v^{2}(x,t))dx,
\end{equation*}
respectively. In order to keep the mass $M_{1}$ and $M_{2}$ conserved, the two control inputs $p(x,t)$ and $q(x,t)$ will are chosen to be of the form $Gf(x,t)$ and $Gh(x,t)$, respectively, where this operator is defined by
\begin{equation}\label{defG}
(G\ell)(x,t) := g(x) \left( \ell(x,t) - \int_{\mathbb{T}} g(y) \ell(y,t) dy \right),
\end{equation}
where  $f$ and $h$ are considered the new control inputs, and $g(x)$ is a given nonnegative smooth function such that $\{ g > 0\} = \omega\subset\mathbb{T}$ and
\begin{equation*}
2 \pi [g] = \int_{\mathbb{T}} g(x) dx = 1.
\end{equation*}
Due to such a choice of $g$, it is easy to see that for any solution $(u,v)$ of \eqref{systema} with $p=Gf$ and $q=Gh$ we have
$$
\frac{d}{dt} M_{1}(u,v) = \int_{\mathbb{T}} Gf(x,t) dx =0 \quad \text{and} \quad
\frac{d}{dt} M_{2}(u,v) = \int_{\mathbb{T}}
Gh(x,t) dx = 0,
$$
that is, the \textit{mass} of the system is indeed conserved. 

To stabilize system \eqref{systema} we want to employ two feedback control laws that help make the energy of the system decreasing, that is,  $E'(u,v)\leq0.$ We will see that this is possible, and so makes sense  to show global answers to the Problems \ref{p1} and \ref{p2}, mentioned before. Before it, let us give a state of the arts of control theory for KdV type systems.

\subsection{State of the art} The study of the controllability and stabilization to the KdV equation started with the works
of Russell and Zhang  \cite{RusselZhang1993,Russel96}  for the system 
\begin{equation}
u_{t}+uu_{x}+u_{xxx}=f\text{, }
\label{I5}%
\end{equation}
with periodic boundary conditions and an internal control $f$. Since then, both controllability and stabilization problems have been intensively studied \cite{Capistrano,CaPaRo,cerpa,Rosier,Zhang2}.

Equation \eqref{I5} is known to possess an infinite set of conserved integral
quantities, of which the first three are
$$I_1(t)=\int_{\mathbb{T}}u(x,t)dx, \quad \quad I_2(t)=\int_{\mathbb{T}}u^2(x,t)dx \quad \text{and} \quad
I_3(t)=\int_{\mathbb{T}}\left(u_x^2(x,t)-\frac{1}{3}u^3(x,t)\right)dx.$$
From the historical origins of the KdV equation involving the behavior of water waves in a shallow channel \cite{Boussinesq,CaZhaSun,Miura,Korteweg}, it is natural to think of $I_{1}$ and $I_{2}$ as expressing
conservation of \textit{volume (or mass)} and \textit{energy}, respectively. The Cauchy problem for equation (\ref{I5}) has been intensively studied for many years (see \cite{Bourgain,Kato,Kenig} and the references therein). 

The first work of Russell and Zhang \cite{RusselZhang1993}  is purely linear. In fact, they had to wait for several years to extend their results to the nonlinear systems \cite{Russel96} until Bourgain \cite{Bourgain} discovered a subtle smoothing property of solutions of the KdV equation posed on a periodic domain, thanks to which Bourgain was able to show that the Cauchy problem \eqref{I5} is well-posed in the space $H^s(\mathbb{T})$, for any $s \geq 0$. This novelty discovered the smoothing property of the KdV equation has played a crucial role in the proofs of the results in \cite{Russel96}. 

\vspace{0.2cm}

\noindent\textbf{Question 1}: \textit{Can one still guide the system by choosing
appropriate control input $h$ from a given initial state $u_0$ to a given terminal state $u
_1$ when $u_0$ or $u_1$ have large amplitude?}

\vspace{0.2cm}

\noindent\textbf{Question 2}: \textit{Do the large amplitude solutions of the
closed-loop system (\ref{I5}) decay exponentially as $t\rightarrow\infty$}?

Laurent \textit{et al.} \cite{laurent2010} gave the positive answers to these questions. These answers are established with the aid of certain properties of propagation of compactness and regularity in Bourgain spaces for the solutions of the associated linear system of \eqref{I5}.  

We have to mention that there are other works in the literature that deal with the models having similar structure as the system \eqref{system} in periodic domains. Micu \textit{et al.}\cite{micu2} gave a rather complete picture of the control properties of \eqref{int_29e_crp_1} on a periodic domain with a locally supported forcing term. According to the values of the four parameters $a$, $b$, $c$, and $d$, the linearized system may be controllable in any positive time, or only in large time, or it may not be controllable at all. 

Recently, Capistrano-Filho \textit {et al.}   \cite{capistranokp2020} considered the problem of controlling pointwise, by means of a time dependent Dirac measure supported by a given point, a coupled system of two Korteweg--de Vries equations \eqref{gg} on the unit circle. More precisely, by means of spectral analysis and Fourier expansion they proved,  under general assumptions on the physical parameters of the system, a pointwise observability inequality which leads to the pointwise controllability by using two control functions. In addition, with a uniqueness property proved for the linearized system without control, they are able to show pointwise controllability when only one control function acts internally.

There are two important points to say about the results shown in \cite{micu2} and \cite{capistranokp2020}. The first one is that the results presented in \cite{micu2} are purely local (controllability and stability), the authors did not use propagation of singularities, provided by the Bourgain spaces, to obtain more general results. In fact, one of the problems left in \cite{micu2} is to prove global results for systems like \eqref{systema}. With respect to the results proved in  \cite{capistranokp2020}, the results are purely linear, and extensions to the non-linear system are only possible in regular spaces.

\subsection{Notation and Main results} Let us introduce some notation and present the main results of the manuscript.   

We denote $\mathcal{D}(\mathbb{T})$ the space of periodic distributions whose dual space is $C^{\infty}(\mathbb{T})$. The Fourier series of periodic distributions is given by
\begin{equation*}
 \mathcal{F}f(k) = \widehat{f}(k)= \frac{1}{2\pi}\int_{0}^{2\pi}f(x)e^{-ikx} dx, \ \ k \in \mathbb{Z}
\end{equation*}
and the inverse Fourier series by
\begin{equation*}
\mathcal{F}^{-1}f(x) = \sum_{k \in \mathbb{Z}} e^{ikx}f(x).
\end{equation*}
For $s>0$, we use the operator $D^{s} = (- \Delta)^{\frac{1}{2}}$ given on the Fourier side as
\begin{equation*}
\widehat{D^{s}f}(k) = |k|^{s} \widehat{f}(\xi).
\end{equation*}
Similarly, we have the operators $J^{s}$ given on the Fourier side as
\begin{equation*}
\widehat{J^{s}f}(k) = \left< k \right>^{s} \widehat{f}(k)
\end{equation*}
where $\left< k \right>:= (1 + |k|) \sim (1 + |k|^{2})^{\frac{1}{2}}$. 
Here we define the $H^{s}(\mathbb{R})$ Sobolev spaces, for $s \in \mathbb{R}$ 
\begin{equation*}
H^{s}(\mathbb{T}) = \{ f \in \mathcal{D}(\mathbb{T}) \ : \ \|f\|_{s} := \|J^{s}f\| < \infty \}
\end{equation*}
%The homogeneous Sobolev spaces $\dot{H}^{s}(\mathbb{T})$ are defined analogously with $D^{s}$ instead of $J^{s}$.
For the Cartesian spaces $H^{s}(\mathbb{T}) \times H^{s}(\mathbb{T})$ we define $\|(u,v)\|_{s}:= \|(u,v)\|_{H^{s}(\mathbb{T}) \times H^{s}(\mathbb{T})} = \|u\|_{s} + \|v\|_{s}$. Throughout this paper we will denote the norm $\|(\cdot, \cdot )\|_{L^{2}(\mathbb{T}) \times L^{2}(\mathbb{T})}$ simply by $\|(\cdot, \cdot )\|$. Let $X$ be one of the previously defined spaces, we will denote $X_{0}$ the function space belong in X with media-value null, i.e., $X_{0} :=\{ u \in X \  :  \ [u]=0 \}$.

The aim of this manuscript is to address the control and stabilization (global) issues. Precisely, we want to give answers for both questions (see Problems \ref{p1} and \ref{p2}) presented at the beginning of this introduction. As first result we will to analyse the exact controllability  for the following linear system
\begin{equation}\label{openloopsystem_int}
\begin{cases}
\partial_{t}u + \partial_{x}^{3}u + \mu \partial_{x} u + \eta \partial_{x} v = Gf,&  \ x \in \mathbb{T},\ t \in \mathbb{R},\\
\partial_{t}v + \alpha \partial_{x}^{3}v + \zeta \partial_{x} v + \eta \partial_{x} u = Gh,&  \ x \in \mathbb{T},\ t \in \mathbb{R},\\
(u(x, 0), v(x,0)) = (u_{0}(x), v_{0}(x)), &  \ x \in \mathbb{T}.
\end{cases}
\end{equation}
Here, $f$ and $g$ are defined as two control inputs and the operator $G$ is given by \eqref{defG}. We have established the following.

\begin{theorem}\label{main_A}
Let $T > 0$ and $s \geq 0$ be given. Then for any $(u_{0}, v_{0})$, $(u_{1}, v_{1}) \in H_{0}^{s}(\mathbb{T}) \times H_{0}^{s}(\mathbb{T})$, there exists a pair of control functions $(f,h) \in L_{0}^{2}(\mathbb{T}) \times L_{0}^{2}(\mathbb{T})$, such that system  \eqref{openloopsystem_int} has a solution in the class $$(u,v) \in C([0,T];H_{0}^{s}(\mathbb{T})) \times C([0,T];H_{0}^{s}(\mathbb{T}))$$ satisfying
\begin{equation*}
(u(x,0), v(x,0)) = (u_{0}, v_{0})\quad \text{and} \quad 
(u(x,T), v(x,T)) = (u_{1}, v_{1}).
\end{equation*}
\end{theorem}

Taking advantage of the results obtained by Bourgain \cite{Bourgain}, we are able to extend the previous local result to the nonlinear system, which is represented by,
\begin{equation}\label{systemcontrollocal_ns_int}
\begin{cases}
\partial_{t}u + \partial_{x}^{3}u + \mu \partial_{x} u + \eta \partial_{x} v + \partial_{x}P(u,v) = Gf,&  \ x \in \mathbb{T},\ t \in \mathbb{R},\\
\partial_{t}v + \alpha \partial_{x}^{3}v + \zeta \partial_{x} v + \eta \partial_{x} u +   \partial_{x}Q(u,v) = Gh,&  \ x \in \mathbb{T},\ t \in \mathbb{R},\\
(u(x, 0), v(x,0)) = (u_{0}(x), v_{0}(x)), &  \ x \in \mathbb{T},
\end{cases}
\end{equation}
where $P(u,v)$, $Q(u,v)$ are defined by \eqref{nonlinearitya}, $G$ is represented by \eqref{defG}, with $f$ and $g$ are control inputs. Thus, our second result deals with the asymptotic behavior of the solutions of \eqref{systema}. In order to stabilize system \eqref{systemcontrollocal_ns_int}, choose the  two feedback controls $$f=-G^*L_{1, \mu , \lambda}^{-1}u \quad \text{ and } \quad h=-G^*L_{\alpha , \zeta, \lambda}^{-1}v,$$
in \eqref{systemcontrollocal_ns_int}, to transform it in a resulting closed-loop system reads as follows
\begin{equation}\label{PVIsystem_int}
\begin{cases}
\partial_{t} u + \partial_{x}^{3} u + \mu \partial_{x} u + \eta \partial_{x} v +   \partial_{x}P(u,v) = -K_{1, \mu , \lambda}u,&  \ x \in \mathbb{T},\ t \in \mathbb{R},\\
\partial_{t} v + \alpha\partial_{x}^{3} v + \zeta \partial_{x} u + \eta \partial_{x} u +  \partial_{x} Q(u,v) = -K_{\alpha , \zeta , \lambda}v,&  \ x \in \mathbb{T},\ t \in \mathbb{R},\\
(u(x,0), v(x,0)) = (u_{0}(x), v_{0}(x)), &  \ x \in \mathbb{T},
\end{cases}
\end{equation}
with the damping mechanism defined by $$K_{\beta , \gamma, \lambda}:=GG^*L_{\beta , \gamma, \lambda}^{-1}.$$  Here, $L_{\beta, \gamma, \lambda}$ is a bounded linear operator from $H^{s}(\mathbb{T})$ to $H^{s}(\mathbb{T})$, $s\geq0$, for details see Section \ref{Sec3}. So, as for Problem \ref{p2}, we have the following affirmative answer.
\begin{theorem}\label{main}
Let $s \geq 0$ and $\gamma \in \mathbb{R}$ be given. There exists a constant $\kappa>0$ such that for any $u_{0},v_0 \in H_{0}^{s}(\mathrm{T})$ the corresponding solution $(u,v)$ of the system \eqref{PVIsystem_int} satisfies
$$
\left\|(u,v)\right\|_{s} \leq a_{s, \gamma}\left(\left\|(u_{0},v_{0})\right\|_{0}\right) e^{-\kappa t}\left\|(u_0, v_0)\right\|_{s},
$$
for all  $t \geq 0$. Here $a_{s, \gamma}: \mathbb{R}^{+} \rightarrow \mathbb{R}^{+}$ is a nondecreasing continuous function depending on $s$ and $\gamma$.
\end{theorem}

To finalize, observe that Theorem \ref{main_A} is purely linear. Thanks to Theorem \ref{main} we guarantee a global controllability for long waves, thus responding to Problem \ref{p1}. The result can be read as follows.

\begin{theorem}\label{main1}
Let $ s\geq0 $ and $ R_0>0 $ be given. There exists a time $T > 0$ such
that if $(u_0,v_0)$, $(u_1,v_1)\in  H_{0}^s(\mathbb{T}) \times H_{0}^s(\mathbb{T}) $ are such that
$$
\|(u_0,v_0)\|_s\leq R_0,\quad \|(u_1,v_1)\|_s\leq R_0,
$$
then one can find two controls input $f,g\in L^2(0,T;H_{0}^s(\mathbb{T})) $ such that system \eqref{systemcontrollocal_ns_int} admits a solution $$ (u,v) \in C([0,T];H^s_0(\mathbb{T}))\times C([0,T];H^s_0(\mathbb{T})) $$ satisfying
\begin{equation*}
(u(x,0), v(x,0))= (u_{0}(x), v_{0}(x)) \ \ \mbox{and} \ \ (u(x, T), v(x,T)) = (u_{1}(x), v_{1}(x)).
\end{equation*}
\end{theorem}

It is important to point out that Theorems \ref{main} and \ref{main1} are valid for the case when we consider in the systems above mentioned $\alpha<0$ and $|\mu|+|\zeta|$ is small enough. This restriction is necessary due to the fact that we need estimates for non-linear terms (see Lemmas \ref{bilinearestlemma} and  \ref{contralbourgainestimate2} in the Section \ref{Sec3}) which needs to be verified when $|\mu|+|\zeta|<<1$, $\alpha<\frac{1}{4}$ and $\frac{1}{\alpha}<\frac{1}{4}$, simultaneously. However, if $B=C=0$, i.e. $\eta=0$ (see system \eqref{systema}), we have two KdV-type systems coupled only in the nonlinear terms. Thus, $\eta=0$ ensures that all the results presented in this manuscript remain valid without any restriction in the constants $\alpha,\mu$ and  $\zeta$.  

\subsection{Structure of the article} Section \ref{Sec2} is devoted to show the  spectral analysis necessary to prove the exact controllability result for the linear system associated to  \eqref{systema}. Next,  Section \ref{Sec3}, we present the Bourgain spaces and its property. Precisely, thanks to linear and nonlinear estimates we are able to prove the global well-posedness results for the system \eqref{PVIsystem_int}.  In the Section \ref{Sec4}, the reader will find the proofs of the main theorems of the article.  Section \ref{Sec6} is devoted to presenting the conclusion of the work and some open issues. Finally, on Appendix \ref{Apendice}, we collect results associated with the system  \eqref{systema}, which were used throughout the paper.

\section{Spectral problem}\label{Sec2}
In this section we study the spectral properties of the linear system associated to \eqref{openloopsystem_int}. Precisely, using Ingham's type theorem, we prove that the exact controllability for \eqref{systema}-\eqref{nonlinearitya} holds.  Consider the following operator
\begin{equation}\label{operatorL}
L= \left( \begin{array}{cc} 
   -\partial^{3} - \mu \partial  &  -\eta  \partial  \\
   -\eta \partial &  -\alpha \partial^{3} - \zeta \partial
\end{array}
\right)
\end{equation}
with domain $\mathcal{D}(L) = H^{3}(\mathbb{T}) \times H^{3}(\mathbb{T})$. This operator has the following properties.

\begin{proposition}\label{propertiesL}
Consider the operator $L$ defined as in \eqref{operatorL}. If $\alpha <0$ and $\zeta - \mu >0$ then $L$ generates a strongly continuous group $S(t)$ in  $L^{2}(\mathbb{T}) \times L^{2}(\mathbb{T})$. Moreover, the eigenfunctions are defined by $e^{-ikx}Z_{k}^{\pm}$, with $k \in \mathbb{Z}$ and  form an orthogonal basis in $L^{2}(\mathbb{T}) \times L^{2}(\mathbb{T})$ satisfying 
\begin{equation*}
\displaystyle Z_{k}^{\pm} \longrightarrow Z^{\pm}, \quad \text{ as } k\to \pm \infty,
\end{equation*}
where $Z^{+}:= (0,0)$ and $Z^{-}:= (0, 2(1- \alpha))$.
\end{proposition}

\begin{proof}
       A simple calculation shows that $L^{*} = -L$ and $\left<Lu, u\right> = - \left< u, Lu\right> =0.$ Thus $L$ and $L^{*}$ are dissipative. Since $\mathcal{D}(L)$ is dense on $L^{2}(\mathbb{T}) \times L^{2}(\mathbb{T})$ follows from \cite[Corollary 4.4]{pazy1983semigroups} that $L$ is an infinitesimal generator of a strongly continuous group of contractions on $L^{2}(\mathbb{T}) \times L^{2}(\mathbb{T})$. 
       
We claim that, for each fixed $k\in\mathbb{Z}$, $e^{-ikx}(\sigma_{k}, \tau_{k})$ is an eigenvector of $L$ with eigenvalue $i\omega_{k}$ if and only if
\begin{equation}\label{systemofeigenfunctions}
\begin{cases}
(k^{3} - \mu k - \omega_{k})\sigma_{k} - \eta k\tau_{k}= 0,\\
- \eta k \sigma_{k} + (\alpha k^{3} - \zeta k  - \omega_{k})\tau_{k}  =0.
\end{cases}
\end{equation}
That is, there exist non-trivial solutions if and only if
%\begin{equation*}
%\left|
%\begin{array}{cc}
%  k^{3} - \mu k - \omega_{k}  &  - \eta k\\
%    - \eta k & \alpha k^{3} - \zeta k  - \omega_{k}
%\end{array}
%\right| =0,
%\end{equation*}
%or equivalently, 
\begin{equation*}
\begin{split}
 \omega_{k}^{2} + \omega_{k} (\zeta + \mu - (1 + \alpha) k^{2}) k + \alpha k^{6} - k^{4}(\zeta + \alpha \mu)  - k^{2} (\eta^{2} - \mu \zeta)=0.
\end{split}
\end{equation*}
Hence, we have two possible exponents, given by the formula
\begin{equation*}
\begin{split}
2\omega_{k}^{\pm} = &  \ k ((1+ \alpha)k^{2} - (\mu + \zeta)  ) \pm \sqrt{k^{2}(\mu + \zeta -(1+ \alpha) k^{2})^{2}-4k^{2}(\alpha k^{4} - k^{2}(\alpha \mu + \zeta) - (\eta^{2} - \mu \zeta)) }\\
%= &  k ((1+ \alpha) k^{2} - (\mu + \zeta)  )  \\
%& \pm k\sqrt{(1 + \alpha)^{2} k^{4} + (\mu + \zeta)^{2} - 2k^{2}(\mu + \zeta)(1 + \alpha) - 4\alpha k^{4} + 4k^{2}(\alpha\mu + \zeta) +4\eta^{2} - 4\mu \zeta)}\\
%= & k ((1+ \alpha) k^{2} - (\mu + \zeta)  ) \pm  k \sqrt{   (\zeta -\mu)^{2} + (1 - \alpha)^{2}k^{4} +2 k^{2}(\zeta - \mu)(1 - \alpha) + 4 \eta^{2} }\\
%= &\  k ((1+ \alpha) k^{2} - (\mu + \zeta)  ) \pm  k \sqrt{[k^{2}(1-\alpha) + (\zeta - \mu)]	^{2} + 4\eta^{2}}\\
= &\ k^{3} (1+ \alpha - (\mu + \zeta)k^{-2}  ) \pm  k^{3} \sqrt{[(1-\alpha) + k^{-2}(\zeta - \mu)]^{2} +4 k^{-4}\eta^{2}},
\end{split} 
\end{equation*}
that is, 
\begin{equation}\label{eigenvalues}
2\omega_{k}^{\pm} = k^{3}\left[(1+ \alpha) - (\zeta + \mu)k^{-2} \pm \sqrt{[(1-\alpha) + k^{-2}(\zeta - \mu)]^{2} + 4k^{-4}\eta^{2}) } \right] .
\end{equation}

If $k \neq 0$, with $\eta \neq 0$, then $\omega_{k}^{-} \neq \omega_{k}^{+}$ and two corresponding non-zero eigenvectors are given by the formula
\begin{equation}\label{basisL2plus}
\begin{split}
Z_{k}^{\pm} &= (\sigma_{k}, \tau_{k}) 
%&=  \left(2\eta k^{-2}, (1 - \alpha) + (\zeta - \mu)k^{-2} \mp \sqrt{[ (1 - \alpha) + (\zeta - \mu)k^{-2}  ]^{2} + 4 k^{-4}\eta^{2}}\right)\\
=\  2k^{-3} \left( \eta k , k^{3} - \mu k - \omega_{k}^{\pm} \right).
\end{split}
\end{equation}
If $k = 0$, then both eigenvalues are equal to zero and two linearly independent eigenvectors are given for example by
\begin{equation}\label{basisL2zeroplus}
Z_{0}^{\pm}= (\sigma_{0} , \tau_{0}) = \left(2 \eta, (1 - \alpha) \mp \sqrt{(1- \alpha )^{2} + 4 \eta^{2}}\right)
\end{equation}
A direct calculation show that
$
Z_{k}^{+} \cdot Z_{k}^{-} = 0,
$
for all $k \in \mathbb{Z}$ and $Z_{k}^{\pm} \longrightarrow Z^{\pm}$ as $k \longrightarrow \pm \infty$.
Thus, 
$
(\phi_{k}^{\pm}, \psi_{k}^{\pm})= e^{-ikx}\cdot Z_{k}^{\pm},
$
where $Z_{k}^{\pm}= (\sigma_{k}^{\pm}, \tau_{k}^{\pm})$ is defined as in \eqref{basisL2plus}-\eqref{basisL2zeroplus}, form an orthogonal basis in $L^{2}(\mathbb{T}) \times L^{2}(\mathbb{T})$ with the eigenvalues given by \eqref{eigenvalues}, showing the proposition.
\end{proof}

\begin{lemma}\label{nulldensity}
Let $\omega_{k}^{\pm}$ be as in \eqref{eigenvalues}. We have
\begin{equation*}
\lim_{k \rightarrow \pm \infty} (\omega_{k+1}^{+} - \omega_{k}^{+}) = + \infty \ \ \mbox{and} \ \  \lim_{k \rightarrow \pm \infty} (\omega_{k+1}^{-} - \omega_{k}^{-}) = - \infty 
\end{equation*}
Consequently, we have that 
\begin{equation*}
D^{+} (\{ \omega_{k}^{\pm}\}) =0.
\end{equation*}
\end{lemma}
\begin{proof}
Since $\omega_{-k}^{+} = -\omega_{k}^{+}$, it suffices to consider the case $k \rightarrow + \infty$. One denotes
$$
T^{\pm}(k) = (1+\alpha) - (\zeta + \mu)k^{-2} \pm \sqrt{[(1 - \alpha) + (\zeta - \mu)k^{-2}]^{2}+ 4k^{-4}\eta^{2}}
$$
Thus,
\begin{equation*}
T^{+}(k) =  2 + O(k^{-2}) \  \ \mbox{as} \ \ k \rightarrow \infty
\end{equation*}
and 
$$
\omega_{k}^{+}= \frac{1}{2}k^{3}T^{+}(k).
$$
Hence
$$
\omega_{k+1}^{+} - \omega_{k}^{+} = (k+1)^{3} - k^{3}+ O(k) = 3k^{2}+ 3k + 1 + O (k)\rightarrow + \infty,  \ \mbox{as} \ k \rightarrow +\infty .
$$
In the similar way
$$
\omega_{k+1}^{-} - \omega_{k}^{-} = \alpha [(k+1)^{3} - k^{3}]+ O(k) \rightarrow  -\infty,  \ \mbox{as} \ k \rightarrow +\infty ,
$$
where the last convergence is due to the fact that $\alpha < 0$. Now, as a consequence of these converges and by definition of $D^+\le 1/\gamma$, where 
$
\gamma = \gamma(\Omega)=\inf  \{ | \omega_k - \omega_n| \ : \ k \neq n\}>0,
$
we have that $D^{+} (\{ \omega_{k}^{\pm}\}) =0.$
\end{proof}
We now need to order our orthonormal basis, let us do it as follows. Consider $(\phi_{k}, \psi_{k}) = e^{-ikx}(\sigma_{k}, \tau_{k})$, so
\begin{equation}\label{ordenedbase}
(\phi_{k}, \psi_{k}) :=
\left\{   
\begin{array}{ccccc}
(\phi_{k}^{+}, \psi_{k}^{+}) = e^{-ikx}(\sigma_{k}^{+}, \tau_{k}^{+}) =e^{-ikx} Z_{k}^{+}, &\mbox{if }&  k = 2k' \mbox{ for all } k' \in \mathbb{Z},\\
(\phi_{k}^{-}, \psi_{k}^{-}) = e^{-ikx}(\sigma_{k}^{-}, \tau_{k}^{-})= e^{-ikx} Z_{k}^{-}, &\mbox{if }&  k = 2k' + 1 \mbox{ for all } k' \in \mathbb{Z}.
\end{array}
\right.
\end{equation}
Therefore, any vector $(u,v)\in H^{s}(\mathbb{T}) \times H^{s}(\mathbb{T})$ can be represented by
\begin{equation*}
(u,v) = \left(\sum_{k \in \mathbb{Z}} a_{k}\phi_{k} , \sum_{k \in \mathbb{Z}} b_{k} \psi_{k}  \right),
\end{equation*}
with the coefficients $a_k$ and $b_k$ are defined by
\begin{equation*}
a_{k} = \left< u, \phi_{k} \right> \ \mbox{and} \ b_{k} = \left< v, \psi_{k}\right> ,
\end{equation*}
where $\left< \cdot , \cdot \right>$ denoting the inner product in $L^{2}(\mathbb{T})$. Consider, also, the following
\begin{equation}\label{ordenedeingevalues}
\omega _{k} = \left\{   
\begin{array}{ccccc}
\omega_{k}^{+}, & \mbox{if } & k = 2k' \mbox{ for all } k' \in \mathbb{Z},\\
\omega_{k}^{-}, & \mbox{if } & k = 2k' + 1 \mbox{ for all } k' \in \mathbb{Z}.
\end{array}
\right.
\end{equation}

With these notions in hand, the following lemma gives the behavior of $\omega_k^{\pm}$ and concludes that the upper density of the set  $\{ \omega_{k}^{\pm}\}$ is zero. It is important to notice that to use \cite[Theorem 4.6]{KomLor2005} we need the following uniform gap condition $\gamma = \inf_{k \neq n} | \omega_{k} -\omega_{n}| > 0,$
where $\omega_{k}$ is defined by \eqref{ordenedeingevalues}. The next proposition will give us such information.

\begin{proposition}[Gap condition]\label{gapcondition}
Let $\omega_{k}$ be as in \eqref{ordenedeingevalues}. Thus,
\begin{equation*}
\lim_{|k|, |r| \rightarrow + \infty} |\omega_{k} - \omega_{r}| = + \infty.
\end{equation*}
\end{proposition}
\begin{proof}
 Start noting that Lemma \ref{nulldensity} ensures the result for $k$ and $n$ both odd or both even. Now, we need guarantee that the same is true for the other cases of $k$ and $n$. Consider  without loss of generality $r = 2k'$ and $k = 2(k' + k'')+1$ for any $k' \in \mathbb{Z}$ and $k''$ is a fixed positive integer. Using the notation of  Lemma \ref{nulldensity}, follows that
$$
\omega_{2(k' + k'') +1}  - \omega_{2k'} = 8(\alpha -1) k'^{3} + \alpha [ 12 k'^{2}(2k''+1) + 6k'(2k'' +1)^{2} + (2k'' + 1)^{3}]+  O (k').
$$
Thus,
\begin{equation*}
\lim_{|k'| \rightarrow + \infty} |\omega_{2(k' + k'')+1} - \omega_{2k'}| = + \infty, 
\end{equation*}
and then the proposition is proved.
\end{proof}

\begin{remark}
Thanks to \cite[Theorem 4.6]{KomLor2005} and Proposition \ref{gapcondition} there exists a subset $\mathbb{K} \subset \mathbb{Z}$ such that  $\overline{span \{e^{-i\omega_{k} t}\}_{k \in \mathbb{K}}}^{L^{2}(0,T)}$  has a unique biorthogonal Riesz basis $\{q_{k}\}\subset L^{2}(0,T)$, 
where 
\begin{equation}\label{defsetK}
\mathbb{K} = \{ k \in \mathbb{Z} \ ; \ \omega_{k} \neq \omega_{r} \ \mbox{for all} \ k \neq r\}.
\end{equation}
\end{remark}

\subsection{Exact controllability: Linear result} With these previous information that concern the spectral properties of the operator $L$, in this section, we will analyse the exact controllability for the linear system \eqref{openloopsystem_int}. 

%Precisely, given an initial state $(u_0,v_0)$ and a terminal state $(u_1,v_1)$ in a
%certain space, we will study the existence of two control functions $f$ and $g$ such that the system \eqref{openloopsystem_int}  admits a solution $(u,v)$ which satisfies $(u(x, T), v(x,T)) = (u_{1}(x), v_{1}(x))$.

Before to present the main result of this section, let us first consider some properties of the homogeneous initial value problem  (HIVP) associated with \eqref{openloopsystem_int}. It is well know, thanks to Proposition  \ref{propertiesL}, that \eqref{openloopsystem_int}, with $f=g=0$, has solution on the Sobolev space $H^s(\mathbb{T})$, for $s\in[0,3]$, which is given by
\begin{equation}\label{groupunitary}
(u(t), v(t)) = (S(t)u_{0},S(t) v_{0}) := \left( \sum_{k} e^{-i(\omega_{k}t + kx)} \widehat{u}_{0} ,  \sum_{k} e^{-i(\omega_{k}t + kx)} \widehat{v}_{0}\right).
\end{equation}
Additionally, using Semigroup Theory, see for instance  \cite[Theorems 1.1 and 1.4]{pazy1983semigroups},  we have that the open loop control system has a unique solution in $$C([0,T]; H^{3}(\mathbb{T})) \cap C^{1}([0,T]; L^{2}(\mathbb{T})) \times C([0,T]; H^{3}(\mathbb{T})) \cap C^{1}([0,T]; L^{2}(\mathbb{T})).$$

\begin{remark}\label{Gproperties}
Operator $G$ defined as in  \eqref{defG} from $L^{2}(\mathbb{T})$ to $L^{2}(\mathbb{T})$ is linear, bounded and self-adjoint. Actually, was proved in \cite[Remark 2.1]{ortegalinares2005} (see also \cite[Lemma 2.20]{micu2})  that operator $G$ is a linear bounded operator from $L^{2}(0,T;H^{s}(\mathbb{T}))$ into $L^{2}(0,T;H^{s}(\mathbb{T}))$, for any $s \geq 0$.
\end{remark}

Now on, we are in position to prove the exact controllability result.
\begin{proof}[Proof of Theorem \ref{main_A}]
Since the functions $(\phi_{k}, \psi_{k})$, defined by \eqref{ordenedbase}, form an orthonormal basis on $L^{2}(\mathbb{T}) \times L^{2}(\mathbb{T})$ and the space $L^{2}_{0}(\mathbb{T}) \times L^{2}_{0}(\mathbb{T})$ is a closed space, we can represent the initial and terminal states like expansions, which are  convergent in $H_{0}^{s}(\mathbb{T}) \times H_{0}^{s}(\mathbb{T})$, as follows %(for more details see \cite{capistranokp2020})
\begin{equation}\label{defexpansiondatas}
\begin{split}
u_{j} = \sum_{k \in \mathbb{Z}} u_{k,j} \phi_{k}, \ \ u_{k,j} = \int_{\mathbb{T}} u_{j}(x) \overline{\phi_{k}(x)} dx,\quad \text{ for } j=0,1,\\
v_{j} = \sum_{k \in \mathbb{Z}} v_{k,j} \psi_{k}, \ \ v_{k,j} = \int_{\mathbb{T}} v_{j}(x) \overline{\psi_{k}(x)} dx,\quad \text{ for } j=0,1.
\end{split}
\end{equation}

The solution of the homogeneous (adjoint) system can be expressed by $(u_{k}(x,t) , v_{k}(x,t)) = (e^{-i\omega_{k}t}\phi_{k}(x), e^{-i\omega_{k}t}\psi_{k}(x)),$
where $\omega_{k}$ are the eigenvalues defined in \eqref{ordenedeingevalues}.  Pick smooth functions $(f,h)$ on $\mathbb{T} \times \mathbb{T} $. Multiplying \eqref{openloopsystem_int} by $(\overline{u_{k}(x,t)}, \overline{v_{k}(x,t)})^{T}$ and using integration by parts on $\mathbb{T}\times(0,T)$, we obtain
\begin{equation}\label{identity5}
\begin{split}
\int_{\mathbb{T}} u(x,T) \overline{u_{k}(x,T)} dx - \int_{\mathbb{T}} u(x,0) \overline{u_{k}(x,0)} dx = \int_{0}^{T} \int_{\mathbb{T}} Gf(x,t)\overline{u_{k}(x,t)}dxdt,\\
\int_{\mathbb{T}} v(x,T) \overline{v_{k}(x,T)} dx - \int_{\mathbb{T}} v(x,0) \overline{v_{k}(x,0)} dx = \int_{0}^{T} \int_{\mathbb{T}} Gh(x,t)\overline{v_{k}(x,t)}dxdt,
\end{split}
\end{equation}
with the previous equality valid for $f$, $h \in L^{2}([0,T]; H_{0}^{s}(\mathbb{T})) $, for any $s \geq 0$, where $(u,v)$ satisfies \eqref{openloopsystem_int}.  Observe that $(\overline{u_{k}}, \overline{v_{k}}) = (e^{i\omega_{k}t}\overline{\phi_{k}(x))}, e^{i\omega_{k}t}\overline{\psi_{k}(x))} $. Moreover, thanks to \eqref{identity5}, we get that
\begin{equation*}
\int_{\mathbb{T}}u(x, T) e^{i\omega_{k}T} \overline{\phi_{k}(x)} dx - \int_{\mathbb{T}} u_{0}(x) \overline{\phi_{k}(x)}dx = \int_{0}^{T} \int_{\mathbb{T}} Gf(x,t) e^{i\omega_{k}t} \overline{\phi_{k}(x)} dx
\end{equation*}
and
\begin{equation*}
\int_{\mathbb{T}}v(x, T) e^{i\omega_{k}T} \overline{\psi_{k}(x)} dx - \int_{0}^{T} v_{0}(x) \overline{\psi_{k}(x)}dx = \int_{0}^{T} \int_{\mathbb{T}} Gh(x,t) e^{i\omega_{k}t} \overline{\psi_{k}(x)} dx.
\end{equation*}
Evaluation of the integrals in \eqref{identity5} with
\begin{equation}\label{idcontrol4}
w_{k} = \int_{\mathbb{T}} u(x,T) \overline{\phi_{k}(x)}dx \quad \text{and} \quad \ z_{k} = \int_{\mathbb{T}} v(x,T) \overline{\psi_{k}(x)}dx
\end{equation}
gives that
\begin{equation}\label{idcontrol1}
\begin{split}
 w_{k} - u_{k,0}e^{-i\omega_{k}T} = \int_{0}^{T}e^{-i\omega_{k}(T -t)}\int_{\mathbb{T}} Gf(x,t)\overline{\phi_{k}(x)}dxdt, \qquad \forall k \in \mathbb{Z},\\
 z_{k} - v_{k,0}e^{-i\omega_{k} T} = \int_{0}^{T}e^{-i\omega_{k}(T - t)}\int_{\mathbb{T}} Gh(x,t)\overline{\psi_{k}(x)}dxdt, \qquad \forall k \in \mathbb{Z}.
\end{split}
\end{equation}

Let us take our control functions $f$ and $h$ in the following way
\begin{equation}\label{defcontrol}
f(x,t) = \sum_{j \in \mathbb{Z}} f_{j} q_{j}(t)G\phi_{j}(x),  \quad \text{and} \quad
h(x,t) = \sum_{j \in \mathbb{Z}} h_{j} q_{j}(t)G\psi_{j}(x).
\end{equation}
Here the coefficients $f_{j}$ and $h_{j}$ must be determined so that, among other things, the series \eqref{defcontrol} is appropriately convergent. Substituting \eqref{defcontrol} into \eqref{idcontrol1} yields,
\begin{equation}\label{auxidcontrol1}
w_{k} - u_{k,0}e^{-i\omega_{k}T} = e^{-i\omega_{k} T} \sum_{j \in \mathbb{Z}}  f_{j} \int_{0}^{T}e^{i\omega_{k}t}  q_{j} (t)dt\int_{\mathbb{T}} G G\phi_{j}(x)\overline{\phi_{k}(x)}dx
\end{equation}
and
\begin{equation}\label{auxidcontrol2}
 z_{k} - v_{k,0}e^{-i\omega_{k} T} = e^{-i\omega_{k}T} \sum_{j \in \mathbb{Z}} h_{j} \int_{0}^{T}e^{
 i\omega_{k}t}  q_{j}(t)dt\int_{\mathbb{T}} G G\psi_{j}(x)\overline{\psi_{k}(x)}dx.
\end{equation}
Thanks to the fact that $\{q_{k}\}_{k \in \mathbb{K}}$ is a biorthogonal Riesz basis to $\{e^{-i\omega_{k} t}\}_{k \in \mathbb{K}}$ in $L^{2}_{0}(0,T)$, for $\mathbb{K}$ defined by \eqref{defsetK}, and due to the Remark \ref{Gproperties} we can get that
\begin{equation}\label{idcontrol2}
\begin{split}
w_{k} - u_{k,0}e^{-i\omega_{k}T} = e^{-i\omega_{k} T} f_{k} \int_{\mathbb{T}}  G\phi_{k}(x)\overline{G\phi_{k}(x)}dx = e^{-i\omega_{k} T} f_{k} \|G\phi_{k}\|^{2},\\
 z_{k} - v_{k,0}e^{-i\omega_{k} T} = e^{-i\omega_{k}T}  h_{k} \int_{\mathbb{T}}  G\psi_{k}(x),\overline{G\psi_{k}(x)}dx = e^{-i\omega_{k}T}  h_{k} \|G\psi_{k}\|^{2},
\end{split} 
\end{equation}
for all $k_{j} \in \mathbb{Z} \setminus \cup_{j =1}^{\ell}\mathbb{K}_{j}$, where $\mathbb{K}_{j}:= \{ k \in \mathbb{Z} \ ; \ \omega_{k} = \omega_{k_{j}} \ \mbox{and} \ k \neq k_{j} \} $. By the definition of $G$, see  \eqref{defG}, yield that
\begin{equation}\label{betak}
\begin{split}
\|G\phi_{k}
\|^{2} = &\int_{\mathbb{T}} \left| g(x) \left( \phi_{k}(x) - \int_{\mathbb{T}} g(y) \phi_{k}(y) dy \right)\right|^{2}dx 
= |\sigma_{k}|^{2} \beta_{k}
\end{split} 
\end{equation}
and 
\begin{equation}\label{gammak}
\begin{split}
\|G\psi_{k}\|^{2} = & \int_{\mathbb{T}} \left| g(x) \left( \psi_{k}(x) - \int_{\mathbb{T}} g(y) \psi_{k}(y) dy \right)\right|^{2}dx = |\tau_{k}|^{2} \beta_{k},
\end{split}
\end{equation}
where
\begin{equation*}
\beta_{k}:= \left\| G \left( \frac{e^{-ikx}}{\sqrt{2\pi}}\right)\right\|^{2}.
\end{equation*}
Since $[g]= \frac{1}{2\pi}$ it is easy to see that $\beta_{0} = 0$. The fact that $g(x)$ is real valued shows that $g(x)\frac{e^{-ikx}}{\sqrt{2\pi}}$ cannot be a constant multiple of $g(x)$ on any interval. Thus, follows that $\beta_{k} \neq 0$, $k>0$ and
\begin{equation*}
\lim_{k \rightarrow \infty} \beta_{k} = \int_{\mathbb{T}} g(x)^{2}dx \neq 0.
\end{equation*}
Its implies that there is a $\delta > 0$ such that
\begin{equation}\label{boundeddelta}
|\beta_{k}| > \delta, \quad \ \ \mbox{for }k \neq 0. 
\end{equation}
Due to the fact that $\sigma_{k}\neq 0$ and $\tau_{k} \neq 0$, for all $k$, we can putting $f_{0}=h_{0} = 0$ and 
\begin{equation}\label{defcoefcontrol}
f_{k} = \frac{u_{k,1}e^{i\omega_{k}T}- u_{k,0}}{|\sigma_{k}|^{2}\beta_{k}} \ \  \mbox{and} \  \ h_{k} = \frac{v_{k,1}e^{i\omega_{k}T} - v_{k,0}}{|\tau_{k}|^{2}\beta_{k}},
\end{equation}
 for all $k \in \mathbb{Z}^* \setminus \cup_{j=1}^{n} \mathbb{K}_{j}$. So  we get, from \eqref{idcontrol2}, that $w_{k} = u_{k,1}$ and $z_{k} = v_{k,1},$ where $u_{k,1}$ and $v_{k,1}$ are given by \eqref{defexpansiondatas}\footnote{Note that clearly $w_{0}$ and $z_{0}$ must be zero.}.  Since $\omega_{k}$ is given by a polynomial of degree 3, each set $\mathbb{K}_{j}$ has at most three elements. So, we can consider  $k_{j,i} \in \cup_{j=1}^{\ell} \mathbb{K}_{j}$ for $i=0,1,2$. In this case, from \eqref{auxidcontrol1}-\eqref{auxidcontrol2} follows that
\begin{equation}\label{defcoefcontrol2}
\begin{split}
 w_{k_{j,i}} - u_{k_{j,i},0} e^{-i \omega_{k_{j,0}}T} = \sigma_{k_{j,i}}e^{-i \omega_{k_{j,0}}T} \sum_{\ell=0}^{2} f_{k_{j,\ell}} \sigma_{k_{j, \ell}} M_{k_{j, \ell}k_{j, i}}, \\
  z_{k_{j,i}} - v_{k_{j,i},0} e^{-i \omega_{k_{j,0}}T} = \tau_{k_{j,i}}e^{-i \omega_{k_{j,0}}T} \sum_{\ell=0}^{2} h_{k_{j,\ell}} \tau_{k_{j, \ell}}  M_{k_{j, \ell}k_{j, i}} ,
\end{split}
\end{equation}
where 
$$
M_{k_{j, \ell} k_{j, i}} := \frac{1}{2 \pi} \int_{\mathbb{T}}GG \left( e^{-ik_{j, \ell}x} \right) \overline{e^{-ik_{j,i} x}}dx.
$$
In other words, $f_{k_{j, \ell}}$ and $h_{k_{j, \ell}}$, for each $j =1,2, \cdots , n$ and $\ell = 0,1,2$, must be satisfy the following matrix identities
\begin{equation*}
\left(  \begin{array}{ccccccccccccccccc}
\sigma_{k_{j,0}} M_{k_{j,0}, k_{j,0}} & \sigma_{k_{j,0}} M_{k_{j,1}, k_{j,0}} &\sigma_{k_{j,0}} M_{k_{j,2}, k_{j,0}}\\
\sigma_{k_{j,1}}M_{k_{j,0}, k_{j,1}} & \sigma_{k_{j,1}} M_{k_{j,1}, k_{j,1}} & \sigma_{k_{j,1}} M_{k_{j,2}, k_{j,1}}\\
\sigma_{k_{j,2}}M_{k_{j,0}, k_{j,2}} & \sigma_{k_{j,2}} M_{k_{j,1}, k_{j,2}} &  \sigma_{k_{j,2}}M_{k_{j,2}, k_{j,2}}
\end{array}
\right) \cdot  \left(
\begin{array}{ccccc}
\sigma_{k_{j,0}}f_{k_{j, 0}} \\
 \sigma_{k_{j,1}}f_{k_{j, 1}} \\
  \sigma_{k_{j,2}}f_{k_{j, 2}}
\end{array}
\right) =  \left( \begin{array}{ccc}
w_{k_{j,0}} e^{i \omega_{k_{j,0}}T} - u_{k_{j,0},0}\\
w_{k_{j,1}} e^{i \omega_{k_{j,0}}T} - u_{k_{j,1},0}\\
w_{k_{j,2}} e^{i \omega_{k_{j,0}}T} - u_{k_{j,2},0}
\end{array} \right)
\end{equation*}
and
\begin{equation*}
\left(  \begin{array}{ccccccccccccccccc}
\tau_{k_{j,0}}M_{k_{j,0}, k_{j,0}} & \tau_{k_{j,0}} M_{k_{j,1}, k_{j,0}} & \tau_{k_{j,0}}M_{k_{j,2}, k_{j,0}}\\
\tau_{k_{j,1}} M_{k_{j,0}, k_{j,1}} &  \tau_{k_{j,1}} M_{k_{j,1}, k_{j,1}} & \tau_{k_{j,1}} M_{k_{j,2}, k_{j,1}}\\
\tau_{k_{j,2}} M_{k_{j,0}, k_{j,2}} &  \tau_{k_{j,2}} M_{k_{j,1}, k_{j,2}} & \tau_{k_{j,2}} M_{k_{j,2}, k_{j,2}}
\end{array}
\right) \cdot  \left(
\begin{array}{ccccc}
\tau_{k_{j,0}} h_{k_{j, 0}} \\
 \tau_{k_{j,1}} h_{k_{j, 1}} \\
  \tau_{k_{j,2}}h_{k_{j, 2}}
\end{array}
\right) = \left( \begin{array}{ccc}
z_{k_{j,0}} e^{i \omega_{k_{j,0}}T} - v_{k_{j,0},0}\\
z_{k_{j,1}} e^{i \omega_{k_{j,0}}T} - v_{k_{j,1},0}\\
z_{k_{j,2}} e^{i \omega_{k_{j,0}}T} - v_{k_{j,2},0}
\end{array} \right).
\end{equation*}

In order to achieve the result, we will need to prove the following two claims.

\vspace{0.2cm}

\noindent\textbf{Claim 1.} The previous systems have a unique solution $(f_{k_{j,0}}, f_{k_{j,1}}, f_{k_{j,2}})$ and $(h_{k_{j,0}}, h_{k_{j,1}}, h_{k_{j,2}})$, for each $j =1, 2, \cdots , n$. 

\vspace{0.2cm}

Indeed, note that the determinant of the above matrices are given by $\sigma_{k_{j,0}} \times\sigma_{k_{j,1}} \times\sigma_{k_{j,2}}\times\det M_{j}$ and $\tau_{k_{j,0}}\times \tau_{k_{j,1}} \times\tau_{k_{j,2}}\times\det M_{j}$, respectively, with $M_j$ defined by
$$M_{j}:=\left(\begin{array}{lll}
m_{k_{j, 0}, k_{j, 0}} & m_{k_{j, 0}, k_{j, 1}} & m_{k_{j, 0}, k_{j, 2}} \\
m_{k_{j, 1}, k_{j, 0}} & m_{k_{j, 1}, k_{j, 1}} & m_{k_{j, 1}, k_{j, 2}} \\
m_{k_{j, 2}, k_{j, 0}} & m_{k_{j, 2}, k_{j, 1}} & m_{k_{j, 2}, k_{j, 2}}
\end{array}\right).$$
Since $\sigma_{k_{j,0}} \times\sigma_{k_{j,1}} \times\sigma_{k_{j,2}}\neq 0$ and $\tau_{k_{j,0}}\times \tau_{k_{j,1}} \times\tau_{k_{j,2}} \neq 0$, we only have show that the hermitian matrices $M_{j}$ are invertible for all $j=1, \cdots , \ell$.
For fixed $j$, let us consider $\Sigma_{2}$ the space spanned by $\Upsilon_{k_{j,i}}= e^{-ik_{j,i}}$, $i=0,1,2$. Let $\rho_{k_{j , \ell}}$ be the projection of $GG(\Upsilon_{k_{j, \ell}})$ onto the space $\Sigma_{2}$, that is, 
\begin{equation*}
\rho_{k_{j ,\ell}} = \sum_{i=0}^{2}M_{k_{j, \ell} k_{j, i}} \Upsilon_{k_{j, i}}.
\end{equation*}
Now, it suffices to show that $\rho_{k_{j, \ell}}$, $\ell = 0,1, 2$, is a linearly independent subset of $\Sigma_{2}$. Assume that there exist scalars $\lambda_{\ell}$, $\ell =0, 1, 2$, such that
$$
\sum_{\ell =0}^{2} \lambda_{\ell} \rho_{k_{j , \ell}}(x) = 0 \iff \sum_{\ell , i =0}^{2} \lambda_{\ell} M_{k_{j, \ell}, k_{j, i}} \Upsilon_{k_{j, i}}(x) =0
$$
Then, it yields that 
\begin{equation*}
\sum_{i = 0}^{2} \sum_{\ell =0}^{2} \left< \lambda_{\ell} G\Upsilon_{k_{j, \ell}}, G \Upsilon_{k_{j, i}}\right> \Upsilon_{k_{j,i}}= \sum_{i =0}^{2} \left< GG \left( \sum_{\ell =0}^{2} \lambda_{\ell} \Upsilon_{k_{j, \ell}} \right) , \Upsilon_{k_{j, i}} \right>\Upsilon_{k_{j, i}}=0
\end{equation*}
Since $\Upsilon_{k_{j, i}}$ is a basis of $\Sigma_{2}$, follows that
$$
\left< GG \left( \sum_{\ell =0}^{2} \lambda_{\ell} \Gamma_{k_{j, \ell}} \right) , \Gamma_{k_{j, i}} \right> =0,
$$
for each $i=0,1,2$. As consequence of the last equality, we get
$$
0=\left< GG \left( \sum_{\ell =0}^{2} \lambda_{\ell} 
 \Upsilon_{k_{j, \ell}} \right) ,  \sum_{i=0}^{2} \lambda_{i}\Upsilon_{k_{j, i}} \right> \iff \sum_{\ell =0}^{2} \lambda_{\ell} \Upsilon_{k_{j, \ell}}  =0 \iff \lambda_{\ell} = 0,
$$
for $\ell = 0,1,2$, showing the Claim 1.

\vspace{0.2cm}

\noindent\textbf{Claim 2.} The functions $f$ and $h$ defined by \eqref{defcontrol} and \eqref{defcoefcontrol} belongs to $L^{2}([0,T]; H_{0}^{s}(\mathbb{T}))$ provided that $(u_{0}, v_{0}), (u_{1},v_{1}) \in H_{0}^{s}(\mathbb{T}) \times H_{0}^{s}(\mathbb{T})$.

\vspace{0.2cm}

In fact, let us write $G\phi_{j}(x)$ and $G\psi_{j}(x)$ as follows
\begin{equation}\label{defcontrolcoef2}
G\phi_{j}(x) = \sum_{k \in \mathbb{Z}} a_{jk}\phi_{k} \ \  \mbox{and} \ \ G\psi_{j}(x) = \sum_{k \in \mathbb{Z}} b_{jk}\psi_{k},
\end{equation}
where
\begin{equation*}
a_{jk} = \int_{\mathbb{T}} G\phi_{j} \overline{\phi_{k}(x)} dx \ \ \mbox{and} \ \ b_{jk} = \int_{\mathbb{T}} G\psi_{j} \overline{\psi_{k}(x)} dx, \ k \in \mathbb{Z}.
\end{equation*}
Therefore, we can see that
\begin{equation*}
f(x,t) = \sum_{j \in \mathbb{Z}}\sum_{k \in \mathbb{Z}}f_{j}a_{jk}q_{j}(t)\phi_{k}(x)\quad \text{and} \quad  h(x,t) = \sum_{j \in \mathbb{Z}}\sum_{k \in \mathbb{Z}}h_{j}b_{jk}q_{j}(t)\psi_{k}(x).
\end{equation*}
Consequently, this yields that
\begin{equation*}
\|f\|_{L^{2}([0,T]; H_{0}^{s}(\mathbb{T}))}^{2} = \int_{0}^{T} \sum_{k \in \mathbb{Z}} (1 + |k|)^{2s} \left| \sum_{j \in \mathbb{Z}} a_{jk} f_{j} q_{j}(t) \right|^{2} dt = \sum_{k \in \mathbb{Z}} (1 + |k|)^{2s} \int_{0}^{T}\left| \sum_{j \in \mathbb{Z}} a_{jk} f_{j} q_{j}(t) \right|^{2} dt.
\end{equation*}
As $\{q_{k}\}_{k \in \mathbb{K}}$ is a Bessel sequence and $\mathbb{Z} \setminus \mathbb{K}$ is a finite set, from the previous identity holds that
\begin{equation}\label{f_control}
\|f\|_{L^{2}([0,T]; H_{0}^{s}(\mathbb{T}))} \leq c \sum_{j \in \mathbb{Z}}|f_{j}|^{2} \sum_{k \in \mathbb{Z}} (1 + |k|)^{2s}|a_{jk}|^{2}.
\end{equation}
Analogously, we can obtain the following estimate for $h$, that is, 
\begin{equation}\label{g_control}
\|h\|_{L^{2}([0,T]; H_{0}^{s}(\mathbb{T}))} \leq c \sum_{j \in \mathbb{Z}}|h_{j}|^{2} \sum_{k \in \mathbb{Z}} (1 + |k|)^{2s}|b_{jk}|^{2}.
\end{equation}

To finish the proof of Claim 2, let us prove that the right hand side of \eqref{f_control} and \eqref{g_control} are bounded. For this, note that
\begin{equation*}
\begin{split}
|a_{jk}| =  \left|\left< G\phi_{j}(x), \phi_{k}(x)\right> \right|
%=&  \left| \left< g(x)\phi_{j}(x) -  g(x) \int_{\mathbb{T}}g(y)\phi_{j}(y)dy, \phi_{k}(x)\right> \right|\\
%= & \left| \int_{\mathbb{T}} \sum_{m \in \mathbb{Z}} \left< g, \phi_{m} \right>\phi_{m}(x) \phi_{j}(x)\overline{\phi_{k}(x)} dx  - \int_{\mathbb{T}}  \sum_{m \in \mathbb{Z}} \left< g, \phi_{m} \right>\phi_{m}(x)  \overline{ \phi_{k}(x)} dx \int_{\mathbb{T}}  \sum_{m \in \mathbb{Z}} \left< g, \phi_{m} \right>\phi_{m}(y)\phi_{j}(y)dy  \right|\\
%= & \left|\frac{1}{\sqrt{2 \pi}}\sigma_{k-j} \sigma_{j}\sigma_{k}\left<g, \phi_{k-j} \right> - \sigma_{k}\sigma_{j} \left< g, \phi_{k}\right>\left< g, \phi_{j}\right>\right|
 \leq   \frac{1}{\sqrt{2\pi}}|\sigma_{k-j}||\sigma_{j}||\sigma_{k}||\left< g, \phi_{k-j} \right>| +|\sigma_{k}||\sigma_{j}||\left<  g, \phi_{k}\right> ||\left< g, \phi_{j}\right>|
\end{split}
\end{equation*}
and, in a similar way,
\begin{equation*}
|b_{jk}| \leq  \frac{1}{\sqrt{2\pi}}|\tau_{k-j}||\tau_{j}||\tau_{k}||\left< g , \psi_{k-j}\right> | +|\tau_{k}||\tau_{j}|| \left< g, \psi_{k} \right> ||\left< g, \psi_{j}\right>|.
\end{equation*}
Hence,
\begin{equation*}
|a_{jk}|^{2}  \leq 2|\sigma_{j}|^{2}| (|\sigma_{k-j}|^{2}|\sigma_{k}|^{2}|\left< g , \phi_{k-j} \right> |^{2} + |\sigma_{k}|^{2} |\left< g, \phi_{k}\right> |^{2}|\left<g, \phi_{j} \right>|^{2}  )
\end{equation*}
and
\begin{equation*}
|b_{jk}|^{2} \leq 2 |\tau_{j}|^{2}(|\tau_{k-j}|^{2}|\tau_{k}|^{2}|\left< g, \psi_{k-j}\right>|^{2} +  |\tau_{k}|^{2}|\left<g, \psi_{k}\right>|^{2} |\left<g, \psi_{j}\right>|^{2}). 
\end{equation*}
Using the last inequalities we can estimate
\begin{equation*}
\begin{split}
\sum_{k \in \mathbb{Z}} (1 + |k|)^{2s}|a_{jk}|^{2} 
%\leq & 2 |\sigma_{j}|^{2} \left[\sum_{k \in \mathbb{Z}} (1 + |k|)^{2s}|\sigma_{k-j}|^{2}|\sigma_{k}|^{2}|\left< g, \phi_{k-j}\right>|^{2} +  \sum_{k \in \mathbb{Z}} (1 + |k|)^{2s} |\sigma_{k}|^{2} |\left<g, \phi_{k} \right>|^{2}|\left<g, \phi_{j}\right>|^{2} \right]\\
%\leq & 2|\sigma_{j}|^{2}\left[ \sum_{k \in \mathbb{Z}} (1 + |k + j|)^{2s}|g_{k}|^{2} +  \sum_{k \in \mathbb{Z}} (1 + |k|)^{2s}  |g_{k}|^{2}|g_{j}|^{2} \ \right]\\
\leq & 2 |\sigma_{j}|^{2} \left[ (1 + |j|)^{2s}\sum_{k \in \mathbb{Z}} (1 + |k|)^{2s}|\left< g, \phi_{k}\right>|^{2} +  |\left< g. \phi_{j} \right>|^{2}\sum_{k \in \mathbb{Z}} (1 + |k|)^{2s}  |\left<g, \phi_{k}\right>|^{2}\right]\\
 \leq & 2 |\sigma_{j}|^{2}\left[(1 + |j|)^{2s} + |\left<g, \phi_{j}\right>|^{2}\right]\|g\|_{s}^{2},
\end{split}
\end{equation*}
and analogously, we have
\begin{equation*}
\sum_{k \in \mathbb{Z}} (1 + |k|)^{2s}|b_{jk}|^{2} \leq  2|\tau_{j}|^{2} \left[(1 + |j|)^{2s} + |\left<g, \psi_{j}\right>|^{2}\right]\|g\|_{s}^{2}.
\end{equation*}
Therefore, \eqref{f_control} and \eqref{g_control} together the previous inequality results
\begin{equation*}
\begin{split}
\|f\|_{L^{2}([0,T]; H_{0}^{s}(\mathbb{T}))}^{2}
%\leq & 2 \sum_{j \in \mathbb{Z} }|f_{j}|^{2} |\sigma_{j}|^{2}[(1 + |j|)^{2s} + |\left< g, \phi_{j}\right> |^{2}]\|g\|_{s}^{2} \\
%\leq & 2C_{0} \sum_{j \in \mathbb{Z} } \frac{|u_{j,1}e^{i\omega_{j}T}- u_{j,0}|^{2}}{|\sigma_{j}|^{4}|\beta_{j}|^{2}}|\sigma_{j}|^{2}[(1 + |j|)^{2s} + |\left< g_, \phi_{j}\right>|^{2}]\|g\|_{H^{s}(\mathbb{T})}^{2}\\
\leq & 2C_{0} \sum_{j \in \mathbb{Z} } \frac{|u_{j,1}e^{i\omega_{j}T}- u_{j,0}|^{2}}{|\sigma_{j}|^{2}|\beta_{j}|^{2}}[(1 + |j|)^{2s} + |\left< g, \phi_{j}\right> |^{2}]\|g\|_{s}^{2},
\end{split}
\end{equation*}
where $C_{0} = \max_{j = 1, \cdots , n} \{ 1, \|M_{j}^{-1}\|^{2}\}$ and  $\|M_{j}^{-1}\|$ denote the Euclidean norms of the Matrices $M_{j}^{-1}$. An analogous inequality is obtained for $\|h\|_{L^{2}([0,T]; H_{0}^{s}(\mathbb{T}))}^{2}$. Putting all these inequalities together and using the relation \eqref{boundeddelta}, we get
\begin{equation*}
\begin{split}
\|(f, h)\|_{L^{2}([0,T]; H_{0}^{s}(\mathbb{T})) \times L^{2}([0,T]; H_{0}^{s}(\mathbb{T}))}^{2} 
 \leq & \ C_0 \delta^{-2}\|g\|_{H_{0}^{s}(\mathbb{T})}^{2}  \sum_{j \in \mathbb{Z} } (1 + |j|)^{2s} \frac{|\sigma_{j}|^{2}( |\widetilde{u_{j,1} }|^{2} + |\widetilde{u_{j,0}}|^{2})}{|\sigma_{j}|^{2}}\\
 & + C_0  \delta^{-2}\|g\|_{H_{0}^{s}(\mathbb{T})}^{2}  \sum_{j \in \mathbb{Z} } (1 + |j|)^{2s} \frac{|\tau_{j}|^{2}( |\widetilde{v_{j,1}} |^{2} + |\widetilde{v_{j,0}}|^{2})}{|\tau_{j}|^{2}},
  \end{split}
\end{equation*}
where $\widetilde{u_{j,i}}$ and $\widetilde{v_{j,i}}$ denote the Fourier coefficients with respect to the orthonormal base $\left\{ \frac{e^{-ijx}}{\sqrt{2\pi}}\right\}_{j \in \mathbb{Z}}$.
So,
$$
\|(f, h)\|_{L^{2}([0,T]; H_{0}^{s}(\mathbb{T})) \times L^{2}([0,T]; H_{0}^{s}(\mathbb{T}))}^{2} \leq K_{0} \delta^{-2}
   \|g\|_{H_{0}^{s}(\mathbb{T})}^{2} (\|(u_{0}, v_{0})\|_{H_{0}^{s}( \mathbb{T}) \times H_{0}^{s}(\mathbb{T})}^{2} +  \|(u_{1}, v_{1})\|_{H_{0}^{s}( \mathbb{T}) \times H_{0}^{s}(\mathbb{T})}^{2})
$$
completing the proof of Claim 2 and showing Theorem \ref{main_A}.
\end{proof}

As a consequence of Theorem \ref{main_A} we have the next result, which will be important to extend the result to the nonlinear system.

\begin{corollary}\label{corcontrollocal}
Equations \eqref{defcontrol}, \eqref{defcoefcontrol} and \eqref{defcontrolcoef2} define, for $s \geq 0$, two bounded operators $\Phi(u_{0}, v_{0}) = f$ and $\Psi(u_{1}, v_{1}) = h$
from $H_{0}^{s}(\mathbb{T}) \times H_{0}^{s}(\mathbb{T})$ to $L^{2}([0,T]; H_{0}^{s}(\mathbb{T}))$ such that
\begin{equation*}
S(T)(u_{0},v_0) + \int_{0}^{T}S(T -\tau) (G\Phi(u_{0}, u_{1}),G\Psi(v_{0}, v_{1}))(\cdot , \tau) d\tau = (u_{1},v_1),
\end{equation*}
for any $(u_{0}, u_{1})$, $(v_{0}, v_{1}) \in H_{0}^{s}(\mathbb{T}) \times H_{0}^{s}(\mathbb{T})$. Moreover, there exists a constant $C_{T,g}:=C(T,g)$ such that the following inequality is verified
\begin{equation*}
\|(\Phi (u_{0}, u_{1}), \Psi (v_{0}, v_{1}))\|_{[L^{2}([0,T]; H_{0}^{s}(\mathbb{T}))]^2} \leq C_{T,g} \left(\|(u_{0}, v_{0})\|_{s} +  \|(u_{1}, v_{1})\|_{s}\right).
\end{equation*}
\end{corollary}

\section{Well-posedness theory in Bourgain spaces}\label{Sec3}
Is well known that the Bourgain in \cite{Bourgain} discovered a subtle smoothing property of solutions of the KdV equation posed on Torus, thanks to which he was able to show that the KdV equation is well-posed in the space $H^s(\mathbb{T})$, for any $s\geq0$.  In this section we will present the smoothing properties to the IVP \eqref{openloopsystem_int}, considering $f=g=0$, which are the key to prove the global control results in this manuscript, using the techniques introduced by Bourgain.

\subsection{Fourier restriction space} Observe that  the IVP \eqref{openloopsystem_int}, with $f=g=0$,  can be rewrite as
\begin{equation*}
\left\{\begin{array}{l}\left(\begin{array}{l}u_{t} \\ v_{t}\end{array}\right)+\left(\begin{array}{cc}1 & 0 \\ 0 & \alpha\end{array}\right)\left(\begin{array}{l}u_{x x x} \\ v_{x x x}\end{array}\right)+\left(\begin{array}{cc}\mu & \eta \\ \eta & \zeta\end{array}\right)\left(\begin{array}{l}u_{x} \\ v_{x}\end{array}\right)=\left(\begin{array}{l}0\\ 0\end{array}\right), \quad x \in \mathrm{T}, t \in \mathbb{R}, \\ \left.\left(\begin{array}{l}u \\ v\end{array}\right)\right|_{t=0}=\left(\begin{array}{l}u_{0} \\ v_{0}\end{array}\right) \in H^{s}\left(\mathbb{T}\right) \times H^{s}\left(\mathbb{T}\right).\end{array}\right.
\end{equation*}
To find an appropriate way to define the $X_{s,b}$ for the targeted system \eqref{openloopsystem_int}, taking into account that$ f=g=0$,  consider the following equivalent system 
\begin{equation}\label{linearsystem1a}
\begin{cases}
\partial_{t} w+ \beta \partial_{x}^{3} w+ \gamma \partial_{x} w=0,&  \ x \in \mathbb{T},\ t \in \mathbb{R},\\
w(0)=w_{0}, & \in H^{s}\left(\mathbb{T}\right).
\end{cases}
\end{equation}
The solution to \eqref{linearsystem1a} is given explicitly by
\begin{equation}\label{group_n}
w(x, t)=\sum_{k \in\mathbb{Z}} e^{i k x} e^{i \phi^{\beta, \gamma}(k) t} \widehat{w_{0}}(k):=S^{\beta, \gamma}(t) w_{0}
\end{equation}
with
$$
\phi^{\beta, \gamma}(k):=\beta k^{3}-\gamma k.
$$
For convenience, $\phi^{1, 0}$ will be written as $\phi$.

\begin{remark}\label{controolsistemdiagonal}
With the notation \eqref{group_n}, note that when we put $\eta =0$ in \eqref{operatorL}, the operator $L$  remains an infinitesimal generator of a strongly continuous group of contraction on $L^{2}(\mathbb{T}) \times L^{2}(\mathbb{T})$ which is given by 
\begin{equation*}
S(t) = e^{-tL} = 
\left(\begin{matrix} e^{-t(\partial_{x}^{3} + \mu \partial_{x})} & 0 \\ 0 & e^{-t(\alpha \partial_{x}^{3} + \zeta \partial_{x})}  \end{matrix}\right)
\end{equation*}
Hence, $S(t)(u_{0}, v_{0}) = (S^{1, \mu}(t)u_{0}, S^{\alpha, \zeta(t)}v_{0})$. In this way, the Corollary \ref{corcontrollocal} also is obtained for $\eta =0$.
\end{remark}

\begin{definition}\footnote{We infer for more details the two references \cite{Bourgain,Kenig}.} For any $ \beta, \gamma, s, b  \in \mathbb{R}$, the Fourier restriction space $X_{s, b}^{\beta , \gamma}$ is defined to be the completion of the Schwartz space $\mathcal{S}\left(\mathbb{T} \times \mathbb{R}\right)$ with respect to the norm
$$
\|w\|_{X_{s, b}^{\beta , \gamma}}:=\left\|\langle k\rangle^{s}\left\langle\tau-\phi^{\beta , \gamma}(k)\right\rangle^{b} \widetilde{w}(k, \tau)\right\|_{\ell^{2}(\mathbb{Z})L^{2}(\mathbb{R})},
$$
where $\widetilde{v}$ refers to the space-time Fourier transform of $v$. In addition, for any $T>0$, $$X_{s, b}^{\beta}([0, T]):=X_{s, b}^{\beta , \gamma ,T}$$ denotes the restriction of $X_{s, b}^{\beta , \gamma}$ on the domain $\mathbb{T} \times[0, T]$ which is a Banach space when equipped with the usual quotient norm.
\end{definition}
As well known (see e. g. \cite{Kenig}), for the periodic KdV equation, one needs to take $b=\frac{1}{2}$. But, this space barely fails to be in $C\left(\mathbb{R}_{t} ; H_{x}^{s}\right)$. To ensure the continuity of the time flow of the solution, will be used the norm $Y_{s,b}^{ \beta , \gamma}$ given by
$$
\|w\|_{Y_{s,b}^{\beta , \gamma }}=\left\|\langle k\rangle^{s} \langle \tau - \phi_{\beta, \gamma} \rangle^{b}\widetilde{w}(k, \tau)\right\|_{\ell^{2}(\mathbb{Z}) L^{1}(\mathbb{R})}
$$
and  the companion spaces will be defined as 
\begin{equation*}
Z^{\beta , \gamma}_{s,b} = X^{\beta , \gamma}_{s,b} \cap Y^{\beta , \gamma}_{s, b - \frac{1}{2}}, \quad b, s \in \mathbb{R},
\end{equation*}
endowed  with the norm
\begin{equation*}
\|w\|_{Z^{\beta , \gamma}_{s,b}} = \|w\|_{X^{\beta , \gamma}_{s, b}} + \|w\|_{Y^{\beta , \gamma}_{s, b - \frac{1}{2}}}.
\end{equation*}
Due the fact that the second term $\left\|\langle k\rangle^{s} \widehat{w}(k, \tau)\right\|_{\ell^{2}\left(\mathbb{Z}\right) L^{1}(\mathbb{R})}$  has already dominated the $L_{t}^{\infty} H_{x}^{s}$ norm of $v,$ it follows that $Z_{s, \frac{1}{2}}^{\beta, \gamma} \subset C\left(\mathbb{R}_{t} ; H_{x}^{s}\right)$ continuously. Lastly, the spaces $$Z_{s, b}^{\beta, \gamma}([0, T]):=Z_{s, b}^{\beta , \gamma ,T}$$ denotes the restriction of $Z_{s, b}^{\beta , \gamma}$ on the domain $\mathbb{T} \times[0, T]$ which is a Banach space when equipped with the usual quotient norm.
\begin{remark}
When $b=-\frac{1}{2}$, the companion spaces $ Z_{s, -\frac{1}{2}}^{\beta , \gamma}$  \textit{via} the norm previously defined is so introduced to control the $Z^{\beta , \gamma}_{s, \frac{1}{2}}$-norm of the integral term from the Duhamel principle (see Lemma \ref{bourgainkdv-estimates})
$$
\|w\|_{Z_{s}^{\beta , \gamma}}=\|w\|_{X_{s,-\frac{1}{2}}^{ \beta , \gamma}}+\left\|\frac{\langle k\rangle^{s} \widehat{w}(k, \tau)}{\left\langle\tau-\phi^{ \beta , \gamma}(k)\right\rangle}\right\|_{\ell^{2}\left(\mathbb{Z}\right) L^{1}(\mathbb{R})}.
$$
\end{remark}

\subsection{Linear and nonlinear estimates}
To obtain global well-posedness result for the system \eqref{systema}, with $p=q=0$, we will need some estimates related with linear and nonlinear IVP associated to this system.  Let us first recall some classic results in the literature for dispersive systems.
\begin{lemma}\label{bourgainkdv-estimates}\footnote{For details about this lemma the authors suggest the following references \cite{collianderetal2003,Tao,Zhang}.}
Let $s$, $b \in \mathbb{R}$ and $T > 0$ be given. There exists a constant $C_{0} > 0$ such that:
\begin{itemize}
    \item [(i)] For any $w \in H^{s}(\mathbb{T})$, 
    \begin{equation*}
    \begin{split}
    \|S(t)^{\beta, \gamma} w \|_{X^{\beta , \gamma ,T}_{s,b}} \leq C_{0} \|w\|_{s};\\
 \|S(t)^{\beta, \gamma} w \|_{Z^{\beta, \gamma ,T}_{s,b}} \leq C_{0} \|w\|_{s};\\
    \end{split}
    \end{equation*}
    \item [(ii)] For any $f \in X^{\beta , \gamma}_{s, b-1}$,
    \begin{equation*}
       \displaystyle  \left\| \int_{0}^{t} S^{\beta , \gamma}(t - \tau) f(\tau) d\tau \right\|_{X^{\beta , \gamma,T}_{s,b}} \leq C_{0}\|f\|_{X^{\beta , \gamma ,T}_{s, b-1}}
    \end{equation*}
    provided that $b > \frac{1}{2}$;
    \item [(iii)] For any $f \in Z_{s, - \frac{1}{2}}^{\beta,T}$,
    \begin{equation*}
      \displaystyle  \left\| \int_{0}^{t} S^{\beta, \gamma} (t - \tau)f(\tau) d\tau \right\|_{Y_{s}^{\beta, \gamma , T}} \leq C_{0} \|f\|_{Z_{s}^{\beta, \gamma , T}}.
    \end{equation*}
\end{itemize}
\end{lemma}

Observe that the Bourgain spaces associated to \eqref{linearsystem1a} will be $X_{s, b}^{1, \mu}$ and $X_{s, b}^{\alpha, \zeta}$ ($Z_{s}^{1, \mu}$ and $Z_{s}^{\alpha, \zeta}$, respectively). In our case, it is important to see that   $\sup_{k \in \mathbb{Z}}|\phi^{\mu} - \phi^{\alpha , \zeta}| = \infty$, which results that the norms  $\|\cdot\|_{X_{s, b}^{1, \mu}}$ and $\|\cdot\|_{X_{s, b}^{\alpha, \zeta}}$  never will be equivalent (see for instance, \cite[Remark 1.1]{axeltheses}). To overcome this difficulty we need appropriate lemmas introduced first by T.  Oh in \cite{Oh} and,  more recently,  by Yang and Zhang in \cite{Zhang}.  Consider $X_{s, b}^{\beta_{i}, \gamma_{i}}$ for $\beta_{i}$ and $\gamma_{i}$, $i=1$ and $2$.  So, we present a lemma proved in \cite[Lemma 3.10]{Zhang} for the case $b= \frac{1}{2}$, here, we are able to extend the result for $b\in(\frac{1}{3},\frac{1}{2}]$.
\begin{lemma}\label{derivativeestimate}
Let $\beta_{1} \neq \beta_{2}$, $s \in \mathbb{R}$, $ \frac{1}{3} < b \leq \frac{1}{2}$ and $0 < T < 1$. There exist constants $\epsilon = \epsilon(\beta_{1}, \beta_{2})$, $C_{1}= C_{1}(\beta_{1}, \beta_{2})$ and $\theta > 0$ such that for any $\gamma_{1}, \gamma_{2}$ with $|\gamma_{1}| + |\gamma_{2}| < \epsilon$
\begin{equation}\label{derivativenorm}
\left\| \partial_{x} w \right\|_{Z_{s, b-1}^{\beta_{2}, \gamma_{2} , T}} \leq C_{1} T^{\theta}\| w \|_{X_{s,b}^{ \beta_{1}, \gamma_{1}, T}}
\end{equation}
is verified for any $w \in X_{s,b}^{\beta_{1}, \gamma_{1}, T}$.
\end{lemma}
\begin{proof}
We must to prove that
\begin{equation}\label{norm1}
\left\| \partial_{x}w\right\|_{X_{s,b - 1}^{\beta_{2}, \gamma_{2}, T}} \leq C_{1} T^{\epsilon}\|w\|_{X_{s, b},}^{\beta_{1}, \gamma_{1}, T}
\  \ \mbox{and}  \ \ 
\left\| \partial_{x}w\right\|_{Y_{s,b-\frac{3}{2}}^{\beta_{2}, \gamma_{2} , T}} \leq C_{1} T^{\epsilon}\|w\|_{X_{s, b}^{\beta_{1}, \gamma_{1}, T} }.
\end{equation}
Thus, it is sufficient to show the following estimates
\begin{equation}\label{norm2}
\left\| \partial_{x}w\right\|_{X_{s,b-1}^{\beta_{2}, \gamma_{2}}} \leq C_{1} \|w\|_{X_{s, b^{-}}^{\beta_{1}, \gamma_{1}}}
\ \ \mbox{and}  \ \ 
\left\| \partial_{x}w\right\|_{Y_{s,b-\frac{3}{2}}^{\beta_{2}, \gamma_{2} , T}} \leq C_{1} \|w\|_{X_{s, b^{-} }^{\beta_{1}, \gamma_{1}} },
\end{equation}
here $b^{-}$ denote $b - \widetilde{\epsilon}$ for $\widetilde{\epsilon} \ll 1$.

We will start by showing that  the first inequality of \eqref{norm2} holds. Using duality approach and Plancherel theorem, we get that 
\begin{equation*}
\begin{split}
\|\partial_{x}w\|_{X_{s,  b-1}^{\beta_{2}, \gamma_{2}
}} = &\sup_{\|g\|_{X_{-s, 1- b}^{\beta_{2}, \gamma_{2}}} \leq 1}  \left|\sum_{k \in \mathbb{Z}} \int_{\mathbb{R}}ik  \widetilde{w}(k, \tau) \widetilde{g}(k, \tau) d\tau \right| 
=  \sup_{\|g\|_{X_{-s, 1- b}^{\beta_{2}, \gamma_{2}}} \leq 1} \left| \sum_{k \in \mathbb{Z}} \int_{\mathbb{R}} H(k, \tau ) \widetilde{W}(k, \tau ) \widetilde{G}(k, \tau)d \tau\right|,
\end{split}
\end{equation*}
where
$$
H(k, \tau ) = \frac{ik}{\left< \tau  - \phi^{\beta_{1}, \gamma_{1}}(k)\right>^{b^{-}} \left< \tau  - \phi^{\beta_{2}, \gamma_{2}}(k)\right>^{1- b}},
$$

$$
\widetilde{W}(k, \tau) = \left< k \right>^{s}\left< \tau  - \phi^{\beta_{1}, \gamma_{1}}(k)\right>^{b^{-}}\widetilde{w}(k, \tau)
$$
and
$$
\widetilde{G}(k, \tau) = \left< k \right>^{-s} \left< \tau -  \phi^{\beta_{2}, \gamma_{2}}(k)\right>^{1-b}\widetilde{g}(k, \tau).
$$
The following claim shows that the function $H(k, \tau  )$ is bounded.
\vspace{0.2cm}

\noindent\textbf{Claim:} For some constant $C_{1} > 0$, which depends only on $\beta_{1}, \beta_{2}$, we have that $$\sup_{(k , \tau) \in \mathbb{Z} \times \mathbb{R}}|H(k, \tau  )| \leq C_{1}.$$ 

\vspace{0.2cm}

In fact, if $|k| \leq 1$ is immediate. If $|k|>1$ note that
\begin{equation*}
\left< \tau  - \phi^{\beta_{1}, \gamma_{1}}(k)\right> \left< \tau  - \phi^{\beta_{2}, \gamma_{2}}(k)\right> \geq  \left| \phi^{\beta_{1}, \gamma_{1}}(k) - \phi_{\beta_{2}, \gamma_{2}(k)} \right|= | ( \beta_{1} - \beta_{2}) k^{3} - (\gamma_{1} - \gamma_{2}) k| .
\end{equation*}
Since $\beta_{1} \neq \beta_{2}$ we can choose $\epsilon \ll 1$ such that $|\gamma_{1}| +|\gamma_{2}| \leq  \epsilon $ and  consequently
$$|\gamma_{1} - \gamma_{2}||k| \leq \frac{1}{2}|\beta_{1} - \beta_{2}||k|^{3}.$$ So, using this previous inequality, yields that
$$ \left< \tau  - \phi^{\beta_{1}, \gamma_{1}}(k)\right> \left< \tau  - \phi^{\beta_{2}, \gamma_{2}}(k)\right>\geq  |\beta_{1} - \beta_{2}||k|^{3} - |\gamma_{1} - \gamma_{2}||k| \geq \frac{1}{2}|\beta_{1} - \beta_{2}||k|^{3}.
$$
Thus, we obtain
\begin{equation*}
\left| H(k, \tau) \right| \leq \frac{C_{1}(\beta_{1}, \beta_{2})|k|}{|k|^{3(b^{-})}\left< \tau  - \phi^{\beta_{2}, \gamma_{2}}(k)\right>^{(1-2b)^{+}}}\leq  \frac{C_{1}(\beta_{1}, \beta_{2})}{|k|^{3(b^{-})-1}} \leq  C_{1}(\beta_{1}, \beta_{2})
\end{equation*}
where we use the fact the $b\in(\frac{1}{3},\frac{1}{2}]$ in the second and  third inequality, respectively. This ends the proof of the claim.
\vspace{0.2cm}

With this in hand, we infer that 
\begin{equation*}
\begin{split}
\|\partial_{x}w\|_{X_{s, b-1}^{\beta_{2}, \gamma_{2}
}}  \leq & \ C_{1} \sup_{\|g\|_{X_{-s,1- b}^{\beta_{2}, \gamma_{2}}} \leq 1}  \sum_{k \in \mathbb{Z}} \int_{\mathbb{R}} \left| \widetilde{W}(k, \tau) \right|\left| \widetilde{G}(k, \tau)\right|d \tau\\ \leq  &\ C_{1} \sup_{\|g\|_{X_{-s, 1- b}^{\beta_{2}, \gamma_{2}}} \leq 1}    \left\|  \widetilde{W} \right\|_{\ell^{2}(\mathbb{Z})L^{2}(\mathbb{R})} \left\|\widetilde{G} \right\|_{\ell^{2}(\mathbb{Z})L^{2}(\mathbb{R})}\\ \leq &    \ C_{1} \left\|w \right\|_{X_{s, b^{-}}^{\beta_{1}, \gamma_{1}}}.
\end{split}
\end{equation*}
Consequently, for $ \frac{1}{3} < b  \leq \frac{1}{2}$, there exists $\theta > 0$ such that
$$
\|\partial_{x} w \|_{X_{s, b-1}^{\beta_{2}, \gamma_{2} ,T }} \leq C_{1} \|w\|_{X_{s,  b^{ -}}^{\beta_{1}, \gamma_{1} ,T }} \leq C_{1}T^{\theta}\|w\|_{X_{s, b }^{\beta_{1}, \gamma_{1} ,T }},
$$
reaching estimate \eqref{norm1}.

Now, to prove the second inequality in \eqref{norm2}, note that by duality we have
\begin{equation*}
\begin{split}
\|\partial_{x} w \|_{Y_{s, b - \frac{3}{2}}^{\beta_{2}, \gamma_{2}}} =& \sup_{\begin{array}{ccc}  \|a_{k}\|_{ \ell^{2}(\mathbb{Z})}  \\ a_{k} \neq 0 \end{array}}  \sum_{k \in \mathbb{Z}}  a_{k} \int_{\mathbb{R}} \frac{ik\left< k \right>^{s} | \widetilde{w}(k, \tau)|}{\left<  \tau - \phi_{\beta_{2}, \gamma_{2}}(k)\right>^{ \frac{3}{2} - b}} d\tau \\
 \leq &  \sup_{\begin{array}{ccc}  \|a_{k}\|_{ \ell^{2}(\mathbb{Z})}  \\ a_{k} \neq 0 \end{array}}   \sum_{k \in \mathbb{Z}} \int_{\mathbb{R}} H(k , \tau) \frac{a_{k}}{\left<  \tau - \phi^{\beta_{2}, \gamma_{2}}(k)\right>^{\frac{1}{2}^{+} }}  \widetilde{W}(k, \tau) d\tau \\
%\leq &  \sup_{\begin{array}{ccc}  \|a_{k}\|_{ \ell^{2}(\mathbb{Z})}  \\ a_{k} \neq 0 \end{array}}   \sum_{k \in \mathbb{Z}} \int_{\mathbb{R}} H(k , \tau)|a_{k}| \left| \widetilde{W}(k, \tau)\right|,
\end{split}
\end{equation*}
where
$$
H(k , \tau) = \frac{|k|}{\left<  \tau - \phi^{\beta_{1}, \gamma_{1}}(k)\right>^{b^{-}} \left<  \tau - \phi^{\beta_{2}, \gamma_{2}}(k)\right>^{(1-b)^{-}}}
$$
and
$$
\widetilde{W}(k, \tau) = \left< k \right>^{s} \left<  \tau - \phi^{\beta_{1}, \gamma_{1}}(k)\right>^{b^{-}}|\widetilde{w}(k, \tau)|.
$$
The claim proved before give us 
\begin{equation*}
\begin{split}
\|\partial_{x} w \|_{Y_{s, -b - \frac{1}{2}}^{\beta_{2}, \gamma_{2}}} \leq&  C_{1} \sup_{\|a_{k}\|_{\ell^{2}(\mathbb{Z})} \leq 1}   \|w\|_{X_{s,  b^{-} }^{\beta_{1}, \gamma_{1}}} \left\|  \frac{|a_{k}|}{\left<  \tau - \phi^{\beta_{2}, \gamma_{2}}(k)\right>^{\frac{1}{2}+}}\right \|_{\ell^{2}(\mathbb{Z})L^{2}(\mathbb{R})} \\
\leq &  C_{1} \sup_{\|a_{k}\|_{\ell^{2}(\mathbb{Z})} \leq 1}   \|w\|_{X_{s,  b^{-} }^{\beta_{1}, \gamma_{1}}}\|a_{k}\|_{\ell^{2}(\mathbb{Z})} \leq  C_{1} \|w\|_{X_{s,  b^{-} }^{\beta_{1}, \gamma_{1}}},
\end{split}
\end{equation*}
showing the second estimate of \eqref{norm2}, and consequently, Lemma \ref{derivativeestimate} is proved.
\end{proof}

The next lemma was borrowed from \cite[Lemmas 4.1 and 4.2]{Zhang} and concerns with the \textit{bilinear estimates} in Bourgain spaces for the term $\partial_{x}(uv)$ when the functions $u$ and $v$ belong in $X_{s, b}^{\beta_{i}, \gamma_{i}}$ for $\beta_{i}$ and $\gamma_{i}$, $i=1$ and $2$, distinct. In fact, the authors in \cite{Zhang} showed the result for general cases on domain $\mathbb{T}_{\lambda}\times \mathbb{R}$, for $\lambda\geq1$. Here, we will revisit the result proving in a simpler way the \textit{bilinear estimates} on $\mathbb{T} \times [0,T]$, which will be used for obtaining our future results.
 
\begin{lemma}\label{bilinearestlemma}
Let $s \geq 0$, $T \in (0,1)$ and $\beta_{1}, \beta_{2} \in \mathbb{R}^{*}$, with $\beta_{1} \neq \beta_{2}$\footnote{The cases $\beta_{1} = \beta_{2}$ and $\gamma_{1} =\gamma_{2}$ is known be true and can be seen in \cite[Proposition 5]{collianderetal2003}.}. Also consider that  $\frac{\beta_{1}}{\beta_{2}} < \frac{1}{4}$. Let $u$ and $v$ functions such that  with $[u]=[v]=0$. There exist constants $\theta > 0$, $\epsilon = \epsilon (\beta_{1}, \beta_{2})>0$ and $C_{2}= C_{2}(\beta_{1}, \beta_{2}) >0$, independent of $T$, $u$ and $v$, such that if $|\gamma_{1}| + |\gamma_{2}| < \epsilon $, we have:
\begin{itemize}
\item[a)] 
If  $u \in    X^{\beta_{1}, \gamma_{1},T}_{s, \frac{1}{2}}$ and  $v \in    X^{\beta_{2}, \gamma_{2},T}_{s, \frac{1}{2}} $, then 
\begin{eqnarray}\label{bilinearestimate}
\|\partial_{x}(uv)\|_{Z^{\beta_{2}, \gamma_{2},T}_{s, - \frac{1}{2}}} \leq C_{2}T^{\theta}\|u\|_{X^{\beta_{1}, \gamma_{1},T}_{s,  \frac{1}{2}}}\|v\|_{X^{\beta_{2} , \gamma_{2} ,T}_{s, \frac{1}{2}}}.
\end{eqnarray}
\item[b)] If $u, v \in X^{\beta_{2}, \gamma_{2},T}_{s, \frac{1}{2}}$ then 
\begin{eqnarray}\label{bilinearestimate2}
\|\partial_{x}(uv)\|_{Z^{\beta_{1}, \gamma_{1},T}_{s, - \frac{1}{2}}} \leq C_{2}T^{\theta}\|u\|_{X^{\beta_{2}, \gamma_{2},T}_{s,  \frac{1}{2}}}\|v\|_{X^{\beta_{2} , \gamma_{2} ,T}_{s, \frac{1}{2}}}.
\end{eqnarray}
\end{itemize}
\end{lemma}
\begin{proof}
We will prove the estimate \eqref{bilinearestimate}. The proof of \eqref{bilinearestimate2} is shown similarly and we omit its demonstration. Let $u\in X_{s, \frac{1}{2}}^{\beta_{1}, \gamma_{1}}$ and $v\in X_{s, \frac{1}{2}}^{\beta_{2}, \gamma_{2}}$, with $[u]=[v]=0$. Necessarily, we must to prove 
\begin{equation}\label{bilinear3}
\|\partial_{x}(uv)\|_{X_{s, - \frac{1}{2}}^{\beta_{2}, \gamma_{2}}} \leq C_{2}T^{\theta} \|u\|_{X_{s, \frac{1}{2}}^{\beta_{1}, \gamma_{1}}}\|v\|_{X_{s, \frac{1}{2}}^{\beta_{2}, \gamma_{2}}}
\end{equation} 
and 
\begin{equation}\label{bilinear3a}
 \|\partial_{x}(uv)\|_{Y_{s, - 1}^{\beta_{2}, \gamma_{2}}}  \leq C_{2}T^{\theta} \|u\|_{X_{s, \frac{1}{2}}^{\beta_{1}, \gamma_{1}}}\|v\|_{X_{s, \frac{1}{2}}^{\beta_{2}, \gamma_{2}}},
\end{equation} 
for some $\theta >0$. 

Firstly, we will show \eqref{bilinear3}. For this end from Plancherel theorem, duality approach and convolution properties, yields that
\begin{equation}\label{bilinear4}
\begin{split}
\displaystyle \|\partial_{x}(uv)\|_{X_{s, -\frac{1}{2}}^{\beta_{2}, \gamma_{2}}}= & \displaystyle \sup_{\|g\|_{X_{-s, \frac{1}{2}}^{\beta_{2}, \gamma_{2}}} \leq 1}\sum_{ \Gamma} \int_{ \Lambda} \frac{|k_{3}| \left< k_{3}\right>^{s}}{\left< k_{1}\right>^{s}\left< k_{2}\right>^{s}}\prod_{i=1}^{3}\frac{|f_{i}(k_{i}, \tau_{i})|}{\left< L_{i}(k_{i}, \tau_{i})\right>^{\frac{1}{2}}} d\Lambda.
\end{split}
\end{equation}
Here $L_{i}(k_{i}, \tau_{i}) = \tau_{i} - \phi^{\beta_{i}, \gamma_{i}}(k_{i})$, for $i=1,2$, $L_{3}(k_{3}, \tau_{3}) = \tau_{3} - \phi^{\beta_{2}, \gamma_{2}}(k_{3})$, $\Gamma$ and $\Lambda$ given by
\begin{equation*}
\Gamma := \left\{ (k_{1}, k_{2}, k_{3}) \in \mathbb{Z}^{3} \ \ : \ \ \sum_{i=1}^{3} k_{i} = 0\right\} \ \ \mbox{and} \ \ \Lambda := \left\{ (k_{1}, k_{2}, k_{3}) \in \mathbb{R}^{3} \ \ : \ \ \sum_{i=1}^{3} \tau_{i} = 0\right\},
\end{equation*}
respectively, with
$$f_{1}(k_{1}, \tau_{1}) = \left< k_{1}\right>^{s} \left< \tau_{1} - \phi^{\beta_{1}, \gamma_{1}}(k_{1})\right>^{\frac{1}{2}} \widetilde{u}(k_{1}, \tau_{1}),$$ $$f_{2}(k_{2}, \tau_{2}) = \left< k_{2}\right>^{s} \left< \tau_{2} - \phi^{\beta_{2}, \gamma_{2}}(k_{2})\right>^{\frac{1}{2}} \widetilde{v}(k_{2}, \tau_{2}) $$ and $$f_{3}(k_{3}, \tau_{3}) = \left< k_{3}\right>^{-s} \left< \tau_{3} - \phi^{\beta_{2}, \gamma_{2}}(k_{3})\right>^{\frac{1}{2}} \widetilde{g}(k_{3}, \tau_{3}).$$
The condition $[u]=[v]=0$ together with the fact that $(k_{1}, k_{2}, k_{3}) \in \mathbb{Z}^{3}$ ensure us that we only need to consider the case $|k_{i}| \geq 1$ for $i=1,2,3$. In addition, as $(k_{1}, k_{2}, k_{3}) \in \Gamma$, we have $\frac{\left< k_{3}\right> }{\left<k_{1}\right>  \left< k_{2} \right> }\leq 1 $. Thus,
$$
 \frac{|k_{3}| \left< k_{3}\right>^{s}}{\left< k_{1}\right>^{s}\left< k_{2}\right>^{s}} \leq (|k_{1}||k_{2}||k_{3}|)^{\frac{1}{2}},
$$
for all $s \geq 0$. Define
\begin{equation}\label{functionH}
H(k_{1}, k_{2}, k_{3}) := \phi^{\beta_{1}, \gamma_{1}}(k_{1})+ \phi^{\beta_{2}, \gamma_{2}}(k_{2}) + \phi^{\beta_{2}, \gamma_{2}}(k_{3}).
\end{equation}

\vspace{0.2cm}

\noindent\textbf{Claim.}
Let $\frac{\beta_{1}}{\beta_{2}} < \frac{1}{4}$. There exist $\epsilon = \epsilon (\beta_{1}, \beta_{2})$ and $\delta > 0$ such that if $|\mu| + |\zeta| < \epsilon$ then the function $H$ defined in \eqref{functionH} is $\delta$-significant on $\mathbb{Z}$, i.e.,
\begin{equation*}
\left< H(k_{1}, k_{2}, k_{3}) \right> \geq \delta \prod_{i=1}^{3} |k_{i}| \ \ \mbox{for any} \ \ (k_{1}, k_{2}, k_{3}) \in \Gamma.
\end{equation*}

Assume that the claim holds true. Thus,
\begin{equation*}
\prod_{i=1}^{3} |k_{i}| \lesssim \left< H(k_{1}, k_{2}, k_{3} \right> = \left< \sum_{i=1}^{3} L_{i}(k_{i}, \tau_{i})\right>.
\end{equation*}
Using the last inequality we obtain
\begin{equation}\label{bilinear5}
\begin{split}
\sum_{ \Gamma} \int_{ \Lambda} \frac{|k_{3}| \left< k_{3}\right>^{s}}{\left< k_{1}\right>^{s}\left< k_{2}\right>^{s}}\prod_{i=1}^{3}\frac{|f_{i}(k_{i}, \tau_{i})|}{\left< L_{i}(k_{i}, \tau_{i})\right>^{\frac{1}{2}}} d\Lambda \lesssim &  \sum_{ \Gamma} \int_{ \Lambda} \left<  \sum_{j=1}^{3} L_{i}(k_{i}, \tau_{i}) \right>^{\frac{1}{2}}\prod_{i=1}^{3}\frac{f_{i}(k_{i}, \tau_{i})}{\left< L_{i}(k_{i}, \tau_{i})\right>^{\frac{1}{2}}} d\Lambda \\
\lesssim &  \sum_{j=1}^{3}\sum_{ \Gamma} \int_{ \Lambda}\frac{\left< L_{j}(k_{j}, \tau_{j}) \right>^{\frac{1}{2}}\prod_{i=1}^{3} f_{i}(k_{i}, \tau_{i})}{\left< L_{i}(k_{i}, \tau_{i})\right>^{\frac{1}{2}}} d\Lambda .
\end{split} 
\end{equation}
Let us estimate each term of the right hand side of \eqref{bilinear5}. For simplicity, we will present the estimate corresponding to $j=1$. The other terms will be estimated in the similar way. For this case, we have
\begin{equation*}
\begin{split}
\sum_{ \Gamma} \int_{ \Lambda} &|f_{1}(k_{1}, \tau_{1})| \frac{|f_{2}(k_{2}, \tau_{2})|}{\left< L_{2}(k_{2}, \tau_{2})\right>^{\frac{1}{2}}}\frac{|f_{3}(k_{3}, \tau_{2})|}{\left< L_{2}(k_{3}, \tau_{3})\right>^{\frac{1}{2}}} = \sum_{ \Gamma} \int_{ \Lambda} \widetilde{g_{1}}(k_{1}, \tau_{1}) \widetilde{g_{2}}(k_{2}, \tau_{2})  \widetilde{g_{3}}(k_{3}, \tau_{3})
\end{split}
\end{equation*}
with
\begin{equation*}
 \widetilde{g_{1}} = | f_{1}(k_{1}, \tau_{1})| \ \mbox{ and } \  \widetilde{g_{i}} =  \frac{|f_{i}(k_{i}, \tau_{i})|}{\left< L_{i}(k_{i}, \tau_{i})\right>^{\frac{1}{2}}} \  \mbox{ for } \ i=2,3.
\end{equation*}
Thus, 
\begin{equation*}
\begin{split}
\sum_{ \Gamma} \int_{ \Lambda} |f_{1}(k_{1}, \tau_{1})| \frac{|f_{2}(k_{2}, \tau_{2})|}{\left< L_{2}(k_{2}, \tau_{2})\right>^{\frac{1}{2}}}\frac{|f_{3}(k_{3}, \tau_{2})|}{\left< L_{2}(k_{3}, \tau_{3})\right>^{\frac{1}{2}}} \lesssim & \sum_{k_{1} \in \mathbb{Z}} \int_{\mathbb{R}}\widetilde{g_{1}}(-k_{1}, - \tau_{1}) \widetilde{g_{2}g_{3}}(k_{1}, \tau_{1} )d\tau_{1} \\
\lesssim & \|g_{1}\|_{\ell^{2}(\mathbb{Z})L^{2}(\mathbb{R})}\|g_{2}\|_{\ell^{4}(\mathbb{Z})L^{4}(\mathbb{R})}\|g_{3}\|_{\ell^{4}(\mathbb{Z})L^{4}(\mathbb{R})}\\
% \lesssim &  \|g_{1}\|_{\ell^{2}(\mathbb{Z})L^{2}(\mathbb{R})}\|g_{2}\|_{X_{0, \frac{1}{3}}^{\beta_{2}, \gamma_{2}}}\|g_{3}\|_{X_{0, \frac{1}{3}}^{\beta_{2}, \gamma_{2}}} \\
 \lesssim & \|u\|_{X_{s, \frac{1}{2}}^{\beta_{1}, \gamma_{1}}} \|v\|_{X_{s, \frac{1}{3}}^{\beta_{2}, \gamma_{2}}} \|g\|_{X_{-s, \frac{1}{3}}^{\beta_{2}, \gamma_{2}}},
\end{split}
\end{equation*}
where we have used that $X_{0,\frac{1}{3}}^{\beta, \gamma}$ is continuously imbedded in the space $\ell^{4}(\mathbb{Z})L^{4}(\mathbb{R})$\footnote{See \cite[Lemma 3.2]{laurent2010} or  \cite[Lemma 3.9]{Zhang}.}. Replacing the last inequality in \eqref{bilinear5} we conclude from \eqref{bilinear4} that
\begin{equation*}
\|\partial_{x}(uv)\|_{X_{s, -\frac{1}{2}}^{\beta_{2}, \gamma_{2}}} \lesssim  \left(\|u\|_{X_{s, \frac{1}{2}}^{\beta_{1}, \gamma_{1}}} \|v\|_{X_{s, \frac{1}{3}}^{\beta_{2}, \gamma_{2}}} + \|u\|_{X_{s, \frac{1}{3}}^{\beta_{1}, \gamma_{1}}} \|v\|_{X_{s, \frac{1}{2}}^{\beta_{2}, \gamma_{2}}} \right),
\end{equation*}
which implies that for any $T \in (0,1)$, there exists a positive constant $C_{2}$, independent of $T$, such that
\begin{equation*}
\|\partial_{x}(uv)\|_{X_{s, -\frac{1}{2}}^{\beta_{2}, \gamma_{2}, T}} \leq C_{2}T^{\frac{1}{6}}\|u\|_{X_{s, \frac{1}{2}}^{\beta_{1}, \gamma_{1}, T}} \|v\|_{X_{s, \frac{1}{2}}^{\beta_{2}, \gamma_{2}, T}},
\end{equation*}
showing \eqref{bilinear3}.

Before presenting the proof of \eqref{bilinear3a}, let us prove the claim. In fact, note that the function $H$ defined by \eqref{functionH} can be rewrite as
\begin{equation*}
H(k_{1}, k_{2}, k_{3}) = -3\beta_{2}k_{1}^{3}h\left(\frac{k_{2}}{k_{1}} \right) - (\gamma_{1} - \gamma_{2})k_{1},
\end{equation*}
where $h(x) = x^{2} + x + \frac{1}{3}(1 - \frac{\beta_{1}}{\beta_{2}})$. Since $\frac{\beta_{1}}{\beta_{2}} < \frac{1}{4}$, so $h$ does not have real roots. Thus, there exists $\delta_{1} >0$ such that $h(x) \geq  \delta_{1}(x^{2} + 1)$ for all $x \in \mathbb{R}$. In addition, we can take $\epsilon$ sufficiently enough such that $|\beta_{1} - \beta_{2}||k_{1}| \leq \frac{1}{2} + \delta_{1}|\beta_{2}||k_{1}|^{3}$ for any $k_{2} \in \mathbb{Z}^{*}$. Hence, for $k_{1}, k_{2}, k_{3} \in \mathbb{Z}^{*}$ satisfying $\sum_{i=1}^{3} k_{i} =0$ we have
\begin{equation*}
\begin{split}
\left< H(k_{1}, k_{2}, k_{3}) \right> \geq & 1+  3 |\beta_{2}||k_{1}|^{3}h\left(\frac{k_{2}}{k_{1}} \right) - |\gamma_{1} - \gamma_{2}||k_{1}| \geq \frac{1}{2} +  3 \delta_{1} |\beta_{2}||k_{1}||k_{1}|^{2}\left(\frac{k_{2}^{2}}{k_{1}^{2}} + 1 \right) - \delta_{1}|\beta_{2}||k_{1}|^{3} \\
\geq & \delta \left( 1 + |k_{1}| \sum_{i}^{3}|k_{i}|^{2}\right) \geq \delta \prod_{i=1}^{3}|k_{i}|,
\end{split}
\end{equation*}
where $\delta$ is a positive constant which depends on $\beta_{1}, \beta_{2}$ and the claim is verified.

\vspace{0.2cm}

To prove \eqref{bilinear3a} using Cauchy-Schwarz inequality and for arguments similar to the one used previously, follows that
\begin{equation}\label{bilinear6}
\|\partial_{x} (uv) \|_{Y_{s, -1}^{\beta_{2}, \gamma_{2}}} \leq I \times \sup_{\begin{array}{ccc}  \|a_{k}\|_{ \ell^{2}(\mathbb{Z})} \leq 1 \\ a_{k} \neq 0 \end{array}}  \left\|\frac{\chi_{\Omega(k_{3})}(L_{3}) a_{k_{3}}}{\left< L_{3} \right>^{1-a}} \right\|_{\ell^{2}_{k_{3}}(\mathbb{Z})L^{2}_{\tau_{3}}(\mathbb{R})} 
\end{equation}
with\begin{equation*}
I = \sup_{\|f_{3}\|_{\ell^{2}(\mathbb{Z})L^{2}(\mathbb{R})} \leq 1}\sum_{\Gamma}  \int_{\Lambda} \frac{|k_{3}|\left<k_{3}\right>^{s}}{\left<k_{2}\right>^{s} \left<k_{1}\right>^{s}} \frac{|f_{1}(k_{1}, \tau_{1})|}{\left< L_{1}(k_{1}, \tau_{1})\right>^{\frac{1}{2}}} \frac{|f_{2}(k_{2},  \tau_{2})|}{\left< L_{2}(k_{2},\tau_{2})\right>^{\frac{1}{2}}}\frac{|\widetilde{f_{3}}(k_{3}, \tau_{3})|}{\left< L_{3}(k_{3}, \tau_{3}) \right>^{a} }d\Lambda
\end{equation*}
and $f_{i}$, $L_{i}$, for $i=1,2,3$, defined as in \eqref{bilinear4}. Here the characteristic function $\chi_{\Omega(k_{3})}(L_{3})$ will be chosen so that 
\begin{equation*}
\left\| \frac{\chi_{\Omega(k_{3})}(L_{3}) }{\left< L_{3} \right>^{1-a}} \right\|_{L^{2}_{\tau_{3}}(\mathbb{R})} \lesssim 1 
\end{equation*}
uniformly in the parameter $k_{3}$,$|k_{i}| \geq 1$,  for $i=1,2,3$, and $a>0$ to be chosen conveniently. 

Now, define $ MAX:= \max \{\left< L_{1}(k_{1}, \tau_{1})\right>,\left< L_{2}(k_{2}, \tau_{2})\right>, \left< L_{3}(k_{3}, \tau_{3})\right>\}$. Since $H$ is $\delta$-significant, we have
\begin{equation*}
\prod_{i=1}^{3} |k_{i}| \lesssim H(k_{1}, k_{2}, k_{3}) = \sum_{i=1}^{3} L_{i}(k_{i}, \tau_{i}) \lesssim \sum_{i}^{3} \left<L_{i}(k_{i}, \tau_{i}) \right>  \lesssim  MAX = \left< L_{1}(k_{1}, \tau_{1})\right>.
\end{equation*}
The rest of the proof will be split in two cases. 

\vspace{0.2cm}

\noindent\textbf{Case 1:} $MAX = \left< L_{1}(k_{1}, \tau_{1})\right>$ or $MAX = \left< L_{2}(k_{2}, \tau_{2})\right>$.

\vspace{0.2cm}

Assume without loss of generality$MAX = \left< L_{1}(k_{1}, \tau_{1})\right>$.  Take $a = \frac{1}{2}^{-} = \frac{1}{2} - \epsilon'$ with $12 \epsilon' = \frac{1}{100}$. In this case, we can put $\Omega(k_{3}) = \mathbb{R}$. Thus, from \eqref{bilinear6} we have
\begin{equation*}
\|\partial_{x} (uv) \|_{Y_{s, -1}^{\beta_{2}, \gamma_{2}}} \lesssim  \sup_{\|f_{3}\|_{\ell^{2}(\mathbb{Z})L^{2}(\mathbb{R})} \leq 1}\sum_{\Gamma}  \int_{\Lambda} |f_{1}(k_{1}, \tau_{1})| \frac{|f_{2}(k_{2},  \tau_{2})|}{\left< L_{2}(k_{2},\tau_{2})\right>^{\frac{1}{2}}}\frac{|\widetilde{f_{3}}(k_{3}, \tau_{3})|}{\left< L_{3}(k_{3}, \tau_{3}) \right>^{\frac{1}{2}^{-}} }d\Lambda.
\end{equation*}
Similarly to $X_{s, -\frac{1}{2}}^{\beta_{2}, \gamma_{2}}$-norm, we obtain 
\begin{equation}\label{bilinear7}
\|\partial_{x} (uv) \|_{Y_{s, -1}^{\beta_{2}, \gamma_{2}}} \lesssim  \sup_{\|f_{3}\|_{\ell^{2}(\mathbb{Z})L^{2}(\mathbb{R})} \leq 1} \lesssim \|u\|_{X_{s, \frac{1}{2}}^{\beta_{1}, \gamma_{1}}}\|v\|_{X_{s, \frac{1}{3}}^{\beta_{2}, \gamma_{2}}}\|f_{3}\|_{X_{0, -\frac{1}{6}^{+}}^{\beta_{2}, \gamma_{2}}} \lesssim \|u\|_{X_{s, \frac{1}{2}}^{\beta_{1}, \gamma_{1}}}\|v\|_{X_{s, \frac{1}{3}}^{\beta_{2}, \gamma_{2}}}.
\end{equation}
\vspace{0.2cm}

\noindent\textbf{Case 2:}  $MAX = \left< L_{3}(k_{3}, \tau_{3})\right>$

\vspace{0.2cm}

This case will be divided in two parts. 

\vspace{0.2cm}

\noindent\textbf{Part I. } $\left<L_{3}(k_{3}, \tau_{3}) \right>^{\frac{1}{100}} \leq \delta \left<L_{1}(k_{1}, \tau_{1}) \right> \left<L_{2}(k_{2},\tau_{2}) \right> $.

\vspace{0.2cm}

Observe that 
\begin{equation}\label{bilinear8}
\left( \prod_{i=1}^{3} |k_{i}| \right)^{\frac{1}{2}} \lesssim \left< L_{3}(k_{3}, \tau_{3})\right>^{\frac{1}{2}} =  \left< L_{3}(k_{3}, \tau_{3})\right>^{\frac{1}{2}-\epsilon'} \left< L_{1}(k_{1}, \tau_{1})\right>^{\frac{1}{12}} \left< L_{2}(k_{2}, \tau_{2})\right>^{\frac{1}{12}}.
\end{equation}
Therefore, choosing $a$ and $\Omega(k_{3})$ as in the Case 1,  from \eqref{bilinear6} results
\begin{equation*}
\begin{split}
\|\partial_{x} (uv) \|_{Y_{s, -1}^{\beta_{2}, \gamma_{2}}} \lesssim & \sup_{\|f_{3}\|_{\ell^{2}(\mathbb{Z})L^{2}(\mathbb{R})} \leq 1}\sum_{\Gamma}  \int_{\Lambda}\frac{|f_{1}(k_{1}, \tau_{1})|}{\left< L_{1}(k_{1}, \tau_{1})\right>^{\frac{5}{12}}} \frac{|f_{2}(k_{2},  \tau_{2})|}{\left< L_{2}(k_{2},\tau_{2})\right>^{\frac{1}{12}}}|\widetilde{f_{3}}(k_{3}, \tau_{3})|d\Lambda\\
\lesssim & \|u\|_{X_{s, \frac{5}{12}}^{\beta_{1}, \gamma_{1}}} \|v\|_{X_{s, \frac{5}{12}}^{\beta_{2}, \gamma_{2}}}.
\end{split}
\end{equation*}

\vspace{0.2cm}

\noindent\textbf{Part II. }  $\left< L_{1}(k_{1}, \tau_{1})\right> \left<L_{2}(k_{2}, \tau_{2}) \right> \ll \delta \left< L_{3}(k_{3}, \tau_{3})\right>^{\frac{1}{100}}$

\vspace{0.2cm}

Note that
$$
\left< L_{3}(k_{3}, \tau_{3}) + H(k_{1}, k_{2}, k_{3}) \right> = \left< L_{1} + L_{2} \right> \ll \delta \left< L_{3} \right>^{\frac{1}{100}}.
$$
Thus, $|H(k_{1}, k_{2}, k_{3})| \sim |L_{3}(k_{3}, \tau_{3})|$ and
\begin{equation}\label{setomega2}
\left< L_{3}(k_{3}, \tau_{3}) + H(k_{1}, k_{2}, k_{3}) \right>  \ll \delta \left< H(k_{1}, k_{2}, k_{3}) \right>^{\frac{1}{100}}.
\end{equation} 
Define for any $k_{3} \in \mathbb{Z}^{*}$ the set
\begin{equation}\label{setomega}
\begin{split}
\Omega^{\delta}(k_{3}) := \left\{ \tau \in \mathbb{R} \ : \ \exists \, k_{1}, k_{2}, \in \mathbb{Z} \ \ \mbox{such that} \ \ \sum_{i=1}^{3} k_{i} = 0 \ \ \mbox{and} \right.\\
\left. \left< L_{3}(k_{3}, \tau_{3}) + H(k_{1}, k_{2}, k_{3}) \right>  \ll \delta \left< H(k_{1}, k_{2}, k_{3}) \right>^{\frac{1}{100}} \right\}.
\end{split}
\end{equation}
From \eqref{setomega2} follows that $L_{3}(k_{3}, \tau_{3}) \in \Omega^{\delta}(k_{3})$.  Taking $a= \frac{1}{2}$ and $\Omega(k_{3}) =  \Omega^{\delta}(k_{3})$ defined by \eqref{setomega} we have\footnote{This prove is extremely technical  and can be found in \cite[Lemma 6.2 case 5.2.2. and Lemma 6.3]{Zhang} which in turn was inspired by \cite[Lemma 7.4]{collianderetal2003}}
\begin{equation*}
\left\|\frac{\chi_{\Omega(k_{3})}(L_{3})}{\left< L_{3}\right>^{\frac{1}{2}}} \right\|_{L^{2}_{\tau_{3}(\mathbb{R})}} \lesssim 1
\end{equation*}
uniformly in $k_{3}$ implying 
\begin{equation*}
 \sup_{\begin{array}{ccc}  \|a_{k}\|_{ \ell^{2}(\mathbb{Z})} \leq 1 \\ a_{k} \neq 0 \end{array}}  \left\|\frac{a_{k_{3}}}{\left< L_{3} \right>^{\frac{1}{2}}} \right\|_{\ell^{2}_{k_{3}}(\mathbb{Z})L^{2}_{\tau_{3}}(\mathbb{R})}  \lesssim 1.
\end{equation*}
Hence, using the first inequality in \eqref{bilinear8} we obtain
\begin{equation*}
\begin{split}
\|\partial_{x} (uv) \|_{Y_{s, -1}^{\beta_{2}, \gamma_{2}}} \lesssim & \sup_{\|f_{3}\|_{\ell^{2}(\mathbb{Z})L^{2}(\mathbb{R})} \leq 1}\sum_{\Gamma}  \int_{\Lambda}  \frac{|f_{1}(k_{1}, \tau_{1})|}{\left< L_{1}(k_{1}, \tau_{1})\right>^{\frac{1}{2}}} \frac{|f_{2}(k_{2},  \tau_{2})|}{\left< L_{2}(k_{2},\tau_{2})\right>^{\frac{1}{2}}}|\widetilde{f_{3}}(k_{3}, \tau_{3})|
\lesssim  \|u\|_{X_{s, \frac{1}{3}}^{\beta_{1}, \gamma_{1}}}\|v\|_{X_{s, \frac{1}{3}}^{\beta_{2}, \gamma_{2}}}.
\end{split}
\end{equation*}
Then, in both situation we obtain
\begin{equation*}
\|\partial_{x}(uv)\|_{Y_{s, -1}^{\beta_{2}, \gamma_{2}}} \lesssim T^{\frac{1}{12}}\|u\|_{X_{s, \frac{1}{2}}^{\beta_{1}, \gamma_{1} , T }}\|v\|_{X_{s, \frac{1}{2}}^{\beta_{2}, \gamma_{2} , T }},
\end{equation*}
showing \eqref{bilinear3a} and, consequently, finishing the demonstration of the lemma. 
\end{proof}

Finally, to finish this section we will prove nonlinear estimates associated with the solutions of \eqref{systema} with $p=q=0$. To do it, we introduce the following notation $$\mathcal{Z}_i:=Z_{s, -\frac{1}{2}}^{\beta_{i}, \gamma_{i},T},$$ for $i=1,2$ and $$\mathcal{Z}:=Z_{s, -\frac{1}{2}}^{\beta_{1}, \gamma_{1},T} \times Z_{s, - \frac{1}{2}}^{\beta_{2}, \gamma_{2},T}.$$

\begin{lemma}\label{contralbourgainestimate2}
Let $(u,v)$ and $ (w,z)$ belong to $ \mathcal{Z}$ with $[u]=[v]=0$. Consider $s, \beta_{i}, \gamma_{i}$, $i=1,2$, as in Lemma \ref{bilinearestlemma} satisfying $\frac{\beta_{2}}{\beta_{1}}<0  $, $P $ and $Q$ defined by \eqref{nonlinearitya}. Then, there exist constants $\theta > 0$, $\epsilon= \epsilon(\beta_{1}, \beta_{2})>0$ and $C_{3}=C_{3}(\beta_{1}, \beta_{2})>0$, independent of $T$, $u$ and $v$, such that if $|\gamma_{1}| + |\gamma_{2}| < \epsilon$, the following estimates are satisfied
\begin{equation}\label{Znonlinearestimate}
\|\partial_{x}(P(u,v), Q(u,v))\|_{\mathcal{Z}} \leq C_{3}T^{\theta}\|(u,v)\|_{X_{s, \frac{1}{2}}^{\beta_{1} , \gamma_{1},T} \times X_{s, \frac{1}{2}}^{\beta_{2}, \gamma_{2},T}}^{2},
\end{equation}
\begin{equation}\label{Znonlinearestimate2}
\|\partial_{x}( P(u, v) - P(w, z))\|_{\mathcal{Z}_1} \leq   C_{3}T^{\theta}  \|(u,v) -(w,z)\|_{\mathcal{Z}}\left( \|(u,v)\|_{\mathcal{Z}} + \|(w,z)\|_{\mathcal{Z}} \right)
\end{equation}
and
\begin{equation}\label{Znonlinearestimate3}
\|\partial_{x}( Q(u, v) - Q(w, z))\|_{\mathcal{Z}_2} \leq  C_{3}T^{\theta}  \|(u,v) -(w,z)\|_{\mathcal{Z}} \left( \|(u,v)\|_{\mathcal{Z}} + \|(w,z)\|_{\mathcal{Z}} \right).
\end{equation}
\end{lemma}
\begin{proof}
First of all, \eqref{Znonlinearestimate} is a direct consequence of Lemma \ref{bilinearestlemma}. Just applying \eqref{bilinearestimate} for $P$ and $Q$ provides us the existence of positive constants $C_{3}$, $ \theta $ and $\epsilon$ such that 
\begin{equation*}
\begin{split}
\|\partial_{x}(P(u,v), Q(u,v))\|_{\mathcal{Z}} \leq &  \ C_{2} T^{\theta} \left( |A|\|u\|_{\mathcal{Z}_1} +|B| \|u\|_{\mathcal{Z}_1}\|v\|_{\mathcal{Z}_2}+ \frac{|C|}{2}\|v\|_{\mathcal{Z}_2} + |D|\|v\|_{\mathcal{Z}_2}\right. \\
& \left. + |C|\|u\|_{\mathcal{Z}_1}\|v\|_{\mathcal{Z}_2} + \frac{|B|}{2}\|u\|_{\mathcal{Z}_1} \right) \\
\leq & \ C_{3}T^{\theta}\|(u,v)\|_{X_{s, \frac{1}{2}}^{\beta_{1} , \gamma_{1},T} \times X_{s, \frac{1}{2}}^{\beta_{2}, \gamma_{2},T}},
\end{split}
\end{equation*}
whenever $|\gamma_{1}| + |\gamma_{2}| < \epsilon$, where $C_{3} = C_{2} \cdot 2 \max \{ |A|, |B|, |C|, |D| \}$.

Let us now prove \eqref{Znonlinearestimate2}. As the proof of \eqref{Znonlinearestimate3} is analogous we will omit it.  Note that we can write
\begin{equation*}
P(u,v) - P(w,z) = A (u-w)(u + w) + B (u-w)v + B(v- z)w + \frac{C}{2}(v - z)(v + z)
\end{equation*}
Thus, gain by \eqref{bilinearestimate}, we get that
\begin{equation*}
\begin{split}
\|\partial_{x}( P(u, v) - P(w, z))\|_{\mathcal{Z}_1} \leq &  C_{3}T^{\theta} \left( \|(u-w)\|_{\mathcal{Z}_1} \|u+w\|_{\mathcal{Z}_1} + \|u - w\|_{\mathcal{Z}_1}\|v\|_{\mathcal{Z}_2}\right.\\& + \|v-z\|_{\mathcal{Z}_2}\|w\|_{\mathcal{Z}_1}+ \left. \|v-z\|_{\mathcal{Z}_2} + \|v+z\|_{\mathcal{Z}_2} \right),
\end{split}
\end{equation*}
which implies
\begin{equation*}
\begin{split}
\|\partial_{x}( P(u, v) - P(w, z))\|_{\mathcal{Z}_1}  \leq  &  C_{3}T^{\theta}  \|(u,v) -(w-z)\|_{\mathcal{Z}}\left( \|(u,v)\|_{\mathcal{Z}} + \|(w,z)\|_{\mathcal{Z}} \right).
\end{split}
\end{equation*}
Therefore, \eqref{Znonlinearestimate2} and \eqref{Znonlinearestimate3} is verified and the proof of the lemma is complete.
\end{proof}

\subsection{Local well-posedness}  Throughout the article, from now on, we will consider the following notations
$$\mathcal{Z}_{s,b}:=Z^{1, \mu , T}_{s,b} \times Z^{\alpha, \zeta, T}_{s,b}, \quad \mathcal{Z}^1_{s,b}:=Z^{1, \mu , T}_{s,b}, \quad \mathcal{Z}^{\alpha}_{s,b}:= Z^{\alpha, \zeta, T}_{s,b},$$
and
$$\mathcal{X}_{s,b}:=X^{1, \mu , T}_{s,b} \times X^{\alpha, \zeta, T}_{s,b}, \quad \mathcal{X}^1_{s,b}:=X^{1, \mu , T}_{s,b}, \quad \mathcal{X}^{\alpha}_{s,b}:= X^{\alpha, \zeta, T}_{s,b}.$$
Additionally, when $b=\frac{1}{2}$, we will denote 
$$\mathcal{Z}_{s}:=Z^{1, \mu , T}_{s,\frac{1}{2}} \times Z^{\alpha, \zeta, T}_{s,\frac{1}{2}}, \quad \mathcal{Z}^1_{s}:=Z^{1, \mu , T}_{s,\frac{1}{2}}, \quad \mathcal{Z}^{\alpha}_{s}:= Z^{\alpha, \zeta, T}_{s,\frac{1}{2}},$$
and
$$\mathcal{X}_{s}:=X^{1, \mu , T}_{s,\frac{1}{2}} \times X^{\alpha, \zeta, T}_{s,\frac{1}{2}}, \quad \mathcal{X}^1_{s}:=X^{1, \mu , T}_{s,\frac{1}{2}}, \quad \mathcal{X}^{\alpha}_{s}:= X^{\alpha, \zeta, T}_{s,\frac{1}{2}}.$$

 Let us now consider the IVP \eqref{systemcontrollocal_ns_int}.  For given $\lambda>0,$ let us define
$$
L_{\beta, \gamma, \lambda} \phi=\int_{0}^{1} e^{-2 \lambda \tau} S^{\beta, \gamma}(-\tau) G G^{*} S^{*}(-\tau) \phi d \tau,
$$
for any $\phi \in H^{s}(\mathbb{T})$ and  $s\geq0$. Clearly, $L_{\lambda}$ is a bounded linear operator from $H_{0}^{s}(\mathbb{T})$ to $H_{0}^{s}(\mathbb{T})$. Moreover, $L_{\beta, \gamma, \lambda}$ is a self-adjoint positive operator on $L_{0}^{2}(\mathbb{T})$, and so is its inverse $L_{\beta, \gamma , \lambda}^{-1}$. Therefore $L_{\beta, \gamma, \lambda}$ is an isomorphism from $L_{0}^2(\mathbb{T})$ onto itself, and the same is true on $H_{0}^{s}(\mathbb{T})$, with  $s\geq0$ (see, for instance, \cite[Lemma 2.4]{laurent2010}).

With these information in hand, choose the two feedback controls $f=-G^*L_{\beta, \gamma, \lambda}^{-1}u$ and $h=-G^*L_{\beta, \gamma, \lambda}^{-1}v,$
in \eqref{systemcontrollocal_ns_int}, to transform this system in a resulting closed-loop system reads as follows
\begin{equation}\label{PVIsystem}
\begin{cases}
\partial_{t} u + \partial_{x}^{3} u + \mu \partial_{x} u + \eta \partial_{x} v +   \partial_{x}P(u,v) = -K_{1, \mu , \lambda}u,&  \ x \in \mathbb{T},\ t \in \mathbb{R},\\
\partial_{t} v + \alpha \partial_{x}^{3} v + \zeta \partial_{x} u + \eta \partial_{x} u +  \partial_{x} Q(u,v) = -K_{\alpha, \zeta, \lambda}v,&  \ x \in \mathbb{T},\ t \in \mathbb{R},\\
(u(x,0), v(x,0)) = (u_{0}(x), v_{0}(x)), &  \ x \in \mathbb{T},
\end{cases}
\end{equation}
with $K_{\beta, \gamma , \lambda}:=GG^*L_{\beta, \gamma, \lambda}^{-1}$ for $\beta = 1$ and $\gamma = \mu$, and for $\beta=\alpha$ and $\gamma= \eta$.  If $\lambda=0$, we have $K_0=GG^{*}$.

\vspace{0.1cm}

We will prove that the IVP \eqref{PVIsystem}  is well-posed in the spaces $H_{0}^s(\mathbb{T}) \times H_{0}^s(\mathbb{T}) $, for $s\geq0$. To prove it, we will borrow the following lemma shown in \cite[Lemma 4.2]{laurent2010} for the case $\beta = 1$ and $\gamma = \mu > 0$. The proof in the general case, presented below, is similar to the one made there and will be omitted.

\begin{lemma}\label{controlbourgainestimate}
For any $\widetilde{\epsilon} > 0$ and $\phi \in Z_{s, \frac{1}{2}}^{\beta, \gamma, T} $ there exists a positive constant $C(\widetilde{\epsilon})>0$ such that
\begin{equation*}
\left\| \int_{0}^{t} S^{\beta , \gamma}(t- \tau) (K_{\lambda} \phi) (\tau) d \tau \right\|_{Z_{s, \frac{1}{2}}^{\beta, \gamma, T}} \leq C(\widetilde{\epsilon}) T^{1 -\widetilde{ \epsilon}}\|\phi\|_{Z_{s, \frac{1}{2}}^{\beta, \gamma , T}}.
\end{equation*}
\end{lemma}

The next local well-posedness result is a consequence of Lemmas \ref{bourgainkdv-estimates},  \ref{contralbourgainestimate2} and \ref{controlbourgainestimate}, and its proof is classical, so we will omit it.

\begin{theorem}\label{lwptheorem}
Let $\lambda \geq 0$ and $s \geq 0$ be given. Then, there exists $\epsilon = \epsilon(\alpha)$ with $|\mu| + |\zeta| < \epsilon$ such that for $T > 0 $, small enough, and any $(u_{0}, v_{0}) \in H_{0}^{s}(\mathbb{T}) \times H_{0}^{s}(\mathbb{T})$  there exists a unique solution $(u,v)$ of \eqref{PVIsystem} in the class 
\begin{equation}\label{solutionclass}
(u,v)\in \mathcal{X} := \mathcal{Z}^{1}_{s} \cap  C([0,T]; L_{0}^{2}(\mathbb{T}))   \times \mathcal{Z}^{\alpha}_s\cap C([0,T]; L_{0}^{2}(\mathbb{T})).
\end{equation}
Furthermore, the following estimate holds
\begin{equation}\label{wpestimate}
\|(u,v)\|_{\mathcal{Z}_s} \leq a_{T,s} (\|(u_{0}, v_{0})\|)\|(u_{0}, v_{0})\|_{H^{s}(\mathbb{T}) \times H^{s}(\mathbb{T})},
\end{equation}
where $a_{s,T}\, :  \mathbb{R}^{+} \longrightarrow \mathbb{R}^{+}$ is a nondecreasing continuous function depending only of $T$, $s$ and constants $\alpha, \mu, \zeta$. 

In addition, for any $T_{0} \in (0,T)$ there exists a neighborhood $U_{0}$ of $(u_{0}, v_{0})$ such that the application $(u_{0}, v_{0})\longmapsto (u,v)$ from $U_{0}$ into $\mathcal{X}$ is Lipschitz.
\end{theorem}

\subsection{Global well-posedness} We check that the system \eqref{PVIsystem} is globally well-posed in the space $H^s(\mathbb{T})$, for any $s\geq0$. Precisely, the result can be read as follows.
\begin{theorem}\label{gwptheorem}
Let  $(u_{0}, v_{0}) \in H_{0}^{s}(\mathbb{T}) \times H_{0}^{s}(\mathbb{T})$, for any $s \geq 0$. Then the solution $(u,v)\in \mathcal{X} $ given in Theorem \ref{lwptheorem} can be extended for any $T > 0$ and still satisfies \eqref{wpestimate}.
\end{theorem}
\begin{proof}
Assume first $s =0$. Multiplying the first equation of \eqref{PVIsystem} by $u$ and the second one by $v$, integrating on $\mathbb{T}\times(0,t)$, for $t\geq0$, we have
\begin{equation*}
\begin{split}
\|(u(\cdot, t), v(\cdot , t))\|^{2} \leq &\  2\|G\|^{2}\|\left[ L_{1, \mu, \lambda}^{-1}\|+ \|L_{\alpha, \zeta, \lambda}^{-1}\|\right] \int_{0}^{t} \|(u,v)(\cdot, \tau)\|^{2},
\end{split}
\end{equation*}
since $G$ and $L_{\beta , \gamma , \lambda}^{-1}$ are continuous in $L^{2}(\mathbb{T})$ and
\begin{equation}\label{identity4}
\begin{split}
\int_{\mathbb{T}}\partial_{x}P(u,v)u + \partial_{x}Q(u,v)v =  \frac{2}{3}  \int_{\mathbb{T}}\frac{d}{dx} \left[ (Au^{3} + Bv^{3}) + B u^{2}v + Cu v^{2} \right]=0.
\end{split}
\end{equation}
Using Gr\"onwall's inequality holds that
\begin{equation}\label{bounded2}
\|(u(\cdot, t), v(\cdot , t))\|^{2}   \leq \|(u_{0}, v_{0})\|^{2}e^{C_{5}t},
\end{equation}
with $C_{5} =  2\|G\|^{2}\left[\|L_{1, \mu , \lambda}^{-1}\| + \|L_{\alpha, \zeta , \lambda}^{-1}\|\right] $. In particular, for $\lambda =0$, from the energy identity, we get
\begin{equation}\label{nonincreasing}
\frac{1}{2}\frac{d}{dt}\|(u(\cdot, t), v(\cdot , t))\|^{2} = - \|(Gu, Gv)\|^{2} \leq 0
\end{equation}
and 
\begin{equation*}
\|(u(\cdot, t), v(\cdot , t))\|^{2} \leq \|(u_{0}, v_{0})\|^{2},
\end{equation*}
which ensures that \eqref{PVIsystem} is globally well-posed in $L_{0}^{2}(\mathbb{T}) \times L_{0}^{2}(\mathbb{T})$. 

Next, we show that \eqref{PVIsystem}  is globally well-posed in the space $H_{0}^{3}(\mathbb{T}) \times H_{0}^{3}(\mathbb{T})$. For a smooth solution $(u,v)$ of \eqref{PVIsystem}, let $(\widetilde{u}, \widetilde{v}) =  (\partial_{t}u, \partial_{t}v)$. Then
\begin{equation*}
\begin{cases}
\partial_{t} \widetilde{u} + \partial_{x}^{3} \widetilde u + \mu \partial_{x}\widetilde{u} + \eta \widetilde{v} + \partial_{x} \widetilde{P}(\widetilde{u}, \widetilde{v}) = - K_{\lambda}\widetilde{u},&  \ x \in \mathbb{T},\ t \in \mathbb{R},\\
\partial_{t} \widetilde{v} + \alpha \partial_{x}^{3} \widetilde v + \zeta \partial_{x}\widetilde{v} + \eta \widetilde{u} + \partial_{x} \widetilde{Q}(\widetilde{u}, \widetilde{v}) = - K_{\lambda}\widetilde{v},&  \ x \in \mathbb{T},\ t \in \mathbb{R},\\
(\widetilde{u}(x,0), \widetilde{v}(x,0)) = (\widetilde{u}_{0}, \widetilde{v}_{0}),&  \ x \in \mathbb{T},
\end{cases}
\end{equation*}
where
\begin{equation*}
\begin{split}
\widetilde{P}(\widetilde{u}, \widetilde{v}) &= 2A u \widetilde{u} + B \widetilde{u}v + B u \widetilde{v} + C v \widetilde{v},\\
\widetilde{Q}(\widetilde{u}, \widetilde{v}) &= 2D v \widetilde{v} + C \widetilde{v}u + C v \widetilde{u} + B u \widetilde{u},\\
\widetilde{u}_{0} &= - K_{\lambda}u_{0} - u_{0}''' - \mu u_{0}' - \eta v_{0}' - P'(u_{0}, v_{0}),\\
\widetilde{v}_{0} &= - K_{\lambda}v_{0} - v_{0}''' - \zeta v_{0}' - \eta u_{0}' - Q'(u_{0}, v_{0}),
\end{split}
\end{equation*}
with $" \ ' \ "$ denoting here the derivative with respect to variable $x$. Observe that
\begin{equation*}
\begin{split}
\|P'(u_{0}, v_{0}), Q'(u_{0}, v_{0}
 )\|  \leq & \frac{2C_{3}}{C_{2}}\|(u_{0}, v_{0})\|\|(u_{0}', v_{0}')\|_{L^{\infty}(\mathbb{T}) \times L^{\infty}(\mathbb{T})} \\
 \leq & \frac{2C_{3}}{C_{2}}\|(u_{0}, v_{0})\|C_{5}\|(u_{0}, v_{0})\|^{\frac{1}{2}}\|(\partial_{x}^{3}u_{0}, \partial_{x}^{3}v_{0})\|^{\frac{1}{2}}\\
 \leq & \frac{2C_{3}C_{5}}{C_{2}}\|(u_{0}, v_{0})\|^{\frac{3}{2}}\|(u_{0}, v_{0})\|_{3}^{\frac{1}{2}},
\end{split}
\end{equation*}
where $C_{5}$ is due to the Gagliardo-Nirenberg inequality. So, we have
\begin{equation}\label{boundedinitialdata}
\begin{split}
 \|(\widetilde{u}_{0}, \widetilde{v}_{0})\| \leq & \|(K_{\lambda}u_{0}, K_{\lambda}v_{0})\| +C_{6} \|(u_{0},v_{0})\|_{3} + 2( |\mu| + |\eta| + |\zeta|)\|(u_{0}, v_{0})\|_{1}\\&\ + \|P'(u_{0}, v_{0}), Q'(u_{0}, v_{0}
 )\|\\
  \leq & \|(u_{0}, v_{0})\| +  2( |\mu| + |\eta| + |\zeta|)\|(u_{0}, v_{0})\|_{1} + (1+C_{6})\|(u_{0}, v_{0})\|_{3} \\ & \ +  \left( \frac{C_{3}C_{5}}{C_{2}}\right)^{2} \|(u_{0}, v_{0})\|^{3}.
\end{split}
\end{equation}
where $C_{6} =  \max\{1, |\alpha|\}$. Since, 
\begin{equation*}
\begin{split}
 \|\partial_{x}( \widetilde{P}(\widetilde{u}, \widetilde{v}), \widetilde{Q}(\widetilde{u}, \widetilde{v}))\|_{\mathcal{Z}_s}   \leq & 2 C_{3}T^{\theta}\|(u,v)\|_{\mathcal{Z}_s}\|(\widetilde{u}, \widetilde{v})\|_{\mathcal{Z}_s},
\end{split}
\end{equation*}
with $|\mu| + |\zeta| < \epsilon $, for some $\epsilon$ that depends only $\alpha$, we obtain from \eqref{boundedinitialdata} that
\begin{equation*}
\begin{split}
\|(\widetilde{u}, \widetilde{v})\|_{\mathcal{Z}_0} \leq &    \ C_{0}\|(\widetilde{u}_{0}, \widetilde{v}_{0})\| + (C(\widetilde{\epsilon})T^{1-\widetilde{\epsilon}} + C_{1}T^{\theta})\|(\widetilde{u},\widetilde{v})\|_{\mathcal{Z}_0}  + 2C_{0} C_{3}T^{\theta} \|(u,v)\|_{\mathcal{Z}_0} \|(\widetilde{u},\widetilde{v})\|_{\mathcal{Z}_0}.
 \end{split}
\end{equation*}
Choosing $T$ as
\begin{equation*}
    C(\epsilon) T^{1- \widetilde{\epsilon}}+ C_{1}T^{\theta} + 2C_{0}C_{3} T^{\theta} d < \frac{1}{2},
\end{equation*} 
we have
\begin{equation*}
\|(u,v)\|_{\mathcal{Z}_0} \leq d = 2C_{0}\|(u_{0}, v_{0})\| 
\end{equation*}
and consequently
\begin{equation*}
\begin{split}
\|(\widetilde{u}, \widetilde{v})\|_{\mathcal{Z}_0} \leq     C_{0} \|(\widetilde{u}_{0}, \widetilde{v}_{0})\| + (C(\epsilon)T^{1-\epsilon} + C_{1}T^{\theta} + 4 C_{0}^{2}C_{3}T^{ \theta}  \|(u_{0},v_{0})\|) \|(\widetilde{u},\widetilde{v})\|_{\mathcal{Z}_0}. 
\end{split}
\end{equation*}
Hence, for $T_{1}$ satisfying
\begin{equation*}
C(\epsilon)T_{1}^{1-\epsilon} + C_{1}T^{\theta} + 4C_{0}^{2}C_{3}T_{1}^{\theta}  \|(u_{0},v_{0})\| < \frac{1}{2},
\end{equation*}
we obtain
\begin{equation*}
\|(\widetilde{u}, \widetilde{v})\|_{\mathcal{Z}_0} \leq  2C_{0} \|(\widetilde{u}_{0}, \widetilde{v}_{0})\|.
\end{equation*}
Therefore, for $T_{0} = \min \{ T, T_{1}\}$, we see that
\begin{equation*}
\|(\widetilde{u}, \widetilde{v})\|_{L^{\infty} (0,T_{0}; L^{2}(\mathbb{T}) \times L^{\infty} (0,T_{0}; L^{2}(\mathbb{T})} \leq C_{7} \|(\widetilde{u}, \widetilde{v})\|_{\mathcal{Z}_0} \leq 2C_{0}C_{7}\|(\widetilde{u}_{0}, \widetilde{v}_{0})\| .
\end{equation*}

From the following equations
\begin{equation*}
\begin{cases}
\partial_{x}^{3}u = - \widetilde{u} - \mu \partial_{x}u - \eta \partial_{x} v - \partial_{x}P(u,v) - K_{\lambda}u, \\
\partial_{x}^{3}v = - \widetilde{v} - \zeta \partial_{x}v - \eta \partial_{x} u - \partial_{x}Q(u,v) - K_{\lambda}v,
\end{cases}
\end{equation*}
we infer that
\begin{equation*}
\begin{split}
\|(\partial_{x}^{3} u, \partial_{x}^{3}v)\| \leq & \|(\widetilde{u}, \widetilde{v})\| +  \|(u,v)\| + \left( C_{7}\sqrt{2\pi} + \frac{2C_{3}}{C_{2}}\|(u,v)\| \right)   \|(\partial_{x}u, \partial_{x}v)\|_{L^{\infty}(\mathbb{T}) \times L^{\infty}(\mathbb{T})} \\
\leq &  \|(\widetilde{u}, \widetilde{v})\| +  \|(u,v)\| + C_{6}\left( C_{7}\sqrt{2\pi} + \frac{2C_{3}}{C_{2}}\|(u,v)\| \right)\|(u,v)\|^{\frac{1}{2}}\|(\partial_{x}^{3}u, \partial_{x}^{3}v)\|^{\frac{1}{2}}\\
 \leq &  \|(\widetilde{u}, \widetilde{v})\| + \frac{1}{2}\|(\partial_{x}^{3}u, \partial_{x}^{3}v)\| + \left[ 1 + \frac{C_{6}^{2}}{2} \left( C_{7}\sqrt{2\pi} + \frac{2C_{3}}{C_{2}}\|(u,v)\| \right)^{2}\right]\|(u,v)\|,
\end{split}
\end{equation*}
for $0 < t <  T_{0}$ and $C_{7} =2 (|\mu| + |\eta| + |\zeta|) $. Consequently, since $(\widetilde{u}_{0}, \widetilde{v}_{0})$ satisfies \eqref{boundedinitialdata} , we get that
\begin{equation*}
\|(u,v)\|_{L^{\infty}(0,T;H_{0}^{3}(\mathbb{T})) \times L^{\infty}(0,T;H_{0}^{3}(\mathbb{T})) } \leq a_{T,3}(\|(u_{0},v_{0})\|) \|(u_{0}, v_{0})\|_{3}.
\end{equation*}
Combining to \eqref{bounded2},  this shows that $(u,v) \in C(\mathbb{R}^{+}; H_{0}^{3}(\mathbb{T})) \times C(\mathbb{R}^{+}; H_{0}^{3}(\mathbb{T})) $ and \eqref{wpestimate} holds true for $s=3$. A similar result can be obtained for any $s \in 3 \mathbb{N}^{*}$. Note that for other values of $s$, the global well-posedness follows by nonlinear interpolation as done in \cite{BonaSunZhang2001}. This achieves the result.
\end{proof}

\section{Control results: Nonlinear problems}\label{Sec4} 
\subsection{A local result: Stabilization result} Our first result is local in the sense that the initial data need to be in a small ball in the energy space to ensure that the solution of the system  goes to zero exponentially, for $t$ sufficiently large. The result is the following one.
\begin{theorem}\label{localstability}
Let $0 < \lambda ' < \lambda$ and $s \geq 0$ be given. There exists $\delta > 0$ such that for any $(u_{0}, v_{0}) \in H_{0}^{s}(\mathbb{T}) \times  H_{0}^{s}(\mathbb{T})$ and $\|(u_{0}, v_{0})\|_{s} \leq \delta$, the corresponding solution $(u,v)$ of \eqref{PVIsystem} satisfies
\begin{equation*}
\|(u(\cdot , t), v(\cdot , t))\|_{s} \leq  C e^{- \lambda ' t} \|(u_{0}, v_{0})\|_{s}, \quad \forall t \geq 0,
\end{equation*}
where $C > 0$ is a constant independent of $(u_{0}, v_{0})$.
\end{theorem}
\begin{proof}
Since $K_{\lambda}$ is a bounded operator, the solution of \eqref{PVIsystem} can be rewritten in its integral form
\begin{equation}\label{integral_form}
\begin{split}
u(t) &= S_{\lambda}^{1, \mu}(t)u_{0}  - \eta \int_{0}^{t}S_{\lambda}^{1, \mu}(t - \tau)\partial_{x}v(\tau) d\tau - \int_{0}^{t} S_{\lambda}^{1, \mu}(t - \tau)\partial_{x}P(u,v)(\tau) d\tau\\
v(t) &= S_{\lambda}^{\alpha, \zeta}(t)v_{0}  - \eta \int_{0}^{t}S_{\lambda}^{\alpha, \zeta}(t - \tau)\partial_{x}u(\tau) d\tau - \int_{0}^{t} S_{\lambda}^{\alpha, \zeta}(t - \tau)\partial_{x}Q(u,v)(\tau) d\tau,
\end{split}
\end{equation}
where $S_{\lambda}^{\beta, \gamma}(t) = e^{-t (\beta \partial_{x}^{3} + \gamma \partial_{x} + K_{\lambda})}$ is the group to the linear system associated to \eqref{PVIsystem}. 

Now, let us consider $\alpha, \mu, \zeta$  satisfying the hypothesis of the Lemma \ref{contralbourgainestimate2} and  \ref{controlbourgainestimate}.  Next, using Lemma \ref{bourgainkdv-estimates} twice, Lemmas  \ref{contralbourgainestimate2} and  \ref{controlbourgainestimate} and finally, the fact that $L_{\beta, \gamma,  \lambda}$ is a bounded linear operator from $H_{0}^{s}(\mathbb{T})$ to $H_{0}^{s}(\mathbb{T})$, for all $s\geq 0$, we can guarantee the existence  of a constant $c>0$ such that  the following inequalities are verified
\begin{equation}\label{integralequationstabilization1}
\|S^{\beta, \gamma}_{ \lambda}(t)\phi\|_{Z_{s, \frac{1}{2}}^{\beta, \gamma , T}} \leq   c \|\phi\|_{s},
\end{equation}
for any $ \phi \in H_{0}^{s}(\mathbb{T})$,
\begin{equation}\label{integralequationstabilization21}
 \left\| \int_{0}^{t} S^{1 , \mu}_{ \lambda}(t- \tau) \partial_{x}v(\tau)d \tau \right\|_{\mathcal{Z}^1_{s}}\leq c \|v\|_{\mathcal{Z}^{\alpha}_s},
  \end{equation}
  \begin{equation}\label{integralequationstabilization21a}
 \left\| \int_{0}^{t} S^{\alpha, \zeta}_{\lambda}(t- \tau) \partial_{x}u(\tau)(\tau)d \tau \right\|_{\mathcal{Z}^{\alpha}_s} \leq c \|u\|_{\mathcal{Z}^1_s}
 \end{equation}
for any $ (u, v ) \in \mathcal{Z}_s$,
\begin{equation}\label{integralequationstabilization2}
 \left\| \int_{0}^{t} S^{1 , \mu}_{ \lambda}(t- \tau) \partial_{x}P(u,v)(\tau)d \tau \right\|_{\mathcal{Z}^1_s} \leq c \|(u, v)\|_{\mathcal{Z}_s}^{2},
 \end{equation}
and
\begin{equation}\label{integralequationstabilization3}
 \left\| \int_{0}^{t} S_{\lambda}(t- \tau) \partial_{x}Q(u,v)(\tau)d \tau \right\|_{\mathcal{Z}^{\alpha}_s} \leq c \|(u, v)\|_{\mathcal{Z}_s}^{2},
 \end{equation}
for any  $ (u, v ) \in \mathcal{Z}_s$ with $[u] = [v]=0$.

Now, thanks to the ideas introduced by  \cite[Theorem 2.1]{Slemrod} and \cite[Proposition 2.5]{laurent2010}, for given $s \geq 0$, there exists some  positive constant $M_{s}$ such that
\begin{equation*}
\|(S^{1, \mu}_{\lambda}(t)u_{0}, S^{\alpha, \zeta}_{\lambda}(t)v_{0})\|_{s} \leq M_{s}e^{- \lambda t}\|(u_{0}, v_{0})\|_{s}, \ \forall \ t \ge 0,
\end{equation*}
where $S^{\beta, \gamma}_{\lambda}$ with $\beta = 1$ and $\gamma = \mu$, and for $\beta=\alpha$ and $\gamma= \eta$ are the groups associated to the linear system \eqref{PVIsystem}. Pick $T>0$ such that $2M_{s}e^{-\lambda T} \leq e^{-\lambda ' T}.$ 
We seek a solution $(u,v)$ to the integral equations \eqref{integral_form}, as a fixed point of the following map
\begin{equation*}
\Gamma_{\lambda} (w,z) = (\Gamma_{\lambda}^{1} (w,z), \Gamma_{\lambda}^{2}(w,z)),
\end{equation*}
defined by 
\begin{equation*}
\Gamma_{\lambda}^{1} (w,z) = S^{1, \mu}_{\lambda}(t)u_{0} - \int_{0}^{t} S_{\lambda}^{1, \mu}(t- \tau)\partial_{x}v(\tau)d\tau - \int_{0}^{t}S^{1, \mu}_{\lambda}(t - \tau) (\partial_{x}P(w,z))(\tau) d \tau
\end{equation*}
and
\begin{equation*}
\Gamma_{\lambda}^{2} (w,z) = S^{\alpha, \zeta}_{ \lambda}(t)v_{0} - \int_{0}^{t} S_{\lambda}^{\alpha, \zeta}(t- \tau)\partial_{x}u(\tau)d\tau- \int_{0}^{t}S^{\alpha , \zeta}_{\lambda}(t - \tau) (\partial_{x}Q(w,z))(\tau) d \tau,
\end{equation*}
in some closed ball $$B_{R}(0)\subset  Z_{s}^{1} \cap L^{2}(0,T; L_{0}^{2}(\mathbb{T})) \times  Z_{s}^{\alpha} \cap L^{2}(0,T; L_{0}^{2}(\mathbb{T}))$$ for the $\|(w,z)\|_{\mathcal{Z}_s}$-norm. This will be done provided that $\|(u_{0}, v_{0})\|_{s} \leq \delta,$ where $\delta$ is a small number to be determined. Furthermore, to ensure the exponential stability with the claimed decay rate, the numbers $\delta$ and $R$ will be chosen in such a way that
\begin{equation*}
\|(u(T), v(T))\|_{s} \leq e^{-\lambda ' T}\|(u_{0}, v_{0})\|_{s}.
\end{equation*}

Since $\Gamma_{\lambda}$ is a contraction in $B_{R}(0)$, by a fixed point argument, its unique fixed point $(u,v) \in B_{R}(0)$ fulfills
\begin{equation*}
\|(u(T), v(T))\| = \|\Gamma_{\lambda}(u,v)\|_{s} \leq e^{-\lambda ' T} \delta.
\end{equation*}
Finally, assume that $\|(u_{0},v_{0})\|_{s} < \delta$. Changing $\delta$ into $\delta' : =  \|(u_{0}, v_{0}) \|_{s}$ and $R$ into $R'= \left( \frac{\delta'}{\delta} \right)^{\frac{1}{2}}R$, we infer that $$\|(u(T), v(T))\|_{s} \leq e^{-\lambda' T}\|(u_{0}, v_{0})\|_{s}$$ and by induction yields $$\|(u(nT), v(nT))\|_{s} \leq e^{-\lambda' n T}\|(u_{0}, v_{0})\|_{s},$$ for any $n \geq 0$. As $\mathcal{Z}_{s}  \cap L^{2}(0,T; L_{0}^{2}(\mathbb{T})) \subset C([0,T]; H_{0}^{s}(\mathbb{T}))$, we infer by the semigroup property that there exists some constant $C'>0$ such that
\begin{equation*}
\|(u(t), v(t))\|_{s} \leq C' e^{-\lambda' t}\|(u_{0}, v_{0})\|_{s},
\end{equation*}
provided that $\|(u_{0}, v_{0})\|_{s} \leq \delta$ and the proof of Theorem \ref{localstability} is completed.
\end{proof}

\subsection{A global result: Stabilization result} As mentioned  the stability result presented in Theorem \ref{localstability} is purely local. Now we are in position to extend it to a global stability.  It is well known  \cite{CaCa,laurent2010,Zhang1} in Control Theory that Theorem \ref{main} is a direct consequence of the following observability inequality.

\vspace{0.2cm}

\noindent\textit{Let $T>0$ and $R_{0}>0$ be given. There exists a constant
$ \rho >0$ such that for any $(u_{0},v_0)\in L_{0}^{2}\left(  \mathbb{T}\right) \times L_{0}^{2}\left(  \mathbb{T}\right)$
satisfying $\left\Vert (u_{0},v_0)\right\Vert\leq R_{0}$,  the corresponding solution $(u,v)$ of \eqref{PVIsystem} satisfies
\begin{equation}\label{obsineq}
\|(u_{0}, v_{0})\|^{2} \leq \rho \int_{0}^{T} \|(Gu, Gv)\|^{2}(t) dt.
\end{equation}}

So,  our next steps is to show the observability inequality.

\begin{proof}[Proof of \eqref{obsineq}] Suppose that \eqref{obsineq} does not occur. Thus,  for any $n \in \mathbb{N}$, there exists $(u_{n,0}, v_{n,0}) := (u_{n}(0), v_{n}(0)) \in L_{0}^{2}(\mathbb{T}) \times L_{0}^{2}(\mathbb{T})$ such that the solution $(u_n,v_n) \in X$ of IVP \eqref{PVIsystem}, given by Theorem \ref{gwptheorem}, satisfies
\begin{equation}\label{databounded}
\|(u_{n,0}, v_{n,0})\| \leq R_{0}
\end{equation}
and
  \begin{equation}\label{contradictioncondition}
\int_{0}^{T} \|(Gu_n, Gv_n)\|^{2}(t) dt < \frac{1}{n} \|(u_{n,0}, v_{n,0})\|^{2}.
\end{equation} 
Since $a_{n}:= \|(u_{n,0}, v_{n,0})\|$ is a bounded sequence in $\mathbb{R}$, we can choose a subsequence of $\{a_{n}\}$, still denoted by $\{a_{n}\}$, such that $\lim_{n \rightarrow \infty} a_{n} = a.$ So, there are two possibilities for the limit, which will be divided in the following cases: $a > 0$ and $a =0.$

\vspace{0.2cm}

\noindent\textit{i. Case  $a > 0$.}

\vspace{0.2cm}
Note that the sequence $(u_{n}, v_{n})$ is bounded in $  L^{\infty}(0,T; L^{2}(\mathbb{T}))   \times L^{\infty}(0,T; L^{2}(\mathbb{T}))$ and, also, in $ \mathcal{X}_0$. Thus, applying Lemma \ref{bilinearestlemma} in each term of $P$ and $Q$ we have that $\partial_{x}P(u_{n},v_{n})$ and $\partial_{x}Q(u_{n},v_{n})$ are bounded in $\mathcal{Z}^{1}_{0, -\frac{1}{2}}$ and $\mathcal{Z}^{\alpha}_{0, -\frac{1}{2}}$, respectively,  with $|\zeta| + |\mu| < \epsilon$, for some $\epsilon \ll 1$ and $\alpha < 0$. In particular, is bounded in $\mathcal{X}^{1}_{0, -\frac{1}{2}}$ and $\mathcal{X}^{\alpha}_{0, -\frac{1}{2}}$, respectively. Additionally,  Bourgain spaces are reflexive and have the following compact embedding $\mathcal{X}_0 \hookrightarrow \mathcal{X}_{0, -1}.$ Therefore, we can extract a subsequence of $\{(u_{n}, v_{n})\}$, still denoted by $\{(u_{n}, v_{n})\}$, such that
\begin{equation*}
(u_{n}, v_{n}) \rightharpoonup (u,v) \; \mbox{in} \; \mathcal{X}_0, \quad
(u_{n}, v_{n}) \longrightarrow (u,v) \; \mbox{in} \; \mathcal{X}_{0, -1}
\end{equation*}
and
\begin{equation*}
(\partial_{x}P(u_{n},v_{n}), \partial_{x}Q(u_{n},v_{n})) \rightharpoonup (f,g) \; \mbox{in} \;\mathcal{X}_{0, -\frac{1}{2}},
\end{equation*}
where $(u,v) \in \mathcal{X}_{0}$ and $(f, g) \in \mathcal{X}_{0, -\frac{1}{2}}$. Moreover, since $\mathcal{X}_0$ is continuously  embedded in $L^{4}((0,T)\times \mathbb{T})$, we have
\begin{equation*}
\|u_{n} v_{n}\|_{L^{2}((0,T)\times \mathbb{T})} \leq \|u_{n}\|_{L^{4}((0,T)\times \mathbb{T}) }^{2} \|v_{n}\|_{L^{4}((0,T)\times \mathbb{T})}^{2} \lesssim \|(u_{n}, v_{n})\|_{\mathcal{X}_0}^{4},
\end{equation*}
which implies that $(P(u_{n}, v_{n}), Q(u_{n}, v_{n}))$ is bounded in $L^{2}( (0,T)\times \mathbb{T} ) \times L^{2}((0,T)\times \mathbb{T} )$. Hence, it follows that $\partial_{x}(P(u_{n}, v_{n}), Q(u_{n}, v_{n}))$ is bounded in $L^{2}(0,T;H^{-1}(\mathbb{T})) \times L^{2}(0,T;H^{-1}(\mathbb{T}))  = \mathcal{X}_{-1,0}. $ Interpolating the spaces $\mathcal{X}_{0, - \frac{1}{2}}$ and $\mathcal{X}_{-1, 0}$, we obtain that  $\partial_{x}(P(u_{n}, v_{n}), Q(u_{n}, v_{n}))$ is bounded in $\mathcal{X}_{-\theta, -\frac{1}{2}(1- \theta) }$ for any $\theta \in [0,1]$. As $0 < \theta < 1$, it follows that $\mathcal{X}_{-\theta, -\frac{1}{2}(1- \theta) }$ is compactly embedded in $\mathcal{X}_{-1 , -\frac{1}{2}}$. Thus, we can extract a subsequence of $\{(u_{n}, v_{n})\}$, still denoted by  $\{(u_{n}, v_{n})\}$, such that
\begin{equation}\label{convergence4}
\partial_{x}(P(u_{n}, v_{n}), Q(u_{n}, v_{n})) \longrightarrow  (f,g) \ \mbox{in} \ \mathcal{X}_{-1 , -\frac{1}{2}}.
\end{equation}
Thanks to \eqref{contradictioncondition} and the continuity of G, we ensures that
\begin{equation}\label{convergence1}
\int_{0}^{T} \|(Gu_{n}, Gv_{n})\|^{2} dt \longrightarrow \int_{0}^{T} \|(Gu, Gv)\|^{2} dt  = 0.
\end{equation}
This convergence  means that $(Gu, Gv)= (0,0)$. Besides, since $g$ is positive on $\omega\subset\mathbb{T}$, we have from definition  \eqref{defG}  that 
\begin{equation*}
u(x,t) =\int_{\mathbb{T}} g(y) u(y,t) dy  = c_{1}(t)\ \ \text{and} \ \ v(x,t) = \int_{\mathbb{T}}  g(y) v(y,t) dy= c_{2}(t)\ \ \text{on }\omega \times (0,T).
\end{equation*}
Thus, letting $n \longrightarrow \infty$, we obtain from \eqref{PVIsystem} that
\begin{equation*}
\begin{cases}
\partial_{t} u + \partial_{x}^{3}u + \mu \partial_{x} u + \eta \partial_{x}v = f,&\ \text{on}\  \mathbb{T}\times(0,T),\\
\partial_{t} v + \alpha \partial_{x}^{3}v + \zeta \partial_{x}v + \eta \partial_{x}u = g , &\ \text{on}\   \mathbb{T}\times(0,T),\\
u = c_{1}(t), \quad v = c_{2}(t),&\ \text{on}\  \omega \times (0,T).
\end{cases}
\end{equation*}

Consider  $w_{n} = u_{n} -u$, $z_{n} = v_{n} - v$, $P_{n} = - \partial_{x}P(u_{n}, v_{n}) + f - K_{0}u_{n}$ and $Q_{n} = - \partial_{x} Q(u_{n}, v_{n}) + g - K_{0}v_{n}$. Thus, $$(w_{n}, z_{n}) \rightharpoonup (0,0)\quad \text{in} \quad \mathcal{X}_{0}$$ and satisfies
\begin{equation*}
\begin{cases}
\partial_{t}w_{n} + \partial_{x}^{3}w_{n} + \mu \partial_{x} w_{n} + \eta \partial_{x} z_{n} = P_{n},\\
\partial_{t}z_{n} + \alpha \partial_{x}^{3}z_{n} + \zeta \partial_{x} z_{n} + \eta \partial_{x} w_{n} = Q_{n}.
\end{cases}
\end{equation*}
Now, note that using the linearity of $G$ we can rewrite
\begin{equation*}
\begin{split}
\int_{0}^{T} \|(Gw_{n}, Gz_{n})\|^{2} dt =& \int_{0}^{T} \|(Gu_{n}, Gv_{n})\|^{2} dt + \int_{0}^{T} \|(Gu, Gv)\|^{2} dt - 2 \int_{0}^{T} \left< Gu_{n}, Gu \right> dt\\
& - 2 \int_{0}^{T} \left< Gv_{n}, Gv \right> dt\\
%= &  \int_{0}^{T} \|(Gu_{n}, Gv_{n})\|^{2} dt +  \int_{0}^{T} \|(Gu, Gv)\|^{2} dt\\
%&  - 2\int_{0}^{T} \left< Gw_{n}, Gu \right> dt -2 \int_{0}^{T} \|G{u}\|^{2} dt \\
% & - 2\int_{0}^{T} \left< Gz_{n}, Gv \right> dt -2 \int_{0}^{T} \|G{v}\|^{2} dt \\
%  = & \int_{0}^{T} \|(Gu_{n}, Gv_{n})\|^{2} dt -  \int_{0}^{T} \|(Gu, Gv)\|^{2} dt\\.
%  &  - 2\int_{0}^{T} \left< Gw_{n}, Gu \right> dt - 2\int_{0}^{T} \left< Gz_{n}, Gv \right> dt.
\end{split}
\end{equation*}
From \eqref{convergence1}, we obtain
\begin{equation}\label{convergence3}
\int_{0}^{T} \|(Gw_{n}, Gz_{n})\|^{2} dt \longrightarrow 0.
\end{equation}
On the other hand, 
\begin{equation}\label{identity1}
 \int_{0}^{T} \|(Gw_{n}, Gz_{n})\|^{2} dt  =I + II + III,
\end{equation}
where, 
\begin{equation*}
I = \int_{0}^{T} \int_{\mathbb{T}} g^{2}(x)\left[w_{n}^{2}(x,t) + z_{n}^{2}(x,t)\right] dx dt,
\end{equation*}
\begin{equation*}
 II =  \int_{0}^{T} \left( \int_{\mathbb{T}} g^{2}(x) dx\right) \left( \int_{\mathbb{T}} g(y) w_{n}(y,t) dy \right)^{2} dt  +  \int_{0}^{T} \left( \int_{\mathbb{T}} g^{2}(x) dx\right) \left( \int_{\mathbb{T}} g(y) z_{n}(y,t) dy\right)^{2}  dt 
 \end{equation*}
 and
 \begin{equation*}
 \begin{split}
III =& -2 \int_{0}^{T} \left( \int_{\mathbb{T}} g^{2}(x) w_{n}(x, t) dx \right) \left( \int_{\mathbb{T}}g(y)w_{n}(y,t) dy\right) dt\\
 &-2 \int_{0}^{T} \left( \int_{\mathbb{T}} g^{2}(x) z_{n}(x, t) dx \right) \left( \int_{\mathbb{T}}g(y)z_{n}(y,t) dy\right) dt.
 \end{split}
 \end{equation*} 
Let us prove that each previous term tends to zero as $n \longrightarrow \infty$. First, a direct application of Lemma \ref{convergence2prop} implies that $II \longrightarrow 0$.  Now, we can estimate $III$ as follows
\begin{equation*}
\begin{split}
%\int_{0}^{T} \left( \int_{\mathbb{T}} g^{2}(x) w_{n}(x, t) dx \right) \left( \int_{\mathbb{T}}g(y)w_{n}(y,t) dy\right) dt + \int_{0}^{T} \left( \int_{\mathbb{T}} g^{2}(x) z_{n}(x, t) dx \right) \left( \int_{\mathbb{T}}g(y)z_{n}(y,t) dy\right) dt \leq\
|III| \leq &\ \|b_{n}\|_{L^{2}(0,T)}\left\|\int_{\mathbb{T}} g^{2}w_{n}(x,t) dx\right\|_{L^{2}(0,T)} + \|c_{n}\|_{L^{2}(0,T)}\left\|\int_{\mathbb{T}} g^{2}z_{n}(x,t) dx\right\|_{L^{2}(0,T)} \\
% \leq\/ &\ |b_{n}\|_{L^{2}(0,T)}\|g\|_{L^{4}(\mathbb{T})}^{2}\|w_{n}\|_{L^{2}(\mathbb{T} \times (0,T))} + \|c_{n}\|_{L^{2}(0,T)}\|g\|_{L^{4}(\mathbb{T})}^{2}\|z_{n}\|_{L^{2}(\mathbb{T} \times (0,T))}  \\
 \leq \/ &\ \|g\|_{L^{4}(\mathbb{T})}^{2}\|(b_{n}, c_{n})\|_{L^{2}(0,T) \times L^{2}(0,T)}\|(w_{n}, z_{n})\|_{\mathcal{X}_{0}} 
 \\
 \leq\/ &  \ C \|(b_{n}, c_{n})\|_{L^{2}(0,T) \times L^{2}(0,T)}, 
\end{split}
\end{equation*}
where $C$ is a positive constant. So, it follows that $III \longrightarrow 0$, when $n \longrightarrow \infty$, again thanks to  the Lemma \ref{convergence2prop}. Lastly, combining the last two convergences with \eqref{convergence3}, we infer by \eqref{identity1}  that $I \longrightarrow 0$. 

We claim that 
\begin{equation}\label{convergence4a}
\|(w_{n}, z_{n})\|_{L^{2}( 0,T; L^{2}( \widetilde{\omega} ) ) \times L^{2}(0,T; L^{2}( \widetilde{\omega} ))  } \longrightarrow 0, \quad \text{as} \ n\to\infty.
\end{equation}
In fact, it is sufficient to observe that
\begin{equation*}
\begin{split}
\|(w_{n}, z_{n})\|_{L^{2}( 0,T; L^{2}( \widetilde{\omega}) ) \times L^{2}(0, T; L^{2}(  \widetilde{\omega})) } & = \int_{0}^{T} \int_{\widetilde{\omega}}  |g(x)|^{-2}  \left[ |g(x)|^{2} w_{n}^{2}(x,t)+   |g(x)|^{2} z_{n}^{2}(x,t)  \right] dxdt\\
& \leq \frac{4}{\|g\|_{L^{\infty}(\mathbb{T})}^{2}} \cdot  \int_{0}^{T} \int_{\mathbb{T}} g^{2}(x)\left[w_{n}^{2}(x,t) + z_{n}^{2}(x,t)\right] dx dt,
\end{split}
\end{equation*}
where $\widetilde{\omega}:= \left\{ g(x) > \frac{\|g\|_{L^{\infty}(\mathbb{T})}}{2} \right\}$ and using this previous inequality \eqref{convergence4a} follows. Additionally, note that
\begin{equation*}
\begin{split}
\|(P_{n}, Q_{n})\|_{\mathcal{X}_{-1, - \frac{1}{2}}} \leq &  \|GG^{*}(u_{n}, v_{n})\|_{\mathcal{X}_{-1, - \frac{1}{2}}} +  \|(\partial_{x}P(u_{n}, v_{n}) - f, \partial_{x}Q(u_{n}, v_{n}) - g)\|_{\mathcal{X}_{-1, - \frac{1}{2}}}\\
   \leq & C \|(Gu_{n}, Gv_{n})\|_{\mathcal{X}_{0,0}} +  \|(\partial_{x}P(u_{n}, v_{n}) - f, \partial_{x}Q(u_{n}, v_{n}) - g)\|_{\mathcal{X}_{-1, - \frac{1}{2}}}.
\end{split}
\end{equation*}
From \eqref{convergence4}, \eqref{convergence1} and the previous inequality we obtain
\begin{equation*}
\|(P_{n}, Q_{n})\|_{\mathcal{X}_{-1, - \frac{1}{2}}} \longrightarrow 0,  \ \ \mbox{as} \ n \longrightarrow \infty.
\end{equation*}

Applying Proposition \ref{propagationcompactness} with $b'= 0$ and $b= \frac{1}{2}$ yields that
\begin{equation}\label{convergence5}
(w_{n}, z_{n}) \longrightarrow 0 \ \mbox{in} \ L^{2}_{loc}(0,T; L^{2}( \mathbb{T})  ) \times  L^{2}_{loc}(0,T; L^{2}( \mathbb{T}) ).
\end{equation}
Consequently, 
\begin{equation}\label{convergence6}
(P(u_{n}, v_{n}), Q(u_{n}, v_{n})) \longrightarrow (P(u,v), Q(u,v))\text{ in } L^{1}_{loc}(0,T; L^{1}(\mathbb{T}))  \times L^{1}_{loc}(0,T; L^{1}(\mathbb{T}) )
\end{equation}
and
$$(\partial_{x} P(u_{n}, v_{n}), \partial_{x} Q(u_{n}, v_{n}))\longrightarrow(\partial_{x}P(u,v), \partial_{x}Q(u,v)),$$ in the distributional sense.  Therefore, $(f,g) = (\partial_{x}P(u,v), \partial_{x}Q(u,v))$ and $(u,v) \in \mathcal{X}_{0}$ satisfies
 \begin{equation*}
\begin{cases}
\partial_{t}u + \partial_{x}^{3}u + \mu \partial_{x} u + \eta \partial_{x}v + \partial_{x}P(u,v) =0,&\ \text{on}\  \mathbb{T}\times(0,T),\\
\partial_{t}v + \alpha \partial_{x}^{3}v + \zeta \partial_{x} v + \eta \partial_{x} u + \partial_{x}Q(u,v) =0,&\ \text{on}\  \mathbb{T}\times(0,T),\\
(u, v) = (c_{1}(t), c_{2}(t)),&\ \text{on}\  \omega \times (0,T).
\end{cases}
 \end{equation*}
 From Corollary \ref{uniquecontinuation}, we infer that $(u,v) = (0,0)$ on $\mathbb{T} \times (0,T)$, which combined with \eqref{convergence5} yields that $(u_{n}, v_{n})\to(0,0)$ in $L^{2}_{loc}((0,T); L^{2}(\mathbb{T})) \times L^{2}_{loc}((0,T); L^{2}(\mathbb{T}))$. We can pick some time $t_{0} \in [0,T]$ such that $(u_{n}(t_{0}), v_{n}(t_{0})) \longrightarrow (0,0)$ in $ L^{2}(\mathbb{T}) \times L^{2}(\mathbb{T}) $. Since
 \begin{equation*}
    \|(u_{n}(0), v_{n}(0))\|^{2} = \|(u_{n}(t_{0}), v_{n}(t_{0}))\|^{2} + \int_{0}^{t_{0}}\|(Gu_{n}, Gv_{n})\|^{2} dt,
 \end{equation*}
 it is inferred that $a_{n} = \|(u_{n}(0), v_{n}(0))\| \longrightarrow 0$ which is a contradiction to the assumption $a > 0$.
\vspace{0.2cm}

\noindent\textit{ii.  Case $a = 0$.}

\vspace{0.2cm}
 
 Note first that $a_{n} >0$ for all $n$. Let $(w_{n}, z_{n}) = \left( \frac{u_{n}}{a_{n}}, \frac{v_{n}}{a_{n}}\right)$ for all $n \geq 1$. Then
 \begin{equation*}
 \begin{cases}
 \partial_{t}w_{n} + \partial_{x}^{3}w_{n} + \mu \partial_{x}w_{n} + \eta \partial_{x}z_{n}  + K_{0}w_{n} +  a_n \partial_{x}P(w_{n}, z_{n})=0,\\
  \partial_{t}z_{n} + \alpha \partial_{x}^{3}z_{n} + \zeta \partial_{x}z_{n} + \eta \partial_{x}w_{n} + K_{0}z_{n} + a_n \partial_{x}Q(w_{n}, z_{n})=0,
 \end{cases}
 \end{equation*}
 \begin{equation}\label{convergencecasezero}
 \int_{0}^{T} \|(Gw_{n},Gz_{n})\|^{2} dt < \frac{1}{n}
 \end{equation}
and
\begin{equation}\label{normcasezero}
\|(w_{ n}(0),z_{ n}(0))\| = \frac{\|u_{n}(0)\|}{\|(u_{0,n}, v_{0,n})\|}+  \frac{\|v_{n}(0)\|}{\|(u_{0,n}, v_{0,n})\|} =1.
\end{equation}
So, we obtain that the sequence $\{(w_{n}, z_{n})\}$ which are bounded in both spaces $L^{\infty}(0,T; L^{2}(\mathbb{T})) \times L^{\infty}(0,T; L^{2}(\mathbb{T}))$ and $\mathcal{X}_{0}$. Indeed, $\|(w_{n}(t), z_{n}(t))\|$ is a nonincreasing function of $t$ and since $a_{n}$ is bounded, we have 
 \begin{equation*}
 \begin{split}
 \|(w_{n}, z_{n})\|_{\mathcal{X}_{0}} \leq& \ C_{0} + \frac{(C (\widetilde{\epsilon})T^{1- \widetilde{\epsilon}} + C_{1}T^{\theta})}{a_{n}} \|(u_{n}, v_{n})\|_{\mathcal{X}_{0}} + \frac{C_{0}C_{3} T^{\theta}}{a_{n}^{2}}\|(u_{n}, v_{n})\|_{\mathcal{X}_{0}}.
 \end{split}
 \end{equation*}
We can extract a subsequence of $\{(w_{n}, z_{n})\}$, still denoted by $\{(w_{n}, z_{n})\}$, such that $$(w_{n}, z_{n}) \rightharpoonup  (w,z) \text{ in } \mathcal{X}_{0}, \quad (w_{n}, z_{n})\rightarrow (w,z) \text{ in } \mathcal{X}_{-1, -\frac{1}{2}}\quad \text{and} \quad (w_{n}, z_{n})\rightarrow (w,z) \text{ in } \mathcal{X}_{-1, 0},$$
as $n\to\infty$. Moreover, the sequence $\{(\partial_{x}P(w_{n}, z_{n}), \partial_{x}Q(w_{n}, z_{n}))\}$ is bounded in the space $\mathcal{X}_{0, -\frac{1}{2}}$, and therefore $$a_{n}(\partial_{x}P(w_{n}, z_{n}), \partial_{x}Q(w_{n}, z_{n}))\longrightarrow 0 \text{ in }\mathcal{X}_{0, -\frac{1}{2}},$$
when $n\to\infty.$ Thus, $(w,z)$ is solution of 
\begin{equation*}
\begin{cases}
\partial_{t}w + \partial_{x}^{3}w + \mu \partial_{x}w + \eta\partial_{x} z =0,&\ \text{on}\  \mathbb{T}\times(0,T),\\
\partial_{t}z + \alpha \partial_{x}^{3}z + \zeta \partial_{x}z +\eta \partial_{x} w =0,&\ \text{on}\  \mathbb{T}\times(0,T),\\
w=c_1(t), \quad z= c_{2}(t),&\ \text{on}\  \omega \times (0,T).
\end{cases}
\end{equation*}
Using Holmgren's uniqueness theorem, we can deduce that  $w(x,t)=c_1(t)=c_1$ and $z(x,t)= c_{2}(t)= c_2$. However, as  $[w]=[z]=0$, we infer that $c_1=c_2=0$. 

According to \eqref{convergencecasezero}
\begin{equation*}
 \int_{0}^{T} \|(Gw_{n},Gz_{n})\|^{2} dt  \longrightarrow 0, \quad \text{ as } n\to\infty,
\end{equation*}
and consequently $(K_{0}w_{n}, K_{0}z_{n})$ converges strongly to $(0,0)$ in $\mathcal{X}_{-1, -\frac{1}{2}}$.  Now, arguing as in the case $a>0$,  that is, applying again Proposition \ref{propagationcompactness}, it follows that $$(w_{n}, z_{n}) \longrightarrow (0,0) \text{ in }L^{2}_{loc}(0,T; L^{2}(\mathbb{T})) \times L^{2}_{loc}(0,T; L^{2}(\mathbb{T})), $$ and  consequently $\|(w_{n}(0), z_{n}(0))\| \rightarrow 0 $ which contradicts  \eqref{normcasezero} and \eqref{obsineq} is shown. 
\end{proof}

\subsection{Controllability result: Nonlinear problem}
We now consider the controllability properties to the nonlinear open loop control system \eqref{systemcontrollocal_ns_int}.  The following result is local and classical.

\begin{theorem}\label{cont_local}
Let $T > 0$ and $s \geq 0$ be given. Then there exists a $\delta > 0$ such that for any $(u_{0}, v_{0})$, $(u_{1}, v_{1}) \in H_{0}^{s}(\mathbb{T}) \times H_{0}^{s}(\mathbb{T})$ and $
\|(u_{0}, v_{0})\|_{s} \leq \delta, \ \ \|(u_{1}, v_{1})\|_{s} \leq \delta,$
one can find two control inputs $(f,h) \in L^{2}([0,T]; H_{0}^{s}(\mathbb{T})) \times  L^{2}([0,T]; H_{0}^{s}(\mathbb{T}))$ such that equation \eqref{systemcontrollocal_ns_int} has a solution $$(u,v) \in C([0,T]; H_{0}^{s}(\mathbb{T})) \times  C([0,T]; H_{0}^{s}(\mathbb{T}))$$ satisfying $(u(x,0), v(x,0))= (u_{0}(x), v_{0}(x)) \ \ \mbox{and} \ \ (u(x, T), v(x,T)) = (u_{1}(x), v_{1}(x)).$
\end{theorem}

\begin{proof} The proof is analogous as done in \cite{CaCa, laurent2010,Zhang1}, precisely the result is consequence of Corollary \ref{corcontrollocal} and Remark \ref{controolsistemdiagonal} taking into account the following maps 
\begin{equation*}
\begin{split}
u(t) = S^{1, \mu} (t)u_{0} + \int_{0}^{t} &S^{1, \mu}(t - \tau)(Gf)(\tau) - \eta \int_{0}^{t}S^{1, \mu}(t - \tau)\partial_{x}v(\tau) d \tau \\&- \int_{0}^{t}S^{1, \mu}(t - \tau)\partial_{x}( P(u,v))(\tau)d\tau
\end{split}
\end{equation*}
and 
\begin{equation*}
\begin{split}
v(t) = S^{\alpha, \zeta}(t)v_{0} + \int_{0}^{t} &S^{\alpha, \zeta}(t - \tau)(Gh)(\tau) - \eta \int_{0}^{t}S^{\alpha, \zeta}(t- \tau) \partial_{x}u(\tau) d \tau\\&- \int_{0}^{t}S^{\alpha, \zeta}(t - \tau)\partial_{x}( Q(u,v))(\tau)d\tau.
\end{split}
\end{equation*}
\end{proof}

\section{Further comments and open problems}\label{Sec6}

This article deals by the first time with the global aspect of control problems for the long waves in dispersive media, precisely systems like \eqref{system} with the coupled nonlinearity \eqref{nonlinearitya}. From the perspective of considering the functions with null mean, the presence of the terms involving the constants $B$ and $C$ in the nonlinearities make the system studied coupled in the linear part (as well as in the nonlinear part) as can be seen in \eqref{system1}, that is, precisely the system with the terms $\eta \partial_{x} v $ and $ \eta \partial_{x} u $, in the first and second equations of the system of \eqref{systema}. Unless that $\eta$ can be considered $0$, when $ [u] = [v] = \beta $ and $ B = -C $, the matrix of the operator $ L $ defined by \eqref{operatorL} is not a diagonal matrix. For this reason, the arguments used in general for a singular dispersive equation (e.g. as the KdV case \cite{laurent2010,Russel96}) in the study of the existence of solutions, controllability and stabilization can not be directly applied here.

Note that the previous works are concentrated in a single KdV equation (or, as in the case of \cite{micu2} and \cite{capistranokp2020}, linear results or local results for coupled KdV systems). It is important to point out that in \cite{micu2} the authors left some open problems with respect to the global controllability. In this spirit, our work is dedicated to cover this lack of results, that is, when presenting global results the article intends to give the first step to understand global control problems in periodic domains for systems like \eqref{systema} with quadratic nonlinearities.

\begin{remark}In what concerns our main results the following remarks are
worth mentioning:
\begin{itemize}
\item  The system \eqref{system}, differently from what happens for the KdV  or for a pair of KdV which is decoupled only by the linear part, admits two families of eigenvalues associated with two different families of eigenfunctions. In this way, we proceed carefully to guarantee some spectral properties essential for obtaining a gap condition that satisfies an Ingham type theorem (see \cite[Theorem 4.6]{KomLor2005}). 
\item Another new and important fact, and itself interesting, is that the orthonormal basis for the space $L^{2}(\mathbb{T})\times L^{2}(\mathbb {T})$ formed by the eigenfunctions is not a pair compound for two identical copies of the basis $\{\frac{1}{\sqrt{2\pi}}e^{ikx}\}_{k\in\mathbb{Z}}$, as usual in this kind of problem, which makes the controls obtained for each equations different but comparable to each other.  
\item The global results presented in Theorems \ref{main} and \ref{main1}  are truly nonlinear and, which is more important, are global properties, which means that the initial and final data are controlled in a ball with no size restrictions. 
\item The novelty is that for the first time \textit{Fourier restriction spaces} introduced by Bourgain \cite{Bourgain} are used in two different dispersions to ensure the  global control results. In fact,  T. Oh \cite{Oh} noticed, for the first time,  the possibility of the occurrence of some resonance when dealing with systems of KdV type equations in the periodic setting. This fact might obstruct the gain of extra regularity. Also, more recently,  Yang and Zhang \cite{Zhang} proved new estimates related to this kind of system (see Section \ref{Sec3}). 
\item It is important to mention that the propagation results has been successfully applied in control theory in several systems represented by single equations, such as wave equation \cite{DeGeLe}, after that  for the Schrödinger equation \cite{laurent20101}, for the Benjamin-Ono equation \cite{linaresrosier2015} , KdV equation \cite{laurent2010}, the Kawahara equation \cite{Zhang1}, biharmonic Schrödinger equation \cite{CaCa}, for the Benney-Luke equation \cite{Quintero} and, finally, for the Benjamin equation \cite{PanVielma}.
\end{itemize}
\end{remark}

So, this work opens a series of situations that can be studied to understand the well-posedness theory and global controllability problems for long waves in dispersive media. We will now detail below the novelties of this work and open issues that seem interesting from a mathematical point of view.

\subsection{Control problems}
The problems in this work were solved requiring some conditions over the constants of the system, namely $\alpha$, $\mu$ and $\zeta $. Under the conditions $\alpha < 0$ and $\zeta - \mu > 0$, we find eigenfunctions associated to the operator $ L $ that define an orthonormal basis in $L^{2}(\mathbb{T}) \times L^{2}(\mathbb{T})$, so through a spectral analysis we were able to show, using the \textit{moment method} \cite{Russell}, that linear system \eqref{openloopsystem_int} is exactly controllable.

In addition, the global control problems are also verified thanks to the smoothing properties of the Bourgain spaces. This is the main novelty in this work, since it is not to our knowledge that systems like \eqref{systema} have global control properties (see \cite{capistranokp2020,micu2}  for local results). In fact, the main tool used here is Bourgain spaces in different dispersions. With the smoothing properties of the Bourgain spaces in hand, the propagation of singularities showed in \cite{laurent2010}, for the single KdV equation, can be extended to the coupled KdV system defined by the operator $ L $ and together with the \textit{bilinear estimates} and a \textit{unique continuation property} we can achieve the global control results. 

It important to point out that even Bourgain spaces are important to establish the results for data in $L^2_0(\mathbb{T})\times L^2_0(\mathbb{T})$, the propagation of compactness property is fundamental to obtain our global results. Indeed, in \cite{Russel96} Bourgain spaces were used however the results were local and the arguments there were ``basically" small perturbations of linear results.

Finally,  there is a drawback in our method,  we are able to solve these global control problems only with two controls input. The use of a control in one of the equations is still an open issue.

\subsection{Well-posedness theory}
As mentioned before, the Bourgain spaces are the key point in showing global results in this article. These spaces have been applied with success in the literature for global control results in less regular spaces (see, for instance, \cite{CaCa,laurent20101,laurent2015,laurent2010,PanVielma}).

The Bourgain spaces related to the linear system associated to \eqref{openloopsystem_int}, with $f=g=0$, can be defined \textit{via} the norm $$ \|W(-t)u\|_{H_{x}^{s}H_{t}^{b}} =: \|u\|_{X_{s,b}} $$ where $W(t) := e^{-tL}$ is the strongly continuous group generated by the operator $ L $. However, in the aspect of the nonlinear problem, it is necessary to make bilinear estimates and there are no studies on such estimates with Bourgain spaces of this nature. To get around this situation, we use two Bourgain spaces associated with each dispersion present in \eqref{openloopsystem_int}. This strategy was used  in \cite{Oh, Zhang} for the study of the well-posedness of  KdV-KdV type systems. Thus, the terms $ - \eta \partial_{x} v $ and $ - \eta \partial_{x} u $ are treated as part of nonlinearity in both studies: the well-posedness theory and in the control problem that concerns the asymptotic behavior of solutions to the problem \eqref{PVIsystem}.

Since such spaces do not have equivalent norms, a natural question that arises is whether we can estimate the components of the first equation with Bourgain's norm regarding the dispersion of the second and the converse. In this sense, lemmas \ref{derivativeestimate} and \ref{bilinearestlemma} provide answers to this problem and are essential for this work. 

Lemma \ref{derivativeestimate} is an extension of Lemma 3.10 presented in \cite{Zhang}. In this work we verify that the result is still valid in the range $\frac{1}{3} < b < \frac{1}{3}$. Since $\beta_{1}, \beta_{2}, \gamma_{1}$ and $\gamma_{2}$ are fixed constants, the lemma  is still true if we change the hypothesis $|\gamma_{1}| + |\gamma_{2}| < \epsilon$ for $\nu := |  \beta_{2} - \beta_{1}| -| \gamma_{1} - \gamma_{2}| \geq \widetilde{\delta} >0$, but the condition $\beta_{2} \neq \beta_{1}$  is still required.

In turn, the key point of Lemma \ref{bilinearestlemma} is that the function $H$ defined by \eqref{functionH} is $\delta$-significant, and for it the condition $|\eta| + |\mu| < \epsilon$ is required. Furthermore, due to the nature of the nonlinearity given in \eqref{nonlinearitya} it is also essential that $\alpha$ is strictly negative. To extend the results presented here to a larger class of constants $\alpha,\mu,\zeta,\eta \in \mathbb {R}^{*} $ it is necessary to find conditions for which the function $ H $ is $\delta $-significant. So, the well-posedness theory, global controllability and global stabilization properties for the system \eqref{system1} where $\zeta $ and $\mu $ are not small enough is still an open problem, as well as the case when $ \alpha> 0$.

\appendix
\section{Auxiliary results}\label{Apendice}
In this appendix we will give some results which were used throughout the paper.  The first result gives us properties of the two integral quantities. 
\begin{lemma}\label{convergence2prop}
Let $(w_{n}, z_{n} )\in \mathcal{X}_s$ be bounded sequences.
Define
\begin{equation*}
b_{n} := \int_{\mathbb{T}} g(y) w_{n}(y,t) dy \ \ \mbox{and} \ \ c_{n}:= \int_{\mathbb{T}}g(y)z_{n}(y,t) dy.
\end{equation*}
If $(w_{n}, z_{n}) \rightharpoonup  (0,0)$ in $\mathcal{X}_s$ then
\begin{equation}\label{convergence2}
b_{n}, c_{n} \longrightarrow 0 \ \mbox{in} \ L^{2}(0,T).
\end{equation}
\end{lemma}
\begin{proof}
By Cauchy-Schwarz inequality  we have
\begin{equation*}
\begin{split}
\|(b_{n}, c_{n})\|_{L^{2}_{T} \times L^{2}_{T} }^{2} \leq & \int_{0}^{T}\|g\|_{L^{2}(\mathbb{T})}^{2} \|(w_{n}(\cdot, t), z_{n} (\cdot , t) )\|_{L^{2}(\mathbb{T})}^{2} dt\\
\leq & \ C \|g\|_{L^{2}(\mathbb{T})}^{2} \|(w_{n}, z_{n})\|_{\mathcal{X}_{0, 0}},
\end{split}
\end{equation*}
for some constant $C>0$. By hypothesis  $$(w_{n}, z_{n}) \rightharpoonup  (0,0) \text{ in }\mathcal{X}_0,$$ since $\mathcal{X}_{0}$ is compactly embedded in $\mathcal{X}_{0,0}$ the result is proved.
\end{proof}

The next result of this appendix shows that we can propagate, due to the smoothing effect of the Bourgain spaces, the compactness of a $ \omega \subset \mathbb{T} $ for the entire space $ \mathbb{T}$. 

\begin{proposition}[Propagation of compactness]\label{propagationcompactness}
Let $T > 0$ and $0 \leq  b' <  b  \leq 1$ be given.  Suppose that $(u_{n}, v_{n}) \in \mathcal{X}_{0, b}$ and $(f_{n}, g_{n}) \in \mathcal{X}_{-2 + 2b, -b}$ satisfies
\begin{equation*}
    \begin{cases}
\partial_{t}u_{n} + \partial_{x}^{3}u_{n}  + \mu \partial_{x} u_{n} + \eta \partial_{x}v_{n} = f_{n},&\ \text{on}\  \mathbb{T}\times(0,T),\\
\partial_{t}v_{n} + \alpha \partial_{x}^{3}v_{n} + \zeta \partial_{x} v_{n} + \eta \partial_{x}  u_{n} \ = g_{n},&\ \text{on}\  \mathbb{T}\times(0,T),
    \end{cases}
\end{equation*}
for $n \in \mathbb{N}$. Assume that there exists a constant $C >0$ such that
\begin{equation}\label{hipcompctness1}
    \|(u_{n}, v_{n})\|_{\mathcal{X}_{0, b}} \leq C, \quad \forall n \geq 1,
\end{equation}
and that
\begin{equation}\label{hipcompactness2}
\begin{split}
\|(u_{n}, v_{n})\|_{\mathcal{X}_{-2 + 2b, -b}} \longrightarrow 0 , &\ \ \mbox{as} \  n \longrightarrow \infty , \quad \quad
\|(f_{n}, g_{n})\|_{\mathcal{X}_{-2 + 2b, -b}} \longrightarrow 0 ,  \ \ \mbox{as} \ n \longrightarrow \infty, \\
&\|(u_{n}, v_{n})\|_{\mathcal{X}_{-1 + 2b, -b}} \longrightarrow 0, \ \ \mbox{as} \ n \longrightarrow \infty .
\end{split}
\end{equation}
In addition, assume that for some nonempty open set $\omega \subset  \mathbb{T}$ it holds
\begin{equation*}
(u_{n}, v_{n}) \longrightarrow (0,0) \ \mbox{in} \ L^{2}(0,T;L^2(\omega)) \times  L^{2}(0,T;L^2(\omega)).
\end{equation*}
Then,
\begin{equation*}
(u_{n}, v_{n}) \longrightarrow (0,0) \ \mbox{in} \  L^{2}_{loc}(0,T;L^2(\mathbb{T}))\times  L^{2}_{loc}(0,T;L^2(\mathbb{T})).
\end{equation*}
\end{proposition}
\begin{proof}
Pick $\phi \in C^{\infty}(\mathbb{T})$ and $\psi \in C_{0}^{\infty}((0,T))$ real valued and set 
\begin{equation*}
\Phi = \phi(x)D^{-2} \ \ \mbox{and} \ \ \Psi = \psi(t)B,
\end{equation*}
where $D$ is defined by 
\begin{equation}\label{k10}
\widehat{D^{r}u}\left(  k\right)  =\left\{
\begin{array}
[c]{ll}%
\left\vert k\right\vert ^{r}\hat{u}\left(  k\right)  & \text{if }%
k\neq0\text{,}\\
\hat{u}\left(  0\right)  & \text{if }k=0\text{.}%
\end{array}
\right.  
\end{equation}
 Since
\begin{equation*}
\begin{split}
\int_{0}^{T} \int_{\mathbb{T}} \Psi u(x,t) v(x,t) dx dt 
= & \int_{0}^{T} \int_{\mathbb{T}} u(x,t) \phi(t)D^{-2} (\phi(x)v(x,t)) dx dt.
\end{split}
\end{equation*}
we have  $\Psi^{*} = \psi(t)D^{-2}\phi(x).$
For any $\epsilon >0$, let $\Psi_{\epsilon} = \Phi e^{\epsilon \partial_{x}^{2}} = \psi(t)\Phi_{\epsilon}$ be regularization of $\Psi$. We define
\begin{equation*}
\begin{split}
\alpha_{ \epsilon}^{1}(u_{n}, v_{n}):=  &\left<[\Psi_{\epsilon} , \mathcal{L}_{1} ]u_{n}, u_{n} \right>_{L^{2}(\mathbb{T} \times(0,T))}, \quad \quad 
\alpha_{\epsilon}^{2}(v_{n}, v_{n}): =  \left<[\Psi_{\epsilon} , \mathcal{L}_{2} ]v_{n}, v_{n} \right>_{L^{2}(\mathbb{T} \times(0,T))},\\
&\alpha_{ \epsilon}^{3} (u_{n}, v_{n}) :=   \left<[\Psi_{\epsilon} , \mathcal{L}_{3} ]u_{n}, v_{n} \right>_{L^{2}(\mathbb{T} \times(0,T))},\\
%\alpha_{n, \epsilon}^{4} := & \left< [\Phi_{\epsilon}, \mathcal{L}_{3}]v_{n} , u_{n} \right>_{L^{2}(\mathbb{T} \times(0,T))},
\end{split}
\end{equation*}
where
\begin{equation*}
\begin{split}
\mathcal{L}_{1}=  \partial_{t} + \partial_{x}^{3} + \mu \partial_{x},\ \
\mathcal{L}_{2} =  \partial_{t} + \alpha \partial_{x}^{3} + \zeta \partial_{x} \ \ \ \text{and} \ \ \
\mathcal{L}_{3} =  \eta \partial_{x}.
\end{split}
\end{equation*}
Note that 
\begin{equation*}
\begin{split}
[\Psi_{\epsilon}, \partial_{t}] w_{n}(x,t) =  - \psi'(t) \phi(x) D^{-2}w_{n}(x,t).
\end{split}
\end{equation*}
Denoting  $\alpha_{n, \epsilon} := \alpha_{\epsilon}^{1}(u_{n}, u_{n}) + \alpha_{\epsilon}^{3}(v_{n}, u_{n}) +\alpha_{\epsilon}^{2}(v_{n}, v_{n}) + \alpha_{\epsilon}^{3}(u_{n}, v_{n}),$
we have
\begin{equation*}
\begin{split}
\alpha_{n, \epsilon}=&  \left< [\Psi_{\epsilon}, \partial_{x}^{3} + \mu \partial_{x}]u_{n} , u_{n} \right> - \left< \psi'(t) \Phi_{\epsilon} u_{n}, u_{n} \right> + \left< [\Psi_{\epsilon}, \alpha \partial_{x}^{3} + \zeta \partial_{x}]v_{n} , v_{n} \right> \\
& - \left< \psi'(t) \Phi_{\epsilon} v_{n}, v_{n} \right>
 + \left< [\Psi_{\epsilon},  \eta \partial_{x}]u_{n} , v_{n} \right> + \left< [\Psi_{\epsilon},  \eta \partial_{x}]v_{n} , u_{n} \right>.
\end{split}
\end{equation*}
On the other hand,
\begin{equation*}
\begin{split}
\alpha_{\epsilon}^{1}(u_{n}, u_{n}) + \alpha_{\epsilon}^{3}(v_{n}, u_{n}) = & \left< \Psi_{\epsilon} (\mathcal{L}_{1} u_{n}), u_{n})\right> + \left< \Psi_{\epsilon}(\mathcal{L}_{3}v_{n}), u_{n} \right> - \left< \mathcal{L}_{1}(\Psi_{\epsilon} u_{n}) , u_{n} \right>  - \left< \mathcal{L}_{3}(\Psi_{\epsilon}v_{n}) , u_{n} \right>\\
= & \left< f_{n} , \Psi_{\epsilon}^{*} u_{n} \right> + \left<\Psi_{\epsilon} u_{n}  , \mathcal{L}_{1}u_{n} \right> + \left< \Psi_{\epsilon}v_{n} ,  \mathcal{L}_{3}u_{n} \right>,
\end{split}
\end{equation*}
since $\mathcal{L}_{1}u_{n} + \mathcal{L}_{3} v_{n} = f_{n}$, $\mathcal{L}_{1}^{*} = - \mathcal{L}_{1}$ and $\mathcal{L}_{3}^{*} = - \mathcal{L}_{3}$. Similarly
\begin{equation*}
\alpha_{\epsilon}^{2}(v_{n}, v_{n}) + \alpha_{\epsilon}^{3}(u_{n}, v_{n}) =  \left< g_{n} , \Psi_{\epsilon}^{*} v_{n} \right> + \left< \Psi_{\epsilon} v_{n}  , \mathcal{L}_{2}v_{n} \right> + \left< \Psi_{\epsilon}u_{n} ,  \mathcal{L}_{3}v_{n} \right>.
\end{equation*}
Thus,
\begin{equation*}
\begin{split}
\alpha_{n, \epsilon} =   \left< f_{n} , \Psi_{\epsilon}^{*} u_{n} \right> +  \left< g_{n} , \Psi_{\epsilon}^{*} v_{n} \right> + \left< \Psi_{\epsilon}u_{n} , f_{n} \right> +  \left< \Psi_{\epsilon}v_{n} , g_{n} \right>.
\end{split}
\end{equation*}
Now, following the ideas almost as done in \cite[Proposition 3.5]{laurent2010} the result is so achieved. 
\end{proof}

The following result concerns the propagation of regularity. Precisely, the result guarantees that if we have gain of derivatives in the spatial space in a subset $ \omega$ of $\mathbb{T}$, then this is also valid in the whole space $\mathbb{T}$.

\begin{proposition}[Propagation of regularity]\label{propagationofregularity}
Let $T>0$, $0 \leq b < 1$, $r \in \mathbb{R}$ and $(p,q) \in \mathcal{X}_{r, -b}$. Let $(u,v) \in \mathcal{X}_{r,b}$ be a solution of
\begin{equation*}
\begin{cases}
\partial_{t} u + \partial_{x}^{3} u +  \mu \partial_{x}u + \eta \partial_{x} v  = p,&\ \text{on}\  \mathbb{T}\times(0,T),\\
\partial_{t} v + \alpha \partial_{x}^{3} v + \zeta \partial_{x}v  + \eta \partial_{x} u = q,&\ \text{on}\  \mathbb{T}\times(0,T).\\
\end{cases}
\end{equation*}
If there exists a nonempty open set $\omega$ of $\mathbb{T}$ such that $$(u,v) \in L^{2}_{loc}(0,T; H^{r + \rho}(\omega)) \times   L^{2}_{loc}(0,T; H^{r + \rho}(\omega)),$$ for some $\rho$ satisfying
\begin{equation*}
0 < \rho \leq \min \left\{1-b , \frac{1}{2} \right\},
\end{equation*}
then  $(u,v) \in L^{2}_{loc}(0,T; H^{r + \rho}(\mathbb{T})) \times   L^{2}_{loc}(0,T; H^{r + \rho}(\mathbb{T}))$.
\end{proposition}

\begin{proof}
Set $s = r + \rho$ and for $n \in \mathbb{N}$ consider
\begin{equation*}
\begin{cases}
u_{n} = e^{\frac{1}{n} \partial_{x}^{2}}u =: \Xi_{n}u, \ \ p_{n} = \Xi_{n}p = \mathcal{L}_{1}u_{n} + \mathcal{L}_{3}v_{n},\\
v_{n} = e^{\frac{1}{n} \partial_{x}^{2}}v =: \Xi_{n}v, \ \ q_{n} = \Xi_{n}q = \mathcal{L}_{2}v_{n} + \mathcal{L}_{3}u_{n}.
\end{cases}
\end{equation*}
There exists a constant $C > 0$ such that
\begin{equation*}
\|(u_{n}, v_{n})\|_{\mathcal{X}_{r,b}} \leq C \quad \text{and} \quad \|(p_{n}, q_{n})\|_{\mathcal{X}_{r, -b}} \leq C, \ \ \forall \ n \in \mathbb{N}.
\end{equation*}
Pick $\phi \in C^{\infty}(\mathbb{T})$ and $\psi \in C_{0}^{\infty} ((0,T))$ as in the proof of Proposition \ref{propagationcompactness}, and set  $\Phi = D^{2s-2}\phi(x)$ and $\Psi= \psi(t) B$,  where $D$ is defined by \eqref{k10}. We have 
\begin{equation*}
\begin{split}
\left< \mathcal{L}_{1}u_{n}  \right. & \left. + \mathcal{L}_{3}v_{n}, \Psi^{*} u_{n}, \right>  + \left< \Psi u_{n} ,   \mathcal{L}_{1}u_{n} + \mathcal{L}_{3}v_{n}\right> \\
%= & \left< \Psi ( \partial_{t} u_{n}), u_{n} \right>  + \left< \Psi (\partial_{x}^{3}  +\mu \partial_{x}) u_{n} ,u_{n} \right> + \left< \eta \Psi \partial_{x} v_{n}, u_{n}  \right>\\
%&  - \left<\partial_{t} ( \Psi u_{n})  , u_{n}\right> - \left<( \partial_{x}^{3} + \mu \partial_{x})\Psi u_{n} , u_{n} \right> - \left< \eta \partial_{x} \Psi v_{n} , u_{n} \right>\\
 = & \left< [\Psi , \partial_{x}^{3} + \mu \partial_{x}]u_{n} , u_{n} \right> +  \left< [\Psi , \eta \partial_{x}]v_{n} , u_{n} \right>  - \left< \psi'(t) \Phi u_{n} , u_{n} \right>,
\end{split}
\end{equation*}
and, similarly
\begin{equation*}
\begin{split}
\left< \mathcal{L}_{2}v_{n} \right. & \left. + \mathcal{L}_{3}u_{n}, \Psi^{*} v_{n}, \right>  + \left< \Psi v_{n} ,   \mathcal{L}_{2}v_{n} + \mathcal{L}_{3}u_{n}\right> \\
 = &  \left< [\Psi , \alpha \partial_{x}^{3} + \zeta \partial_{x}]v_{n} , v_{n} \right> +  \left< [\Psi , \eta \partial_{x}]u_{n} , v_{n} \right> - \left< \psi'(t) \Phi v_{n} , v_{n} \right>,
\end{split}
\end{equation*}
where $\mathcal{L}_i$, for $i=1,2,3$, were defined on Proposition \ref{propagationcompactness}. With this in hand, the result follows in a similar way as proved  in \cite[Proposition 3.6]{laurent2010}.
\end{proof}

We are in position to prove the unique continuation property for our dispersive operator.

\begin{corollary}\label{uniquecontinuation}
Let $\omega$ be a nonempty set in $\mathbb{T}$ and $$(u,v) \in \mathcal{X}_{0}^1 \cap C([0,T]; L_{0}^{2}(\mathbb{T})) \times \mathcal{X}_{0}^{\alpha} \cap C([0,T]; L_{0}^{2}(\mathbb{T}))$$ be solution of  \eqref{PVIsystem}. Suppose that $(u,v)$ satisfies
 \begin{equation*}
\begin{cases}
\partial_{t}u + \partial_{x}^{3}u + \mu  \partial_{x} u + \eta  \partial_{x}v + \partial_{x}P(u,v) =0,&\ \text{on}\  \mathbb{T}\times(0,T),\\
\partial_{t}v + \partial_{x}^{3}v + \zeta \partial_{x}v + \eta \partial_{x} u + \partial_{x}Q(u,v) =0,&\ \text{on}\  \mathbb{T}\times(0,T),\\
(u(x,t), v(x,t)) = (c_{1}(t), c_{2}(t)),& \ \mbox{for a.e.} \ (x,t) \in \omega \times (0,T),
\end{cases}
 \end{equation*}
 where $c_{1}$, $c_{2} \in L^{2}(0,T) \times L^{2}(0,T)$. Then $(u(x,t), v(x,t)) \equiv (0,0)$ for a.e.  $(x,t) \in \mathbb{T} \times (0,T) $.
\end{corollary}

\begin{proof}
Since  $(u(x,t), v(x,t)) = (c_{1}(t), c_{2}(t))$ for a.e. $(x,t) \in \omega \times (0,T)$, we have that
\begin{equation}\label{identity2}
\partial_{t}u = c_{1}'(t) = 0\quad \text{and}\quad
\partial_{t}v = c_{2}'(t) = 0.
\end{equation}
Pick a time $t \in (0,T)$ as above and define $(p, q) := (\partial_{x}^{3}u ( \cdot, t) , \partial_{x}^{3}v (\cdot, t)) $. Thus, it holds that  $(p, q ) \in H^{-3}(\mathbb{T}) \times H^{-3}(\mathbb{T})$ with $(p, q) = (0,0)$.  Decompose $p$ and $q$ as 
\begin{equation*}
    p(x) = \sum_{k \in \mathbb{Z}} \widehat{p_{k}}e^{ikx} \ \mbox{and} \  q(x)= \sum_{k \in \mathbb{Z}} \widehat{q_{k}}e^{ikx}.
\end{equation*}
the convergence of the Fourier series being in $H^{-3}(\mathbb{T})$. Since  $p$ and $q$ are  real-valued functions, we also have that $\widehat{p_{-k}} = \widehat{p_{k}}$  for all $k$ and the same is true for $q$. Then,  
\begin{equation*}
0 = p(x) = \sum_{ k > 0}  \widehat{p_{k}}e^{ikx} \quad \text{and} \quad 0 = q(x) = \sum_{k > 0}\widehat{q_{k}}e^{ikx},
\end{equation*}
 for each $x \in \omega$. Applying \cite[Lemma 2.9]{linaresrosier2015} to $p$ and $q$ we obtain $(p, q)  \equiv (0,0)$ on $\mathbb{T}$. It follows, $\partial_{x}^{3} u = \partial_{x}^{3}v = \partial_{x} u = \partial_{x}v =0$ on $\mathbb{T}$ for a.e $t \in(0,T)$. Hence, 
 \begin{equation*}
(u(x,t), v(x,t)) = (c_{1}(t), c_{2}(t)) \ \mbox{for a.e.} \ (x,t) \in \mathbb{T} \times (0,T).
 \end{equation*}
As in \eqref{identity2} we deduce that for a.e.  $(x,t) \in \mathbb{T} \times (0,T)$ we have $(c'_{1}(t), c'_{2}(t))=(0,0)$ for $c_{1}$, $c_{2} \in \mathbb{R}$, which,
combined with the fact that $[u_{0}] = [v_{0}]=0$, gives that  $c_{1} =  c_{2}=0$. The proof of Corollary \ref{uniquecontinuation} is complete.
\end{proof}

\subsection*{Acknowledgments} %The authors wish to thank the referee for his/her valuable comments which improved this paper. 
Capistrano–Filho was supported by CNPq grant 307808/2021-1,  CAPES grants 88881.311964/2018-01 and 88881.520205/2020-01,  MATHAMSUD grant 21-MATH-03 and Propesqi (UFPE).  Gomes was supported by PNPD-INCTMat  Capes grant 88887505895/2020-00. This work was done during the postdoctoral visit of the second author at the Universidade Federal de Pernambuco, who thanks the host institution for the warm hospitality.

\end{document}